\documentclass[11pt]{amsart}

\usepackage[a4paper,top=4.6cm,bottom=4.7cm,inner=3.3cm,outer=3.3cm]{geometry}
\usepackage[pdftex]{graphicx,graphics} 
\usepackage{amsmath,amssymb,amsfonts,color,hyperref} 
 \usepackage{epsfig}

\usepackage{graphicx,tikz} 
\usetikzlibrary{shapes.symbols}
 \usetikzlibrary{shapes.geometric}
 \usetikzlibrary{shapes.callouts}
 \usetikzlibrary{shapes.misc}
 \usetikzlibrary{shapes.multipart}
  \usetikzlibrary{shapes.arrows}

\usepackage{graphicx,psfrag,mathrsfs}
\usepackage{tikz}   
\usepackage{pgfplots}
\usetikzlibrary{external}

\tikzset{external/force remake}

\hypersetup{
    linktoc=page,
   linkcolor=red,          
  citecolor=blue,        
    filecolor=blue,      
   urlcolor=cyan,
    colorlinks=true           
}

\numberwithin{equation}{section} 
\newcommand{\wwphi}{\widetilde {\widetilde{\phi}}}

\newcommand\ld{L}
\newcommand\supp{\mathrm{supp}}
\newtheorem{theorem}{Theorem}[section]
\newtheorem{proposition}[theorem]{Proposition}

\newtheorem{remark}[theorem]{Remark}
\newtheorem{lemma}[theorem]{Lemma}
\newtheorem{definition}[theorem]{Definition}

\newtheorem{corollary}[theorem]{Corollary}

\newtheorem{conjecture}[theorem]{Conjecture}

\newcommand\ep{{\varepsilon}}

\newcommand\zu{[0,1]}

\newcommand\la{\lambda}
\newcommand\mk{\medskip}
\newcommand\sk{\smallskip}

\newcommand{\R}{\mathbb{R}}
\newcommand{\N}{\mathbb{N}}
\newcommand{\Z}{\mathbb{Z}}

\newcommand\zud{[0,1]^d}

\newcommand \ds{\displaystyle}
\newcommand \ml{multifractal }

\newcommand\si{\sigma}
 \newcommand{\SD}{\mathscr{S}_{d,\mathcal{M}}}
 \newcommand{\MDs}{\mathscr{E}_d}
 \newcommand{\SDs}{\mathscr{S}_{d}}
\newcommand{\sd}{\sigma}

\newcommand{\TD}{\mathscr{T}_{d,\mathcal{M}}}
\newcommand{\TDs}{\mathscr{T}_d}

 \newcommand\La{\Lambda}
 \newcommand{\MD}{\mathcal{M}_d}

\newcommand\wb{\widetilde{B}}

\newcommand\omegat{\omega^\mu}
\newcommand{\alm}{\alpha_{\min}}

\newcommand\MF{MF } 
\newcommand\SMF{SMF } 

\newcommand\lai{ \overline{\lambda}}
\newcommand\ki{ \overline{k}}
\newcommand\ji{ \overline{j}}
\newcommand\lau{ \lambda^{\uparrow}}

 \newcommand\lah{\widehat{\la}}
  
\newcommand\dom{\mbox{dom}}

\begin{document}
\title[Besov spaces in multifractal environment]{Besov spaces in multifractal environment and the Frisch-Parisi conjecture}

\author{Julien Barral \and St\'ephane Seuret }         

\address{Julien Barral, LAGA, CNRS UMR 7539, Institut Galil\'ee, Universit\'e Sorbonne Paris Nord, Sorbonne Paris Cit\'e, 99 avenue Jean-Baptiste Cl\'ement , 93430  Villetaneuse, 
France}
\address{St\'ephane Seuret, Universit\'e Paris-Est, LAMA (UMR 8050), UPEMLV, UPEC, CNRS, F-94010, Cr\'eteil, France}
\date{Received: date / Revised version: date}

\begin{abstract}
In this article,   a solution to the so-called Frisch-Parisi conjecture is brought.  This achievement is based on three ingredients  developed in this paper. First  almost-doubling fully supported Radon measures on $\R^d$ with a prescribed singularity spectrum are constructed. Second we define new \textit{heterogeneous} Besov  spaces $B^{\mu,p}_{q}$ and find a characterization using wavelet coefficients. Finally, we fully describe the multifractal nature of typical functions in the function spaces $B^{\mu,p}_{q}$. Combining these  three results, we find Baire function spaces   in which typical functions have  a prescribed  singularity spectrum and satisfy a multifractal formalism. This yields an answer to the Frisch-Parisi conjecture.

\end{abstract}

\maketitle

\tableofcontents
\section{Introduction}

This paper deals with multifractal analysis of functions, which originates from the first geometric quantification of the H\"older singularities structure in fully developed turbulence \cite{Mandel2,Mandel1,FrischParisi}. This subject is an instance of the natural concept of multifractality, which comes into play as soon as, given a mapping $h:X\to A$ between a metric space $(X,d)$ and a set $A$, one describes geometrically the level sets of $h$ by considering the mapping $\sigma:\alpha\in A\mapsto \dim h^{-1}(\{\alpha\})$, where $\dim$ stands for the Hausdorff dimension.  Indeed, in many   situations, the  level sets of~$h$ form an uncountable family of  disjoint fractal sets, and $\sigma$ is sometimes called {\it multifractal spectrum}.  This spectrum provides a hierarchy between these level sets, according to their size measured by their Hausdorff dimension. Such spectra occur  in many mathematical fields, such as harmonic and functional analysis (in the description of fine properties of  Fourier series \cite{JaffRiemann,BayHeu} or typical elements in functional spaces \cite{BUC_Nagy,JAFF_FRISCH}),  probability theory (to describe fine properties of Brownian motion or SLE curves~\cite{OreyTaylor,Perkins,Lawleracta,GMS}, multiplicative chaos and Gaussian free field, random covering problems~\cite{BM04,HuMillerPeres,RV,BFan}),  ergodic theory, dynamical and iterated function systems (to analyse  Gibbs/harmonic measures on conformal repellers, Birkhoff averages,  and self-similar measures \cite{pesin2,Makarov,Feng2007,FL09,Shmerkin_annals}), metric number theory (Diophantine approximation and ubiquity theory~\cite{Jarnik,HillVelani2,BS}, shrinking targets problems and dynamical covering problems~\cite{HillVelani,FanSchmTrou}), the previous references being far from exhaustive. 

\sk 

%
 In the multifractal analysis  of a real valued function $f \in L^\infty_{\rm loc}(\R^d)$, the function~$h$~of interest~is the pointwise H\"older exponent function $h_f$, which is defined as follows. Given~$x_0 \in \R^d$, and $H\in \R_+$, $f$ is said to
belong to ${\mathscr C}^{H}(x_0)$ if
there exist  a polynomial~$P$ of degree at most $\lfloor H\rfloor$, a constant $C>0$, and a neighborhood $V$ of $x_0$ such that 
$$
\forall\, x\in V, \quad   |f(x)-P(x-x_0)|\leq C|x-x_0|^{H}.
$$


\sk

The associated multifractal spectrum,  also called {\it  singularity spectrum} of $f$, is  the mapping
  $$\sigma_f: H \in \R\cup\{\infty\} \mapsto \dim
\, E_f(H)\in [0,d]\cup\{-\infty\}, \  \mbox{where } E_f(H):= h_f^{-1}(\{H\})$$
(note that $E_f(H)=\emptyset$ for $H<0$).  Again, $\dim$ stands for the Hausdorff dimension, with the convention   $\dim \emptyset =-\infty$. The function $f$ is said to be  {\em multifractal} when $E_f(H)\neq\emptyset$ for at least two values of $H$. 

\sk
 
The idea of considering this spectrum goes back to the physicists U. Frisch and G. Parisi~\cite{FrischParisi}, who aimed at  quantifying geometrically  the local variations of  the velocity field of a turbulent fluid, and introduced the term {\it multifractal}. Another fundamental idea pointed out by Frisch and Parisi consisted in coupling   the singularity spectrum  with a large deviations approach, in order to statistically describe the H\"older singularities distribution (similar to  Mandelbrot's approach   for measures \cite{Mandel2}). This led to the notion of {\em multifractal formalisms}  for functions. Since defining rigorously such a formalism is a bit involved and will be done later in Section~\ref{statement}, let us say at the moment that schematically,  in such a formalism,  the  singularity spectrum $\sigma_f$ of a H\"older continuous function $f$  is always dominated by (and in good cases, coincides with) the Legendre-Fenchel transform 
$$\zeta_f^ *(H) := \inf_{q\in\R} Hq-\zeta_f(q)$$
 of a function $\zeta_f:\R\to\R$, called the {\it scaling function} or the {\em $L^q$-spectrum} of $f$: for every $h\geq 0$, $\sigma_f(h) \le \zeta_f^ *(h)$. The mapping $\zeta_f$  is a kind of free energy function  encapsulating the asymptotic statistical distribution of the   H\"older singularities as the   observation scale tends to 0, and it can be numerically estimated \cite{JAFFARD:2006:A}. For instance,  in their seminal article, Frisch and Parisi used for $\zeta_f$  the scaling exponent of the moments of the increments of $f$,  informally defined~as 
$$
|h|^{-d}\int_\Omega|f(x+h)-f(x)|^q\, {d}x\sim |h|^{\zeta_f(q)} \quad \text{as }h\to 0,
$$ 
where $\Omega$ is a fixed bounded domain on which $f$ is supposed to be fully supported. The heuristics developed in \cite{FrischParisi} lead to seek for the largest as possible classes of functions for which the  equality 
\begin{equation}\label{FPvalid}
\sigma_f(H)=\zeta_f^*(H)
\end{equation}
 holds at any $H$ such that $\zeta_f^*(H)\ge 0$. In such a situation, one says that {\em the multifractal formalism holds} for $f$, or that $f$ satisfies the multifractal formalism. Then,  the spectrum $\sigma_f$ is a continuous concave mapping with support included in $(0,+\infty)$, and assuming that the topological support of $f$ is full, one necessarily has $\sigma_f(H)=d=-\zeta_f(0)$ for some $H\geq 0$ (for instance the level set $E_f(H)$ may have a positive Lebesgue measure). 
 
\sk We will come back to rigorous definitions of multifractal formalisms for functions and measures in  Sections~\ref{secMFF} and \ref{wqb}. The concept of multifractal formalism motivated many works in geometric measure theory \cite{BRMICHPEY,Olsen,LN,JLVVOJAK}, dynamical systems  in connection with the thermodynamic formalism~\cite{Pesin}, and analysis~\cite{Jaf97a,JAFF_FRISCH,JaffardPSPUM}. It provides a powerful framework to describe the fine geometric structure of invariant measures of various dynamical systems~\cite{ColletLebPor,Rand,Pesin} and     self-similar and self-affine measures~\cite{Kin95,Olsen,Olsen98,LN,FL09,BarralFeng2},  self-similar functions~\cite{Jaf97a},  as well as  limit measures or functions in multiplicative chaos theory~\cite{HW,BM04,BaJin10}. The  singularity spectrum and  its suitable extensions to non-bounded functions have also been used to describe the regularity properties of celebrated functions like Riemann's and Brjuno's functions \cite{JaffRiemann,SU,JM},   stochastic processes like L\'evy processes and   general classes of Markov processes~\cite{jaffard1999levy,BFJS-markov,yang2018ihp}, as well as L\'evy processes in multifractal time~\cite{BaSeuLevy}. 

\begin{figure} 
\includegraphics[scale=0.093]{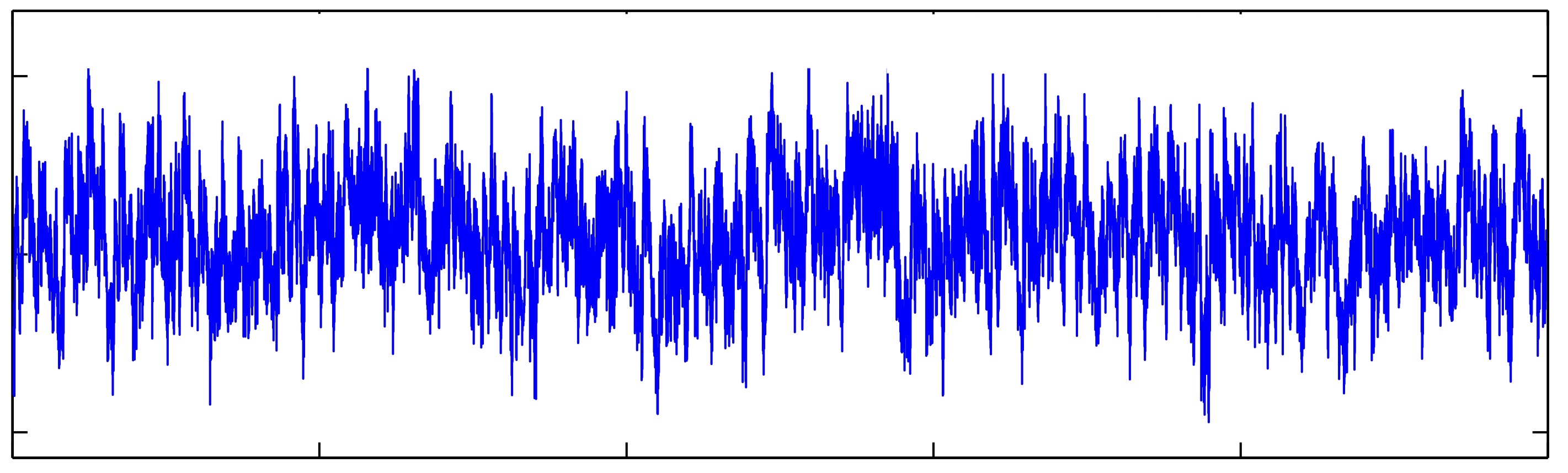} \includegraphics[scale=1.7]{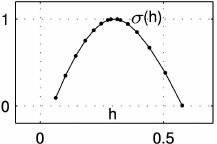}
 \caption{Estimated singularity spectrum (right) for the 1D velocity of a turbulent flow (left) - Credit to P. Abry, H. Wendt}
\label{figure00}
 \end{figure}

\sk

Multifractal formalisms are also relevant in many applications, due to the existence of  stable  algorithms that precisely estimate scaling functions $\zeta_f$  of numerical data. 
Then,  a key observation is that for most of real-life data associated with  intermittent phenomena, their estimated singularity spectra $\zeta_f^*$ have a characteristic  strictly concave bell shape (see \cite{MandMemor} and Figure~\ref{figure00}). This is also the case for the singularity spectra of important classes of functions possessing scaling properties \cite{Jaf97a,BaSeuLevy,BaJin10}. This behavior is in striking contrast to the results established for typical functions in the classical functional spaces, where ``typical'' is meant in the sense of Baire categories
\footnote{Recall that in a Baire topological space $E$, a property $\mathcal{P}$ is called typical, or generic,  when the set $\{f \in E:  \text{ $f$ satisfies $\mathcal{P}$}\}$ is of second category in~$E$, or equivalently contains a dense $G_\delta$-set, that is the intersection of a countable family of  dense and open sets. One says that typical elements in $E$ satisfy $\mathcal{P}$ when $\mathcal{P}$ is  typical in $E$.
}.  Indeed, it has been proved that typical increasing real functions (Buczolich\&Nagy \cite{BUC_Nagy}), typical functions in   Sobolev and Besov spaces (Jaffard \cite{JAFF_FRISCH}, Jaffard\&Meyer \cite{JaffardMeyer}), and typical measures (Buczolich\&Seuret, Bayart \cite{BuS2,Bay}) satisfy a multifractal formalism and have  an affine increasing singularity spectrum. One   concludes that, from the multifractal standpoint,   realistic behaviors are not reproduced by typical  elements in the standard function spaces. A precise statement regarding the typical  singularity spectrum in Besov spaces  is recalled  in Remark \ref{rem-spectra2}, and  the validity of   multifractal formalisms in these spaces is discussed in Section~\ref{secMFF} (see also Figure~\ref{figure2}). 

\sk

On the other hand, the previous genericity results show that many multifractal functions do satisfy some multifractal formalism without assuming  any scale invariance properties. In~\cite{JAFF_FRISCH}, Jaffard seeks for Baire topological spaces of functions in which typical functions have a prescribed singularity spectrum, and do obey some multifractal formalism. He names this inverse problem   ``Frisch-Parisi conjecture'', and provides a partial solution to it: he finds that some  intersections of homogeneous Besov spaces are Baire topological spaces in which typical functions possess an increasing compactly supported singularity spectrum, with a prescribed concave part, and another part which is necessarily linear; moreover, typical elements partially obey some multifractal formalism (see Section~\ref{fpc} for a detailed description of Jaffard's result). Again, no scale invariance is assumed.

The  \textit{Frisch-Parisi conjecture}   considered by Jaffard is formulated as follows:

\begin{conjecture}[Frisch-Parisi conjecture]\label{FP}

Let $\SDs$ be the set of functions $\sigma: \R\to [0,d]\cup \{-\infty\}$ such that  $\sigma$ is concave, continuous,  with compact support included in $(0,+\infty)$ and  whose maximum equals~ $d$. For every $\sigma \in \SDs$, there exists a Baire functional space  of functions defined on $\R^d$ in which  any typical element $f$  obeys some multifractal formalism and   satisfies $\sigma_f=\sigma$.
\end{conjecture}

\sk

\begin{figure}
 \includegraphics[scale=0.04]{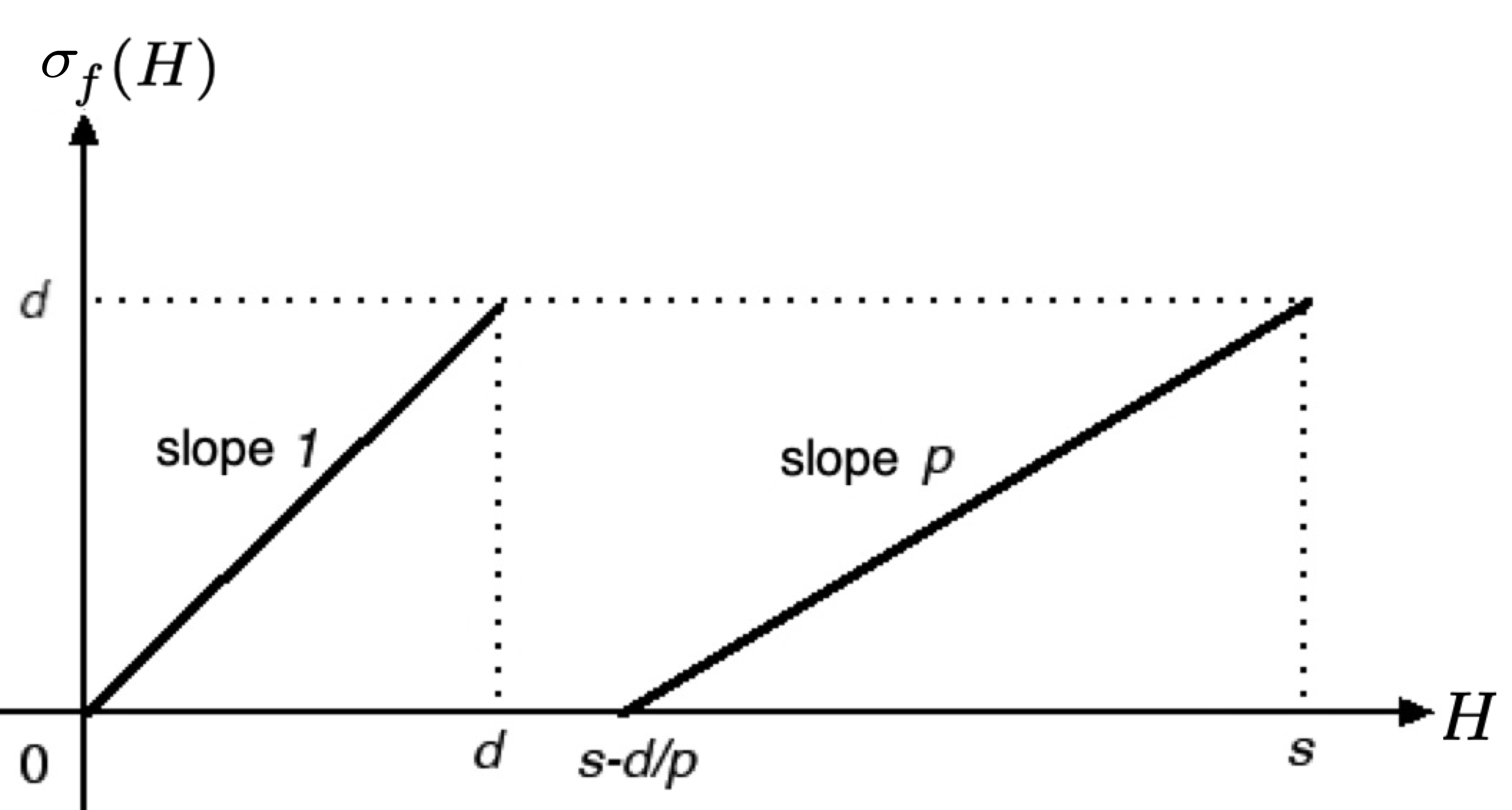}
 \caption{Typical multifractal spectrum of probability measures (left) or functions in $B^{s,p}_q(\R^d)$ when $s>d/p$ (right).}
\label{figure2}\end{figure}

Note that the set $\SDs$ consists of those mappings $\sigma$ which are admissible to be the  singularity spectrum of some H\"older continuous function $f:\R^d\to \R$ whose pointwise H\"older exponents range in a compact subinterval of $(0,+\infty)$, such that $\dim E_f(H)=d$ for at least one exponent $H$,  and which satisfies some multifractal formalism. The   formalism for functions adopted in this paper will be specified in Section~\ref{secMFF}. It is based on the  so-called wavelet leaders multifractal formalism, and developed by Jaffard in particular in \cite{JaffardPSPUM}.

\sk

In the present paper, we introduce natural Baire function spaces in which typical functions have a prescribed  bell-shape  singularity spectrum, and satisfy the multifractal formalism mentioned above. This construction follows from three ingredients   developed in this paper, each of them having its own interest. 
 
\sk
 First we prove the existence of almost-doubling and $\Z^d$-invariant Radon measures fully supported on~$\R^d$ with prescribed  singularity spectrum, and which satisfy the standard multifractal formalisms for measures developed in \cite{BRMICHPEY,Olsen} (Theorem \ref{th-construct-measures} and Corollary \ref{th-construct-measures}). Up to now, such a result was  only known  for measures supported on a totally disconnected set~\cite{Barralinverse} (see also~\cite{BuS3} for results on the prescription of the  singularity spectrum for measures). These measures possess scaling-like properties.
  
Second, we introduce new functional spaces $B^{\mu,p}_q(\R^d)$ that we call Besov spaces in {\em multifractal environment}, whose definition is based on  a modification of the usual notion of $L^p$-moduli of smoothness. These spaces depend  on an  almost-doubling capacity~$\mu$, that we call {\em environment}. Then, we study the wavelet decomposition of functions belonging to $B^{\mu,p}_q(\R^d)$, and prove that  the intersection of  suitable perturbations of the space $B^{\mu,p}_q(\R^d)$ defines a Fr\'echet space $\widetilde B^{\mu,p}_q(\R^d)$ naturally  characterized in terms of  wavelet coefficients   (see Definition \ref{devbseovmu} and Theorem \ref{th_equivnorm_2}). 
 
Finally, thanks to the previous wavelet characterization, we perform the multifractal analysis of typical functions in  $\widetilde B^{\mu,p}_q(\R^d)$,  when the environment $\mu$ is any positive power of  the almost doubling measures   built before (Theorems \ref{main} and \ref{validity}). 

Using the spaces $\widetilde B^{\mu,p}_q(\R^d)$ with suitable parameters $\mu$, $p$ and $q$,   a by-product of   the previous results gives   the   answer to the Frisch-Parisi conjecture:
\begin{theorem}\label{solution}
Conjecture \ref{FP} is true. 
\end{theorem}

 It is worth noting that although  typical functions in  $\widetilde B^{\mu,p}_q(\R^d)$ are multifractal and satisfy a multifractal formalism, they do not possess any sort of self-similar structure, consolidating the idea that the multifractal behavior  is generic and not reserved to exceptional situations.
\sk

Our three main results are precisely stated  in the next section.

\section{Statements of the main results}
 \label{statement}

\subsection{Some notations and definitions.} \label{statement-1}  

%

The Lebesgue measure on $\R^d$ is denoted   $\mathcal{L}^d$. 

If $E$ is a Borel subset of $\R^d$, the Borel $\sigma$-algebra of $E$ is denoted $\mathcal B(E)$.  $|E|$ stands for the Euclidean diameter of $E$. 

Given $x\in \R^d$ and $r\in \R_+$, the closed Euclidean ball centered at $x$ with radius $r$ is denoted $B(x,r)$.

\sk

For $j\in\Z$,   $\mathcal D_j$  stands for the collection of closed dyadic cubes of generation $j$, i.e. the cubes $\lambda_{j,k}=2^{-j}k+2^{-j}[0,1]^d$, where $k\in\Z^d$. We also set  $\mathcal D=\bigcup_{j\in\Z} \mathcal D_j$, and if $\lambda=\lambda_{j,k}\in \mathcal D_j$ we denote $x_\la= 2^{-j} k$.

For $x\in\R^d$, $\lambda_j(x)$ stands for the closure of the unique dyadic cube of generation $j$, product of semi-open to the right  dyadic intervals, which contains $x$. 

\sk

For  $j\in\Z$,  $\lambda\in\mathcal D_j$, and $N\in \N^*$, $N \lambda$ denotes  the cube with same center as $\la$ and radius equal to $N\cdot 2^{-j-1}$ in  $(\R^d,\|\,\|_\infty)$. For instance, $3\la$ is  the union of those $\lambda'\in \mathcal D_j$ such that  $\partial \la \cap \partial\lambda'\neq\emptyset$ ($\partial \la $ stands for the frontier of the cube $\la$).
\medskip

\sk

The domain of a function $g:\R\to \R\cup\{-\infty\}$ is defined as $g^{-1}(\R)$, and denoted by $\dom(g)$. 
When $g$ is concave and finite, one sets  $g'(+\infty)=\lim_{t\to+ \infty} g'(t^+)$ and $g'(-\infty)=\lim_{t\to-\infty} g'(t^+)$. 

\sk

The family of H\"older-Zygmund spaces  is denoted $\{\mathscr {C}^s(\R^d)\}_{s>0}$ (see \cite{Meyer_operateur,Triebel} for instance for thorough expositions of classical functional spaces).  

\sk

\begin{definition}
The set of H\"older set  functions on $\mathcal B(\R^d)$ is defined as
\begin{equation}
\label{defSrd}
\mathcal{H}(\R^d) = \big \{\mu:\mathcal{B}(\R^d) \to \R_+\cup\{\infty\}:\, \exists \, C,s>0, \ \forall \, E\subset \R^d, \  \mu(E) \leq C |E|^s  \big\}.
\end{equation}
Then,  the set of H\"older capacities is defined as
\begin{equation}
\label{defCrd}
\mathcal{C}(\R^d) =  \big\{\mu\in\mathcal{H}(\R^d):\, \forall\, E, F\in\mathcal B(\R^d), \ E\subset F \Rightarrow \mu(E)\leq \mu(F) \big\}.
\end{equation}
and the set of   H\"older Radon measures is defined as
\begin{equation}
\label{defMrd}
\mathcal{M}(\R^d) =  \big\{\mu \in \mathcal{C}(\R^d) : \mu \mbox{ is a  Radon measure}\big\}.
\end{equation}
The  topological support $ \supp(\mu)$ of $\mu \in\mathcal{H}(\R^d)$ is the set of points $x\in\R^d$ for which $\mu(B(x,r))>0$ for every $r>0$. A capacity  $\mu$ is {\em fully supported} when $ \supp(\mu) =  \R^d$.

\medskip

Similarly, one defines the sets $\mathcal{H}(\zu^d) $, $\mathcal{C}(\zu^d)$ and $\mathcal{M}(\zu^d)$  by replacing $\R^d$ by $\zu^d$ in the above definitions.
\end{definition}

\begin{definition}\label{defmutothes}
For $s>0$, a set function   $\mu\in\mathcal{H}(\R^d) $ is $s$-H\"older  when there exists $C>0$ such that $\mu(E)\le C|E|^{s}$ for all $E\in\mathcal{B}(\R^d)$.

Then, for $\mu\in\mathcal{H}(\R^d) $,  $s>0$, and $E\in\R^d$, define  
\begin{align*}
\mu^s (E)= \mu(E)^s\quad \text{and}\quad 
\mu^{(+s)} (E)= \mu(E) |E|^s,
\end{align*}
and if $\mu$ is $s_0$-H\"older,  then for all $s\in (0,s_0)$, define 
$$
\mu^{(-s)} (E)=\begin{cases} 0&\text{if } |E|=0,\\
\mu(E) |E|^{-s}&\text{if }0<|E|<+\infty,\\
\infty&\text{otherwise.}
\end{cases}
$$
\end{definition}
\noindent Starting from  $\mu\in\mathcal{H}(\R^d)$, $\mu^s$,  $\mu^{(+s)}$ and  $\mu^{(-s)}$ as defined above still  belong to $\mathcal{H}(\R^d)$.

\subsection{Almost-doubling measures with prescribed multifractal behavior} \label{constr.meas.}
   
Multifractal formalisms for measures find their origin in works by physicists who proposed to  characterize ``strange sets" by considering, for any invariant  probability measure $\mu$ on  such a set $S$, the partition of $S$ into   iso-H\"older sets of $\mu$. They further estimated  the ``fractal'' dimensions of these sets using the Legendre transform of some free energy function, the $L^q$-spectrum, closely related to the Renyi generalized dimensions   \cite{Hentschel,Halsey}. Their ideas were later rigorously formalized by mathematicians~(see, e.g. \cite{BRMICHPEY,LN,Olsen}).  

\mk

The local behavior of elements of $\mathcal{H}([0,1]^d)$ is described via their pointwise H\"older exponents, also called local dimensions in the case of measures.

\begin{definition}\label{formalism1} Let $\mu \in\mathcal{H}(\zu^d)$. For $x\in \supp(\mu)$,    the lower and upper pointwise H\"older exponents  of $\mu$ at $x$  are respectively defined by 
$$
\underline h_\mu(x)=\liminf_{j\to+\infty}\frac{\log_2 \mu(\lambda_j(x))}{-j} \  \text{ and } \  \overline h_\mu(x)=\limsup_{j\to\infty}\frac{\log _2\mu(\lambda_j(x) )}{{-j}}.
$$
 Whenever $\underline h_\mu(x) = \overline h_\mu(x)$,  the common limit is called  $ h_\mu(x)$.
Then, for $\alpha\in\R$,  
\begin{align*} 
&\underline E_\mu( \alpha)=\left \{x\in \supp(\mu): \underline h_\mu(x)= \alpha\right \} \ \ \ \ \ 
 \overline E_\mu( \alpha)=\left \{x\in\supp(\mu): \overline h_\mu(x)= \alpha\right \},\\
 &\ \ \ \ \ \ \ \   \mbox{ and } \ \ \ \ \ \ \ \  E_\mu( \alpha)= \underline E_\mu(\alpha)\cap   \overline E_\mu( \alpha).
 \end{align*}
 The singularity (or multifractal) spectrum of $\mu$ is then the mapping 
$$
\sigma_\mu : \  \alpha\in\R\longmapsto \dim \underline E_\mu( \alpha).
$$
\end{definition}

\begin{definition}  
The $L^q$-spectrum of $\mu \in\mathcal{H}(\zu^d)$ with ${\rm supp}(\mu)\neq\emptyset$ is defined by
\begin{equation*}\label{taumubis}
\tau_{\mu}:\, q\in\R\mapsto \liminf_ {j\to+\infty} -\frac{1}{j }\log_2 \sum_{\substack{\lambda\in\mathcal D_j,\, \lambda\subset [0,1]^d,\\ \mu(\lambda)>0}} \mu(\lambda)^q.
\end{equation*}
\end{definition}

\begin{figure}  \begin{center}\begin{tikzpicture}[xscale=0.9,yscale=0.9]
{\small
\draw [->] (0,-2.8) -- (0,1.5) [radius=0.006] node [above] {$\tau_\mu(t)$};
\draw [->] (-1.,0) -- (3.7,0) node [right] {$t$};
 \draw [thick, domain=1.7:3.5, color=black]  plot ({\x},  {-ln(exp(\x*ln(1/5)) +exp(\x*ln(0.8)))/(ln(2)});
 \draw [thick, domain=0.6:1.7, color=black]  plot ({\x},  {-ln(exp(\x*ln(1/5)) +exp(\x*ln(0.8)))/(ln(2)});
\draw [thick, domain=-0.8:0.6, color=black]  plot ({\x},  {-ln(exp(\x*ln(1/5)) +exp(\x*ln(0.8)))/(ln(2)});
  \draw[dashed] (0,0) -- (2.8,1);
\draw [fill] (-0.1,-0.20)   node [left] {$0$}; 
\draw [fill] (-0,-1) circle [radius=0.03]  node [left] {$-d$ \ }; 

\draw [fill] (1,-0) circle [radius=0.03] [fill] (1,-0.3) node [ ] {$ $}; }
\end{tikzpicture} \hskip .5cm
 \begin{tikzpicture}[xscale=1.9,yscale=2.6]
    {\tiny
\draw [->] (0,-0.2) -- (0,1.25) [radius=0.006] node [above] {$  \sigma_\mu(H)=\tau_\mu^*(H)$};
\draw [->] (-0.2,0) -- (2.8,0) node [right] {$H$};
\draw [thick, domain=0:5]  plot ({-(exp(\x*ln(1/5))*ln(0.2)+exp(\x*ln(0.8))*ln(0.8))/(ln(2)*(exp(\x*ln(1/5))+exp(\x*ln(0.8)) ) )} , {-\x*( exp(\x*ln(1/5))*ln(0.2)+exp(\x*ln(0.8))*ln(0.8))/(ln(2)*(exp(\x*ln(1/5))+exp(\x*ln(0.8))))+ ln((exp(\x*ln(1/5))+exp(\x*ln(0.8))))/ln(2)});
\draw [thick, domain=0:5]  plot ({-( ln(0.2)+ ln(0.8))/(ln(2)) +(exp(\x*ln(1/5))*ln(0.2)+exp(\x*ln(0.8))*ln(0.8))/(ln(2)*(exp(\x*ln(1/5))+exp(\x*ln(0.8)) ) )} , {-\x*( exp(\x*ln(1/5))*ln(0.2)+exp(\x*ln(0.8))*ln(0.8))/(ln(2)*(exp(\x*ln(1/5))+exp(\x*ln(0.8))))+ ln((exp(\x*ln(1/5))+exp(\x*ln(0.8))))/ln(2)});
  \draw[dashed] (0,1) -- (2.6,1);
\draw [fill] (-0.05,0.10)   node [left] {$0$}; 
\draw [fill] (0,1) circle [radius=0.03] node [left] {$d \ $}; 
\draw  [fill] (0.32,0) circle [radius=0.03] node [below] {$\alpha_{\min} = \tau_\mu'(+\infty) $};
\draw  [fill] (2.32,0) circle [radius=0.03] node [below] {$\alpha_{\max} = \tau_\mu'(-\infty)  $};
\draw  [fill] (1.32,1) circle [radius=0.03]  [dashed]   (1.32,1) -- (1.32,0)  [fill] (1.32,0) circle [radius=0.03]  node [below] {$\tau_\mu'(0)$};
}
\end{tikzpicture}
\end{center}
\caption{{\bf Left:} Free energy function of $\mu\in \mathcal{C}(\zu^d)$ satisfying the MF. {\bf Right:} The  singularity spectrum of $\mu$.
 }
\label{figureLegpair}
\end{figure}
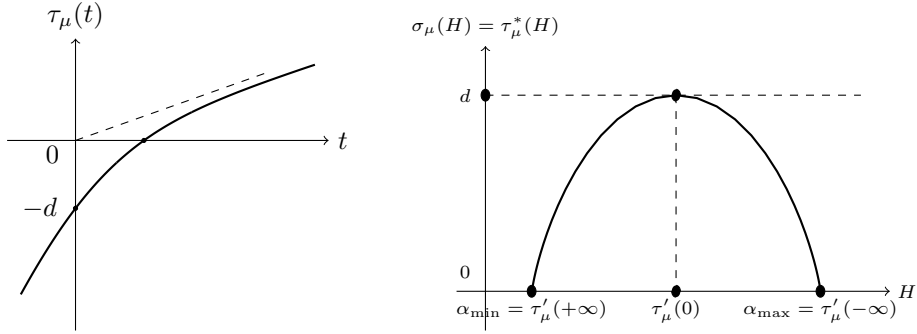
%

Then, one always has (see \cite{BRMICHPEY,JLVVOJAK})
$$
\sigma_\mu(\alpha)\le \tau_\mu^*(\alpha):= \inf_{q\in \R} q\alpha -\tau_\mu(q).
$$
In particular, if $\mu\in\mathcal M([0,1]^d)$, since $\tau_\mu(1)=0$ one has  $\si_\mu(\alpha)\leq \alpha $ for every $\alpha\in\R$.

\begin{definition} \label{def_formalism_measure}
A  set function $\mu \in\mathcal{H}(\zu^d)$ with  ${\rm supp}(\mu)\neq\emptyset$ is said to {\em obey the multifractal formalism (MF)} over an interval $I\subset \R$ when  for all $\alpha\in I $,
\begin{equation}\label{mfforsetfunctions}
\sigma_\mu(\alpha) =\tau_\mu^*(\alpha) .
\end{equation}

The capacity $\mu \in\mathcal{H}(\zu^d)$ is said to   obey the strong multifractal formalism (SMF) over $I$ if  \eqref{mfforsetfunctions}   holds for all $\alpha\in I$ when  $\dim E_\mu(\alpha) =\tau_\mu^*(\alpha) $. 

When $I=\R$, one simply says that the MF or the SMF holds for $\mu$.
\end{definition}

\begin{remark} 
Note that the   H\"older exponents   are sometimes  defined as
\begin{align*}
\underline h_\mu(x)&=\liminf_{r\to 0^+}\frac{\log \mu(B(x,r))}{\log (r)}  \ \text{ and }  \ \overline h_\mu(x)=\limsup_{r\to 0^+}\frac{\log \mu(B(x,r))}{\log (r)}, \\
\mbox{or } \ \ \underline h_\mu(x)& =\liminf_{j\to\infty}\frac{\log_2 \mu(3\lambda_j(x))}{-j}   \ \text{ and }  \  \overline h_\mu(x)=\limsup_{j\to\infty}\frac{\log _2\mu(3\lambda_j(x) )}{{-j}}
\end{align*}
(after defining $\mu (A)=\mu(A\cap [0,1]^d)$ for $A\in \mathcal{B}(\R^d)$). In this case,  $\mu(\lambda)$ is replaced by  $\mu(3\lambda)$     in the definition of the $L^q$-spectrum. However, in this paper we   mainly consider doubling or ``almost doubling'' capacities for which all the previous notions of exponents, level sets, singularity spectrum and $L^q$-spectrum do not depend on whether  dyadic cubes or centered balls are considered. 
\end{remark}

  When $\mu\in \mathcal{M}([0,1]^d)$, it is known  \cite{LN,Barralinverse} that $\tau_\mu'(-\infty)<+\infty$ if and only if $\tau_\mu$ is finite in a neighborhood of $0^-$, and in this case $\tau_\mu:\R\to\R $ is a non-decreasing, concave map with $\tau_\mu(1)=0$. If, in addition, $\mu$ has full support in $\zu^d$, then $\tau_\mu(0)=-d$, and  $\tau_\mu^*$   reaches its maximum, equal to $d$, exactly over the interval $[\tau_\mu'(0)^-,\tau_\mu'(0)^+]$. Moreover,
  $$\mathrm{dom}(\tau_\mu^*)=[\tau_\mu'(+\infty),\tau_\mu'(-\infty)] = \{\alpha\in\R: \tau_\mu^*(\alpha)\ge 0\}.$$

\sk
\begin{definition}\label{def2.5}
Let $\TD$ be the set of concave increasing functions $\tau:\R\to \R$ such that $\tau(1)=0$, $\tau(0)=-d$ and $\mathrm{dom}(\tau^*)$ is a compact subset of $(0,+\infty)$. 

Let $\SD$ be the set of functions $\sd: \R\to [0,d]\cup \{-\infty\}$ such that  $\sd$ is compactly  supported with support included in $(0,+\infty)$, concave,  continuous, $\sd\leq \mbox{Id\,}_{\R}$ and there exist two exponents  $D,D' >0$ such that  $\sd(D)=D$ and $\sd(D')=d$.  
\end{definition}

The set $\TD$ is the class of admissible $L^q$-spectra associated with measures fully supported on $\zu^d$, and $\SD$ is the class  of admissible singularity spectra for measures strongly obeying the  \MF  with an $L^q$-spectrum in $\TD$. One easily checks that these two sets $\SD$ and $\TD$ are Legendre transforms of each other.

Note that $\SD$  is similar to the set  $\SDs$ introduced in Conjecture \ref{FP}, up to two differences.  First, there is an extra  condition    $\sigma\leq Id_{\R}$, which is necessary for $\si$ to be the  singularity spectrum of a Radon measure. Second,  the  existence of two exponents  $D,D' >0$ such that  $\sd(D)=D$ and $\sd(D')=d$  is necessary to be the singularity spectrum of a fully supported measure  obeying the \MF  (see Remark~\ref{justif} in Section~\ref{AddMF} for justifications of these facts). Observe also that, the set $\SDs$    defined in  Conjecture~\ref{FP} is related to $\SD$ by the formula
$$\SDs = \{\si (s\cdot) : \si \in\SD, \  s>0\}.$$

Given $\sigma\in \mathscr{S}_d$, a natural question concerns the existence of a fully supported  $\mu\in \mathcal M([0,1]^d)$ obeying the \SMF and satisfying  $\sigma_f=\sigma$. The answer is positive, and the measures solving the problem may even possess additional properties introduced now.

\begin{definition}\label{defAD} 
Let $\Phi $ be the set of non decreasing functions $\phi:\N\to\R_+$ such that $  \lim_{j\to+ \infty }\frac{ \phi(j)}{ j }=0$ .

A capacity  $\mu\in \mathcal{C}(\R^d)$ is {\em almost doubling} when  there exists  $\phi \in \Phi$  such that 
\begin{equation} 
\label{ad}
\mbox{ for all $x\in\supp(\mu)$ and $j\in \N$, } \  \mu(3\lambda_j(x))\le e^{\phi(j)}\mu(\lambda_j(x)).
\end{equation}
\end{definition}

Equivalently, there is a mapping $\phi:(0,1]\to \R_+$ such that $\lim_{r\to 0^+}\frac{\phi(r)}{\log(r)}=0$ and for all $x\in\supp(\mu)$ and $r\in (0,1]$ one has 
\begin{equation*}
\label{ad2}
 \mu(B(x,2r)) \le e^{\phi(r)}\mu(B(x,r)).
\end{equation*} 
When $\phi$ is constant,  the capacity $\mu$ is  doubling in the usual sense.   

\begin{definition}\label{mildcond1} 
A set function $\mu\in\mathcal{H}(\R^d)$ satisfies  property (P) if there exist  $C,s_1, s_2>0$ and  $\phi\in\Phi$  such that:

\begin{itemize} 
\item[(P$_1$)] for all $j\in\N$ and $\lambda\in \mathcal D_j$,
\begin{equation}
\label{minmaj1}
C^{-1} 2^{-j s_2}\le \mu(\lambda)\le C 2^{-j s_1}.
\end{equation} 

\item[(P$_2$)]   for all $j,j'\in\N$ with $j'\ge j$,  for  all $\lambda, \widetilde \lambda\in\mathcal D_j$  such that $\partial\lambda\cap \partial  \widetilde \lambda\neq\emptyset$,  and $\lambda'\in\mathcal D_{j'}$ such that $\lambda'\subset \lambda$:
\begin{equation}\label{propmu}
C^{-1}2^{-\phi(j)} 2^{(j'-j)s_1}  \mu(\lambda')\le \mu(\widetilde \lambda)\le C2^{\phi(j)} 2^{(j'-j)s_2}  \mu(\lambda').
\end{equation}
\end{itemize} 
\end{definition}

For $\mu\in \mathcal{H}(\R^d)$,   (P$_1$) is a uniform H\"older control, from above and below, of $\mu$, and (P$_2$) is a rescaled version of (P$_1$), which implies the almost doubling property. 
 Our result on prescription of multifractal behavior for measures  is the following.


\begin{theorem}\label{th-construct-measures}
There exists a family of measures $\MD$ in $\mathcal M(\mathbb R^d)$ such that :
\begin{enumerate}
\item  Every $\mu \in \MD$ is $\mathbb Z^d$-invariant, fully supported on $\R^d$, satisfies property (P), and  $\mu_{|\zu^d}$   obeys the  SMF.
\item $\SD=\{\sigma_{\mu_{|\zu^d}}:\,\mu\in \mathcal M_d\}$. 
\end{enumerate}
 \end{theorem}
The family $\MD \subset \mathcal{M}(\mathbb R^d) $  is built in Section \ref{wqb}, by   constructing, for every $ \sd\in  \SD$, a fully supported  Borel probability measure $\mu$ on $\zu^d$, which    obeys the SMF, and such that $\sigma_\mu=\sd$. Then  $\MD$ is obtained by periodization of such measures $\mu$. 

\sk 
  Theorem \ref{th-construct-measures}  can be equivalently stated as follows: for every  $\tau\in \TD$, there exists a Borel probability measure $\mu$ with support equal to $\zu^d$, which     obeys the \SMF and such that ${\tau}_\mu=\tau$. This result was established in \cite{Barralinverse}, but  only for measures with totally disconnected support. 
A quite  different method is used in the present article, and the obtained measures possess additional properties.

\sk

\noindent To solve the Frisch-Parisi conjecture \ref{FP},   a class of capacities larger than $\MD$  is~needed.

\begin{definition}\label{defEd}
The   set  $\MDs \subset \mathcal{C}(\R^d)$ is  defined as the set of positive powers of measures $\mu\in \MD$, i.e.
 \begin{equation}
\label{defmd*}
\MDs=\{\mu^s:\mu\in\mathcal M_d,  \ s>0\}.
\end{equation}
An element of $\MDs$ is called a multifractal environment. 
\end{definition}

\begin{remark} 
 (1) A  direct computation  shows that for any $s>0$ and any   $\mu\in\mathcal{H}(\R^d)$,
$$\mbox{for every $t\in \R$, } \ \ \tau_{\mu^s_{|\zu^d}} (t) = \tau_{\mu_{|\zu^d}}(st).$$

\noindent (2) It is immediate to check that as soon as  $\mu\in\mathcal H(\R^d)$ satisfies property $( P)$,  the set functions  $\mu^s$, $\mu^{(+s)}$,  and $\mu^{(-s)}$ of Definition \ref{defmutothes}  satisfy   (P) as well (when $s$ is small enough for  $\mu^{(-s)}$), and that $\mu^s_{|\zu^d}$ has  $H\mapsto \sigma_{\mu_{|\zu^d}}(H/s)$ as  singularity spectrum.
\end{remark}

\subsection{Besov spaces in almost doubling environments and their wavelet characterization}\label{BSME}
Standard Besov spaces can be defined by using $L^p$ moduli of smoothness, and are characterized using decay rate of   wavelet coefficients.   To define   Besov spaces in multifractal environment,    the classical definition of $L^p$ moduli of smoothness is extended as follows.
 
\begin{definition}
For $h\in \R^d$ and $f:\R^d \to\R$, consider   the finite difference operator   $\Delta_h f: x\in\R^d\mapsto  f(x+h)-f(x)$. Then, for $n\geq 2$, set $\Delta ^n_hf = \Delta_h ( \Delta^{n-1}_h f)$.

For every fully supported set function $\mu\in\mathcal{H}(\mathbb R ^d)$, for every $n\in\N^*$, $h\in\R^d\setminus\{0\}$ and $x\in \R^d$, set  
\begin{equation*} 
\Delta ^{\mu,n}_hf(x)  = \frac{\Delta ^n_hf(x)}{\mu(B[x,x+nh])},
\end{equation*}
where for $x,y\in\R^d$, $B[x,y]$ stands for the Euclidean ball of diameter $[x,y]$. 

\medskip

For $p\in [1,+\infty]$, the  $\mu$-adapted $n$-th order $L^p$ modulus of smoothness of $f$  is defined at any $t>0$  by
\begin{align}
\label{defomegat}
  \omegat_n(f,t, \mathbb R^d) _p &= \sup_{t/2\le |h|\leq t} \| \Delta ^{\mu,n}_hf \|_{L^p(\mathbb R^d)}.
\end{align}
  \end{definition}

Observe that when $\mu(E)=1$ for every set $E$, then $ \omegat_n(f,t, \mathbb R^d) _p$ is a modification  of the standard  $n$-th order $L^p$ modulus of smoothness of $f$  defined by
\begin{align}
\label{defomega}
 \omega_n(f,t, \mathbb R^d) _p &= \sup_{0\le  |h|\leq t} \| \Delta ^n_hf \|_{L^p(\mathbb R^d)}.
\end{align}
The relation between \eqref{defomegat} and \eqref{defomega} are investigated in Section \ref{sec-description}, one shall keep in mind that they coincide for regular doubling measures $\mu$.

%
\medskip  
  
Recall that when $s>0$, and $p,q\in [1,+\infty]$, the Besov space $B^{s,p}_q(\R^d)$ is the set of  those functions $f:\R^d\to \R$ such that $\|f\|_{L^p(\R^d)}<+\infty$ and 
\begin{equation}
\label{defbosovosc}
 |f|_{B^{s,p}_q(\mathbb R^d)} = \|(2^{js}  (\omega_n(f, 2^{-j}, \R^d) _p)_{j\in\N}\|_{\ell^q(\N)} <+\infty,
\end{equation} 
where $n\geq s$ is an integer. The dependence on $n$ in   $|f|_{B^{s,p}_q(\mathbb R^d)}$ is voluntarily omitted. Indeed, the norm $ \|f\|_{B^{s,p}_q(\mathbb R^d)} =  |f|_{B^{s,p}_q(\mathbb R^d)} +\|f\| _{L^p(\mathbb R^d)}$ makes $B^{s,p}_q(\R^d)$ a Banach space, and different values of $n>s$ yield equivalent norms (see \cite[Remark 3.2.2]{cohen-book}). 
\begin{definition}[Besov spaces in $\mu$-environment]
\label{devbseovmu}
Let $\mu\in\mathcal{H}(\mathbb R^d)$ satisfy property (P$_1$) of Definition~\ref{mildcond1} with exponents $0< s_1\leq s_2$, and consider an integer $n\ge \lfloor s_2+\frac{d}{p}\rfloor+1$. 

For $1\leq p,q\leq \infty$, the Besov space  in $\mu$-environment  $B^{\mu,p}_q(\mathbb R^d)$ is the set  of those functions $f:\R^d\to \R$ such that $\|f\|_{L^p(\R^d)}<+\infty$   and 
\begin{equation}
\label{defbpqmuosc}
 |f|_{B^{\mu,p}_q(\mathbb R^d) } = \|2^{jd/p}  (\omegat_n(f, 2^{-j}, \mathbb R^d) _p)_{j\in\N}\|_{\ell^q(\N)} <+\infty.
\end{equation}

Finally, let 
\begin{equation*}
\label{defwb}
\wb^{\mu,p}_q(\R^d) = \bigcap_{0<\ep<\min (s_1,1)} B^{\mu^{(-\ep)},p}_q(\R^d). 
\end{equation*}
\end{definition}
 At this stage, both ${B}^{\mu,p}_{q}(\mathbb R^d)$ and $\widetilde {B}^{\mu,p}_{q}(\mathbb R^d)$  depend a priori  on the choice of $n$. However,  the dependence in $n\ge \lfloor s_2+\frac{d}{p}\rfloor+1$ can be dropped for ${B}^{\mu,p}_{q}(\mathbb R^d)$ when~$\mu$ is a doubling capacity, and also for  $\widetilde {B}^{\mu,p}_{q}(\mathbb R^d)$ under the (rather weak)   extra property (P$_2$) of Definition~\ref{mildcond1}  (see Theorem~\ref{th_equivnorm_2} for a precise statement).  Moreover,  endowed with the norm $\|\,\|_{L^p(\R^d)}+ |\, |_{B^{\mu,p}_q(\mathbb R^d) }$, ${B}^{\mu,p}_{q}(\mathbb R^d)$ is a Banach space. Hence, $\widetilde {B}^{\mu,p}_{q}(\mathbb R^d)$ is naturally endowed with a Fr\'echet space structure, as the intersection of a nested family of such spaces.   The Fr\'echet spaces   $\widetilde{B}^{\mu,p}_{q}(\mathbb R^d)$   are the  spaces providing  the solution to the Frisch-Parisi conjecture,  see Section \ref{solution}.

Recall that $\mathcal{L}^d$ stands for the $d$-dimensional Lebesgue measure. For $\mu=(\mathcal{L}^d)^{\frac{s}{d}-\frac{1}{p}}$,  it is proved that  ${B}^{\mu,p}_{q}(\R^d)={B}^{s,p}_{q}(\R^d)$  when $s>d/p$. A multifractal element $\mu\in \mathcal S(\R^d)$ should thus be considered as defining an \textit{heterogeneous} environment imposing local distorsions in the computation of the moduli of smoothness in comparison  to positive powers of $\mathcal{L}^d$, which are   \textit{homogeneous} in space. Like for $B^{s,p}_q(\mathbb R^d)$,  to study the typical multifractal behavior in ${B}^{\mu,p}_{q}(\R^d)$ and $\wb^{\mu,p}_q(\R^d)$, it is essential to establish a wavelet characterization of these spaces. Such a  characterization exists    for $\wb^{\mu,p}_q(\R^d)$ when~$\mu$ is almost doubling, and for ${B}^{\mu,p}_{q}(\R^d)$ when $\mu$ is doubling (see Theorem~\ref{th_equivnorm_2}).

\mk
 
\textbf{Wavelet characterizations}.  It is  standard   that classical Besov spaces are characterized in terms of wavelet coefficients decay.
Let $\Lambda=\bigcup_{j\in\Z} \Lambda_j $, where for $j\in \Z$ 
\begin{equation*}
\label{defLambda}
\Lambda_j=\{(i,j,k): i\in \{1,\ldots,2^d-1\}, \ k\in \Z^d\}.
\end{equation*}

\medskip

Let $\phi$ be a  scaling function and $\{\psi^{(i)}\}_{i=1,...,2^d-1}$ be a family of wavelets associated with $\phi$ so that $(\phi,\{\psi^{(i)}\}_{i=1,...,2^d-1})$ defines a multi-resolution analysis with reconstruction  in $L^2(\R^d)$ (see~\cite[Ch. 2 and 3]{Meyer_operateur} for a general construction). 

For every $\lambda=(i,j,k)\in\Lambda$,   denote by $\psi_\lambda$ the function $x\mapsto \psi^{(i)}(2^j x-k)$. Then, the functions $2^{dj/2}\psi_\lambda$, $j\in\Z$, $\lambda\in\Lambda_j$, form an orthonormal basis of $L^2(\R^d)$, so that  every $f\in L^2(\R^d)$ can be expanded as
$$ 
f=\sum_{j\in\Z}\sum_{\lambda\in\Lambda_j} c_\lambda\psi_\lambda, \quad 
\text{with}\quad  
c_\lambda=\int_{\R^d} 2^{dj}\psi_\lambda(x) f(x)\,\mathrm{d}x
$$
(pay attention to the $L^\infty$ normalisation used to define the wavelet coefficients $(c_\la)_{\la\in \La}$).

\sk 

\begin{definition}\label{fr}
For every $r\in\N$, call  $\mathcal F_r$ the set of those functions $\big \{\phi,\{\psi^{(i)}\}_{i=1,...,2^d-1}\big\}$  which define an $r$-regular  multi-resolution analysis with reconstruction  in $L^2(\R^d)$ and such that $\phi$ and the $\psi^{(i)}$ are compactly supported,  $r$ times continuously differentiable functions, and every  $\psi^{(i)}$ has  $r$ vanishing moments. 
\end{definition}
Recall that a mapping $ \psi:\R^d\to\R $  has $r$ vanishing moments when for every multi-index $\alpha\in \N^d$ of length smaller than or equal to $r$,  $\int_{\R^d}x_1^{\alpha_1}\cdots x_d^{\alpha_d} \psi (x)\mathrm{d}x=0$.

It is standard that $\mathcal{F}_r \neq \emptyset$ (see \cite[Prop. 4, section 3.7]{Meyer_operateur} for instance).
%

\sk

Fix $r\in\N^*$ and $ \Psi\in \mathcal F_r$. For any $f\in L^p(\R^d)$, $1\le p\le \infty$, set
$$
\beta(k)=\int_{\R^d} f(x) \phi (x-k)\, \mathrm{d}x\quad (k\in\Z^d).
$$
Then $f \in L^p(\R^d)$ can also be written
\begin{equation}
\label{decompf}
f=\sum_{k\in\Z^d}\beta(k) \phi(\cdot-k)+\sum_{j\in\N}\sum_{\lambda\in\Lambda_j} c_\lambda\psi_\lambda.
\end{equation}
 Further, for $r>s>d/p$,   Besov spaces are characterized by their wavelet coefficients as follows (see  \cite[Ch. 6]{Meyer_operateur}, \cite{Triebel}, or \cite[Corollary 3.6.2]{cohen-book}): 
\begin{equation}\label{normbesovbeta}
f\in {B}^{s,p}_{q}(\R^d) \Longleftrightarrow \begin{cases}\beta\in\ell^p(\Z^d),\\
(\ep_j)_{j\in\mathbb N}\in\ell^q(\mathbb  N),\ \text{where } \ep_j=\displaystyle \Big\|\Big(2^{j(s-d/p)}{c_\lambda} \Big )_{\lambda\in\Lambda_j}\Big\|_{\ell^p({\Lambda_j})},
\end{cases}
\end{equation}
 Moreover, the norm $\|\beta\|_{\ell^p(\Z^d)}+\|(\ep_j)\|_{\ell^q(\N)}$ is equivalent to the     norm $ \|f\|_{B^{s,p}_q(\R^d)}$ defined in \eqref{defbosovosc}. Note that the functions $\psi^{(i)}$ then belong to ${B}^{s,p}_{q}(\R^d) $. Also,  ${B}^{s,p}_{q}(\R^d)\hookrightarrow {B}^{s-\frac{d}{p},+\infty}_{+\infty}(\R^d)=\mathscr C^{s-\frac{d}{p}}(\R^d)$.

 \mk
  
Let us  write   $\mu(\lambda)=\mu(\lambda_{j,k})$ for every $\lambda=(i,j,k)$, and   introduce the  quantity 
\begin{equation}
\label{defbesovwavelet}
 |f|_{\mu,p,q} = |f|_{\mu,p,q,\Psi}=\| (\ep^\mu_j)_{j\in\N}\|_{\ell^q(\N)}\mbox{,\  where }  \ \ep^\mu_j=\left\|\left (  \frac{c_{\lambda}}{ \mu(\lambda )}\right )_{\lambda \in \Lambda_j}\right\|_{\ell^p({\Lambda_j})}.
 \end{equation}
 In \eqref{defbesovwavelet}, the wavelet coefficients are computed with the given $\Psi \in \mathcal{F}_r$, but   the dependence on $r$ and $\Psi$ is omitted to make the notations lighter. This is justified by the fact that in what follows, $r$   depends only on  $\mu$, and in the cases  relevant to us (i.e when $\mu$ satisfies (P)), the wavelet characterization of  $\widetilde {B}^{\mu,p}_{q}(\mathbb R^d)$ is independent on $\Psi\in\mathcal F_r$.

\medskip

The wavelet characterizations of $B^{\mu,p}_q(\R^d)$ and $\widetilde B^{\mu,p}_q(\R^d)$ are the following.

\begin{theorem}
\label{th_equivnorm_2} 
Let $\mu \in \mathcal{C}(\R^d)$ be  an almost doubling capacity. Let $0<s_1\le s_2$ and $r=  \lfloor s_2+\frac{d}{p}\rfloor+1$. Suppose that property~(P)  holds for $\mu$  with the exponents $ (s_1,s_2)$ and  that ${B}^{\mu,p}_{q}(\mathbb R^d)$ has been constructed by using the $\mu$-adapted $n$-th order  $L^p$ moduli of smoothness, for some integer  $n\ge r$. Let  $\Psi\in\mathcal F_{r}$. 
   
For every $\ep\in (0,1)$, there exists a constant $C_\ep >1$ such that  for all $f\in L^p(\R^d)$, 
\begin{align}
\label{eqnorm} \|f\|_{L^p(\mathbb R^d)}+|f|_{\mu,p,q}& \leq C_\ep(  \|f\|_{L^p(\mathbb R^d)}+ |f|_{{B}^{\mu^{(+\ep)},p}_{q}(\mathbb R^d)  }),\\
\label{eqnorm2} \|f\|_{L^p(\mathbb R^d)}+ |f|_{{B}^{\mu,p}_{q}(\mathbb R^d)} & \leq C_\ep( \|f\|_{L^p(\mathbb R^d)}+|f|_{\mu^{(+\ep)},p,q} )
.
\end{align}
 
 Moreover, when $\mu$ is doubling and satisfies property (P) with $\phi=0$,  the norms $\|\ \|_{L^p}+ |\ |_{\mu,p,q} $ and $ \|\ \|_{L^p}+|\ |_{B^{\mu,p}_q(\R^d)} $ are equivalent.
 \end{theorem}

As a consequence, when $\mu$ is doubling and satisfies (P) with $\phi=0$ (this occurs when $\mu$ is a Gibbs measure, see Remark \ref{remgibbs}), the space $B^{\mu,p}_q(\R^d)$ possesses two equivalent definitions based either on $L^p$ moduli of smoothness  or on wavelet coefficients, and this definition is independent of the choice of $n\ge r$ and $\Psi\in\mathcal F_r$. For $\widetilde{B}^{\mu,p}_q(\R^d)$, when $\mu$ satisfies property (P), combining \eqref{eqnorm} and \eqref{eqnorm2} shows that $f\in\widetilde{B}^{\mu,p}_q(\R^d)$ if and only if $ \|f\|_{L^p(\mathbb R^d)}+ |f|_{{B}^{\mu^{(+\ep)},p}_{q}(\mathbb R^d)  }<+\infty$ for every $\ep>0$, hence also giving a wavelet characterization of $\widetilde{B}^{\mu,p}_q(\R^d)$. 

Moreover, given $\Psi\in \mathcal F_r$, the family of Banach spaces 
$$
\Big \{ B_\ep:=({B}^{\mu^{(-\ep)},p}_q(\R^d),  \|\cdot  \|_{L^p(\mathbb R^d)}+ | \cdot |_{\mu^{(-\ep)},p,q,\Psi}\Big \}_{0<\ep<\min (s_1,1)}$$ 
is non decreasing, and $B_\ep\hookrightarrow  B_{\ep'}$ for all $0<\ep<\ep'<\min (s_1,1)$. This implies that the  space $\widetilde{B}^{\mu,p}_q(\R^d)$ can be endowed with a Fr\'echet space structure,  of which a countable  basis of neighborhoods of the origin is given  by
\begin{equation}\label{Nm0}
\left\{\mathcal N_m  = \left\{f\in \widetilde {B}^{\mu,p}_{q}(\R^d): \, \|f\|_{L^p(\mathbb R^d)}+ |f|_{{B}^{\mu^{(-\frac{1}{m})},p}_{q}(\mathbb R^d)}< \frac{1}{m}\right \}\right \}_{\substack{m\in\N,\\ m>\max (1,s^{-1}_1)}}.
\end{equation}

\begin{remark}
\label{remark-inclusions}
(1) When    (P$_1$) is satisfied,     the   embeddings   $
  {B}^{s_2+\frac{d}{p},p}_{q}(\R^d)\hookrightarrow  {B}^{\mu,p}_{q}(\R^d) $ $ \hookrightarrow {B}^{s_1+\frac{d}{p},p}_{q}(\R^d)$ and $ {B}^{\mu^{(+\ep)},p}_{q}(\R^d)\hookrightarrow  {B}^{\mu,p}_{q}(\R^d)$ hold.

 \medskip
 
 \noindent
 (2) It is direct from the proof of Theorem~\ref{th_equivnorm_2} that  under the   weaker assumption that   (P) holds for all $(s'_1, s'_2)$ such that $ 0<s'_1<s_1\le s_2< s'_2$,   the statement remains true. 
 \end{remark}
 
By Remark~\ref{remark-inclusions}\,(2), when $\mu\in\mathscr{E}_d$ (see Definition~\ref{defEd}), since property (P) holds with any $(s_1,s_2)$ such that $0<s_1<\tau_\mu'(+\infty)\le \tau_\mu'(-\infty)<s_2$,   $\widetilde B^{\mu,p}_q(\R^d)$ will always be considered as defined for an integer $n \ge s_\mu$, where
\begin{equation}\label{rmu}
s_\mu=\left \lfloor \tau_\mu'(-\infty)+\frac{d}{p}\right \rfloor+1,
\end{equation} 
and the wavelet characterization of $\widetilde B^{\mu,p}_q(\R^d)$ holds with $\Psi\in \mathcal F_{s_\mu}$. 
%
%

\subsection{Typical  singularity spectrum in   Besov spaces in multifractal environment }\label{jaffres}

Our result on the multifractal nature of the elements of $\widetilde {B}^{\mu,p}_{q}(\R^d)$ when $\mu\in{\MDs}$ (i.e.  powers of  measures $\MD$  defined by \eqref{defmd*}) is the following. The multifractal formalism's validity  is dealt with in the next section.

\begin{theorem} 
\label{main} 
Let $\mu \in \MDs$, let $p,q\in [1,+\infty]$, and consider the mapping
\begin{equation}
\label{deftaumup}
\zeta_{\mu,p}(t)=
\begin{cases}
\displaystyle \frac{p-t}{p}\tau_\mu\left (\frac{p}{p-t} t \right )&\text{if } t\in (-\infty,p) \\
\tau_\mu'(+\infty)t&\text{if } t\in[p,+\infty).
\end{cases}
\end{equation}
\begin{enumerate} 
\item For all $f\in  \widetilde {B}^{\mu,p}_{q}(\R^d) $,
\begin{align}
\label{majodim}
\sigma_f(H)\le 
\begin{cases} \zeta_{\mu,p}^*(H)&\text{ if }H\le \zeta_{\mu,p}'(0^+)\\
\ \ \ d &\text{ if } H> \zeta_{\mu,p}'(0^+) .
\end{cases}
\end{align}
\item For typical  functions  $f\in \widetilde {B}^{\mu,p}_{q}(\R^d)$, 
one has $\sigma_f=\zeta_{\mu,p}^*$. 

\end{enumerate}
\end{theorem}

The possible shapes of $\si_f$ when $f$ is typical in  $\widetilde B^{\mu,p}_q(\R^d)$ are investigated in detail in Section \ref{sec-description} (see Proposition ~\ref{lemtetap}):   depending on the values of $p$ and  $\si_\mu^*(\alpha_{\min})$, various phenomena may occur. See for instance Figures~\ref{typical}  and~\ref{fig_functions2} for a representation of   typical singularity spectrum  in $ \widetilde {B}^{\mu,p}_{q}(\R^d) $, according to whether $\si_\mu(\alpha_{\min})=0$ or $\si_\mu(\alpha_{\min})>0$. Next remark gathers key information, proved in Proposition \ref{lemtetap}.  
\begin{remark}  
\label{rem-spectra}
\begin{enumerate}
\item The map $\zeta_{\mu,p}$ is always concave. Also, it is immediate that
 $\zeta_{\mu,+\infty}=\tau_\mu$, so   typical functions in $\widetilde {B}^{\mu,+\infty}_{q}(\R^d) $ satisfy  $\si_f= \tau_\mu^*$.
\item
 The support of  $\zeta_{\mu,p}^*$ is the compact subinterval $[\zeta_{\mu,p}(+\infty),\zeta'_{\mu,p}(-\infty)] \subset (0,+\infty)$. Moreover,  since $\zeta_{\mu,p}(0)=\tau_\mu(0)=-d$,  the maximum of $\zeta_{\mu,p}^*$  is  $d$, and it is reached at $H$ if and only if $H\in  [\zeta_{\mu,p}'(0^+),\zeta_{\mu,p}(0^-)]$.  

\item One has $\zeta'_{\mu,p}(-\infty)\le \tau'_{\mu}(-\infty)+\frac{d}{p}$ (see the first item of Section~\ref{features}). 

 \end{enumerate}
\end{remark}
\begin{remark}  
\label{rem-spectra2}
 The set of environments $\mathscr E_d$  includes   all the positive powers of the Lebesgue measure $\mathcal L^d$. Taking $s>d/p$ and $\mu =(\mathcal L^d)^{s/d-1/p}$, Theorem~\ref{main} coincides with the celebrated Jaffard's theorem  \cite{JAFF_FRISCH}, which can be stated as follows:
\begin{enumerate} 
\item For all $f\in  B^{s,p}_{q}(\R^d)$,
$\ds 
{\sigma}_f(H)\le 
\begin{cases} \min \big \{p\big (H- (s-\frac{d}{p})\big ),d\big \} &\text{ if }H\ge s-d/p,\\
-\infty&\text{ if }H<s-d/p.
\end{cases}
$ 

\item Typical   $f\in  B^{s,p}_{q}(\R^d)$ satisfy
$\ds 
{\sigma}_f(H)=
\begin{cases}p\big (H- (s-\frac{d}{p})\big )&\text{ if }H\in [s-d/p, s],\\
-\infty&\text{ otherwise.}
\end{cases}
$
\end{enumerate}
 
In this case, $\tau_\mu(t)=(s-d/p) t-d$ so $\tau_\mu'(-\infty)=\tau_\mu'(+\infty)=s-d/p$, $\tau_\mu^*(H)=d$ if $H=s-d/p$ and $-\infty$ otherwise. Hence, $\zeta_{\mu,p}(t)=st-d$ for $t<p$ and $\zeta_{\mu,p}(t)=(s-d/p)t$ for $t\ge p$, whose Legendre transform is   the typical spectrum   in $B^{s,p}_q(\R^d)$.

\end{remark}

\begin{remark} 
\label{remgibbs}
Gibbs measures are a wide class of examples for which Theorem~\ref{main} holds for  $ {B}^{\mu,p}_{q}(\R^d)$ (and not only $\widetilde {B}^{\mu,p}_{q}(\R^d)$), since they are doubling and satisfy (P$_2$) with $\phi\equiv 0$.   Gibbs measure are defined as follows:  let $\varphi:\R^d\to \R$ be a $\Z^d$-invariant real valued H\"older continuous function. The sequence of Radon measures 
$$
\nu_n(\mathrm{d}x)= \frac{\displaystyle\exp\left (S_n\varphi(x)\right)}{\int_{[0,1]^d}  \exp\left (S_n\varphi(t)\right)\mathcal{L}^d(\mathrm{d}t)}\, \mathcal{L}^d(\mathrm{d}x), \quad \text{where }S_n\varphi(x)=\sum_{k=0}^{n-1} \varphi(2^n x),
$$
converges vaguely to a $\Z^d$-invariant Radon measure $ \nu $ fully supported on $\R^d$, called Gibbs measure associated with $\varphi$. Then, $\tau_{\nu_{|[0,1]^d}}(t)=tP(\varphi)-P(t\varphi)$, where $
P( \varphi)=\lim_{n\to+ \infty}\frac{1}{n}\log \int_{[0,1]^d} 2^n \exp\left (S_n \varphi(x)\right)\mathcal{L}^d(\mathrm{d}x)
$ is the topological pressure of $ \varphi $. 
Moreover, $\tau_{\nu_{|[0,1]^d}}$ is analytic (see \cite{ParryPollicott,Pesin}). 

One can show  that,  when $p=+\infty$, or when $\tau_{\nu_{|[0,1]^d}}'(+\infty)=0$,  or when the potential~$\varphi$ reaches its minimum at 0,  the proofs developed hereafter when $\mu\in \mathscr E_d$ remain true (up to slight modifications) when   $\mu=\nu^s $ for some $s>0$,  and   the conclusions of Theorem~\ref{main}    in $ {B}^{\mu,p}_{q}(\R^d)$ and $\widetilde {B}^{\mu,+\infty}_{q}(\R^d)$  still hold. 
 
\end{remark}
%

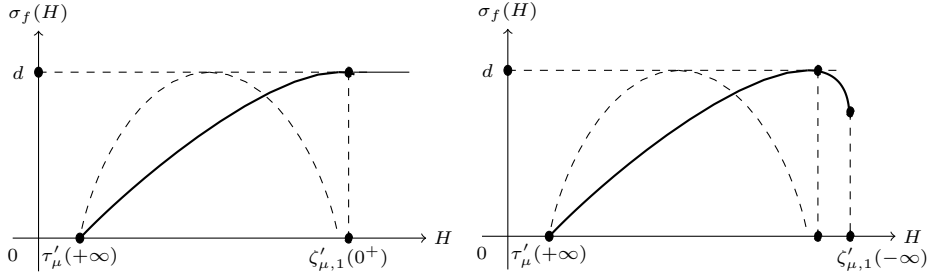
\begin{figure}   \hskip .5cm
 \begin{tikzpicture}[xscale=1.7,yscale=2.2]
    {\tiny
\draw [->] (0,-0.2) -- (0,1.25) [radius=0.006] node [above] {${\sigma}_f(H) $};
\draw [->] (-0.2,0) -- (3.,0) node [right] {$H$};
\draw [dashed, domain=0:5]  plot ({-(exp(\x*ln(1/5))*ln(0.2)+exp(\x*ln(0.8))*ln(0.8))/(ln(2)*(exp(\x*ln(1/5))+exp(\x*ln(0.8)) ) )} , {-\x*( exp(\x*ln(1/5))*ln(0.2)+exp(\x*ln(0.8))*ln(0.8))/(ln(2)*(exp(\x*ln(1/5))+exp(\x*ln(0.8))))+ ln((exp(\x*ln(1/5))+exp(\x*ln(0.8))))/ln(2)});
\draw [dashed, domain=0:5]  plot ({-( ln(0.2)+ ln(0.8))/(ln(2)) +(exp(\x*ln(1/5))*ln(0.2)+exp(\x*ln(0.8))*ln(0.8))/(ln(2)*(exp(\x*ln(1/5))+exp(\x*ln(0.8)) ) )} , {-\x*( exp(\x*ln(1/5))*ln(0.2)+exp(\x*ln(0.8))*ln(0.8))/(ln(2)*(exp(\x*ln(1/5))+exp(\x*ln(0.8))))+ ln((exp(\x*ln(1/5))+exp(\x*ln(0.8))))/ln(2)});
\draw [thick, domain=0:0.2]  plot ({-( ln(0.2)+ ln(0.8))/(ln(2)) +(exp(\x*ln(1/5))*ln(0.2)+exp(\x*ln(0.8))*ln(0.8))/(ln(2)*(exp(\x*ln(1/5))+exp(\x*ln(0.8)) ) )-\x*( exp(\x*ln(1/5))*ln(0.2)+exp(\x*ln(0.8))*ln(0.8))/(ln(2)*(exp(\x*ln(1/5))+exp(\x*ln(0.8))))+ ln((exp(\x*ln(1/5))+exp(\x*ln(0.8))))/ln(2)} , {-\x*( exp(\x*ln(1/5))*ln(0.2)+exp(\x*ln(0.8))*ln(0.8))/(ln(2)*(exp(\x*ln(1/5))+exp(\x*ln(0.8))))+ ln((exp(\x*ln(1/5))+exp(\x*ln(0.8))))/ln(2)});

\draw [thick, domain=0:5]  plot ({-(exp(\x*ln(1/5))*ln(0.2)+exp(\x*ln(0.8))*ln(0.8))/(ln(2)*(exp(\x*ln(1/5))+exp(\x*ln(0.8)) ) ) -\x*( exp(\x*ln(1/5))*ln(0.2)+exp(\x*ln(0.8))*ln(0.8))/(ln(2)*(exp(\x*ln(1/5))+exp(\x*ln(0.8))))+ ln((exp(\x*ln(1/5))+exp(\x*ln(0.8))))/ln(2)} , {-\x*( exp(\x*ln(1/5))*ln(0.2)+exp(\x*ln(0.8))*ln(0.8))/(ln(2)*(exp(\x*ln(1/5))+exp(\x*ln(0.8))))+ ln((exp(\x*ln(1/5))+exp(\x*ln(0.8))))/ln(2)});
 \draw [fill] (-0.1,-0.10)   node [left] {$0$}; 
\draw [fill] (0,1) circle [radius=0.03] node [left] {$d \ $}; 
\draw  [fill] (0.32,0) circle [radius=0.03] node [below] {$\tau_\mu'(+\infty)$}; 
  \draw[dashed] (0,1) -- (2.6,1);
\draw  [fill] (2.4,1) circle [radius=0.03]  [dashed]   (2.4,1) -- (2.4,0)  [fill] (2.4,0) circle [radius=0.03]  node [below] {$\zeta_{\mu,1}'(0^+) $};
\draw  [fill] (2.4,1)-- (2.87,1);

 }
\end{tikzpicture}
 \begin{tikzpicture}[xscale=1.7,yscale=2.2]
    {\tiny
\draw [->] (0,-0.2) -- (0,1.25) [radius=0.006] node [above] {${\sigma}_f(H)$};
\draw [->] (-0.2,0) -- (3.,0) node [right] {$H$};
\draw [thick, domain=0:5]  plot ({-(exp(\x*ln(1/5))*ln(0.2)+exp(\x*ln(0.8))*ln(0.8))/(ln(2)*(exp(\x*ln(1/5))+exp(\x*ln(0.8)) ) ) -\x*( exp(\x*ln(1/5))*ln(0.2)+exp(\x*ln(0.8))*ln(0.8))/(ln(2)*(exp(\x*ln(1/5))+exp(\x*ln(0.8))))+ ln((exp(\x*ln(1/5))+exp(\x*ln(0.8))))/ln(2)} , {-\x*( exp(\x*ln(1/5))*ln(0.2)+exp(\x*ln(0.8))*ln(0.8))/(ln(2)*(exp(\x*ln(1/5))+exp(\x*ln(0.8))))+ ln((exp(\x*ln(1/5))+exp(\x*ln(0.8))))/ln(2)});

\draw [thick, domain=0:1]  plot ({-( ln(0.2)+ ln(0.8))/(ln(2)) +(exp(\x*ln(1/5))*ln(0.2)+exp(\x*ln(0.8))*ln(0.8))/(ln(2)*(exp(\x*ln(1/5))+exp(\x*ln(0.8)) ) )-\x*( exp(\x*ln(1/5))*ln(0.2)+exp(\x*ln(0.8))*ln(0.8))/(ln(2)*(exp(\x*ln(1/5))+exp(\x*ln(0.8))))+ ln((exp(\x*ln(1/5))+exp(\x*ln(0.8))))/ln(2)} , {-\x*( exp(\x*ln(1/5))*ln(0.2)+exp(\x*ln(0.8))*ln(0.8))/(ln(2)*(exp(\x*ln(1/5))+exp(\x*ln(0.8))))+ ln((exp(\x*ln(1/5))+exp(\x*ln(0.8))))/ln(2)});

\draw [dashed, domain=0:5]  plot ({-(exp(\x*ln(1/5))*ln(0.2)+exp(\x*ln(0.8))*ln(0.8))/(ln(2)*(exp(\x*ln(1/5))+exp(\x*ln(0.8)) ) )} , {-\x*( exp(\x*ln(1/5))*ln(0.2)+exp(\x*ln(0.8))*ln(0.8))/(ln(2)*(exp(\x*ln(1/5))+exp(\x*ln(0.8))))+ ln((exp(\x*ln(1/5))+exp(\x*ln(0.8))))/ln(2)});
\draw [dashed, domain=0:5]  plot ({-( ln(0.2)+ ln(0.8))/(ln(2)) +(exp(\x*ln(1/5))*ln(0.2)+exp(\x*ln(0.8))*ln(0.8))/(ln(2)*(exp(\x*ln(1/5))+exp(\x*ln(0.8)) ) )} , {-\x*( exp(\x*ln(1/5))*ln(0.2)+exp(\x*ln(0.8))*ln(0.8))/(ln(2)*(exp(\x*ln(1/5))+exp(\x*ln(0.8))))+ ln((exp(\x*ln(1/5))+exp(\x*ln(0.8))))/ln(2)});
  \draw[dashed] (0,1) -- (2.6,1);
\draw [fill] (-0.1,-0.10)   node [left] {$0$}; 
\draw [fill] (0,1) circle [radius=0.03] node [left] {$d \ $}; 
\draw  [fill] (0.32,0) circle [radius=0.03] node [below] {$\tau_\mu'(+\infty)$};
\draw  [fill] (2.4,1) circle [radius=0.03]  [dashed]   (2.4,1) -- (2.4,0)  [fill] (2.4,0) circle [radius=0.03]  ;
\draw  [fill] (2.65,0.75) circle [radius=0.03]  [dashed]   (2.65,0.75) -- (2.65,0)  [fill] (2.65,0) circle [radius=0.03]  (2.90,-0.02)  node [below] {$\zeta_{\mu,1}'(-\infty)$};
 }
\end{tikzpicture}
\caption{{\bf Left:} Upper bound for the singularity spectrum of  every $f \in \widetilde {B}^{\mu,1}_q(\R^d)$. {\bf Right:} Singularity spectrum of a typical  $ f\in \widetilde {B}^{\mu,1}_q(\R^d)$. The dashed graph represents the (initial)  singularity  spectrum of $\mu$. When $p=+\infty$ and  $ f$ is typical in $\widetilde {B}^{\mu,+\infty }_q(\R^d)$, $\si_f=\si_\mu$.  }
\label{typical}
\end{figure}

%
%
%

\subsection{Multifractal formalism for functions   in $\widetilde {B}^{\mu,p}_{q}(\R^d)$}\label{secMFF}  
 The formalism used in this paper is based on the one developed by Jaffard in~\cite{JaffardPSPUM}, Let us begin with the definition of wavelet leaders.
\begin{definition}[Wavelet leaders]
Given $ \Psi\in \bigcup_{r\in\N}\mathcal F_r$ and $f\in L^p_{{\rm loc}}(\R^d)$ for   $p\in [1,+\infty]$, denoting the wavelet coefficients  of $f$ associates with  $\Psi$ by $(c_\lambda)_{\lambda\in \Lambda}$, the wavelet leader of $f$ associated with $\lambda\in \mathcal D $ (see Section \ref{statement-1} for the notations)  is defined as: 
\begin{equation}
\label{defleaders}
\ld^f_\lambda=\sup\{|c_{\lambda'}|:\, \lambda'=(i,j,k)\in \Lambda, \,\lambda'_{j,k}\subset 3\lambda\}.
\end{equation}
\end{definition}
Pointwise H\"older exponents of 
H\"older continuous functions  (recall  Definition~\ref{defexp}) are related to the wavelet leaders as follows (see \cite[Corollary 1]{JaffardPSPUM}). 

\begin{proposition} Let $r\in\mathbb \N^*$ and $\Psi\in\mathcal F_r$.  If  $f\in {\mathscr C}^\ep(\R^d)$ for some $\ep>0$,  then  for every $x_0\in \R^d$, $h_f(x_0)<r$ if and only $\liminf_{j\to \infty} \frac{\log \ld^f_{\lambda_{j}(x)}}{\log(2^{-j})}<r$, and in this case 
\begin{equation}
\label{defleaders2}
h_f(x_0)=\liminf_{j\to \infty} \frac{\log \ld^f_{\lambda_{j}(x)}}{\log(2^{-j})}.
\end{equation}
\end{proposition}
Hence, as observed by Jaffard, and rephrased in the language of the present paper, if the support of $\sigma_f$ is bounded and sufficiently smooth wavelets $\Psi$ are used,  then the singularity spectrum $\sigma_f$ of $f$ coincides with the  singularity spectrum  of the capacity $\nu\in \mathcal{C}(\R^d)$ defined by $\nu(B) =  \sup\left \{\ld^f_{\lambda}: \lambda\in\mathcal D,\, \lambda\subset B\right \}$ for all  $B\in\mathcal B(\R^d)$. 

\sk

In order to estimate from above the  singularity spectrum $\sigma_f$ of   $f\in\widetilde {B}^{\mu,p}_{q}(\R^d)$, it is then natural to consider, exactly as it was done for the elements of $\mathcal H([0,1]^d)$,  the $L^q$-spectrum of $f$ relative to $\Psi$ defined as follows:  For any  $N\in \N^*$,    set
\begin{equation}\label{zetaNPsi}
\zeta^{N,\Psi}_f = \liminf_{j\to+ \infty}\zeta^{N,\Psi} _{f,j} ,\text{ where } \  \zeta^{N,\Psi}_{f,j} : t\in\R\mapsto -\frac{1}{j}\log_2 \Big( \!\!  \sum_{ {\lambda\in\mathcal D_j, \  \lambda\subset N[0,1]^d,\,  \ld^f_\lambda>0}} \!\!\!\!\!\!  (\ld_\lambda^f)^t\Big). 
\end{equation}
Recalling the notations of Section~\ref{statement}, $(N[0,1]^d)_{N\in\N^*}$ is the increasing sequence of boxes $[- (N-1)/2,(N+1)/2]^d$, whose union  covers $\R^d$. 
\begin{definition}The  $L^q$-spectrum of $f$ relative to $\Psi$ is the concave function  
\begin{equation}\label{zetaPsif}
\zeta^\Psi_f  = \inf\{ \zeta^{N,\Psi}_f : N\in\N^*\} = \lim_{N\to +\infty}  \zeta^{N,\Psi}_f .
\end{equation}
\end{definition}
 The concavity of $\zeta^\Psi_f$ follows from the fact that  $(\zeta^{N,\Psi}_f)_{N\geq 1}$ is a non-increasing sequence of functions, 

It is  remarkable   that  ${\zeta^\Psi_f}_{|\R_+}$ does not depend on~$\Psi$~\cite[Th. 3]{JaffardPSPUM}. This would be the case over   $\R$ if  $\Psi$  belonged to the Schwarz class~\cite[Th. 4]{JaffardPSPUM}. However, the wavelet characterization of ${B}^{\mu,p}_{q}(\R^d)$ makes it necessary to use compactly supported wavelets, which never belong to $\mathscr C^\infty(\R^d)$ \cite{Daubechies1988}.   
For simplicity,   ${\zeta^\Psi_f}_{|\R_+}$ is simply denoted  by ${\zeta_f}_{|\R_+}$.

\smallskip

 Also, when $H<r$, the Legendre transform $(\zeta^\Psi_f )^*(H)$ of $\zeta^\Psi_f $ at $H$ provides an upper bound for $\dim E_f(H) $, i.e. one has 
\begin{equation}\label{UPF}
\sigma_f(H)\le (\zeta^\Psi_f)^*(H).
\end{equation}

\sk 

Let us now define  the multifractal formalism for functions. It combines Jaffard's  formalism  with wavelet leaders, and a variant of it, mainly used to control the decreasing part of the $\si_f$  whenever it exists. This variant is necessary since   when $\mu\in\mathscr{E}_d$, $q<+\infty$  and the elements of $\Psi$ are smooth, it is generic in  $\widetilde {B}^{\mu,p}_{q}(\R^d)$ that ${\zeta^\Psi_f}_{|\R_-^*}$ equals $-\infty$. Hence,  for $H\ge (\zeta_f)'(0^+)$,  $(\zeta^\Psi_f)^*(H)$ only provides the trivial upper bound $\sigma_f(H)\le d$. 

\begin{definition} 
\label{Multifractal-formalism} Let $r\in\N^*$ and  $f \in \bigcup_{s>0} \mathcal{C}^s(\R^d)$. Suppose that ${\sigma}_f$ has a compact domain included in $(0,r)$.  Let $I \subset \dom(\si_f)$ be a compact interval.  
\begin{enumerate}
\item The wavelet leaders  multifractal formalism (WMF)  holds for $f$ on $I$ when there exists $\widetilde r\ge r$ such that  for all $H\in I$ and all $\Psi\in \mathcal{F}_{\widetilde r}$, , ${\sigma}_f(H)=(\zeta_f^\Psi)^*(H)$.

\smallskip

\item  The weak wavelet leaders multifractal formalism (WWMF)  holds for $f$ on $I$ relatively to $\Psi\in\mathcal F_r$ when the following property holds:   there exists an increasing sequence $(j_k)_{k\in \mathbb N}$ such that for all $N\in \N$, $\lim_{k\to\infty}  \zeta^{ {N },\Psi}_{f,j_k} =\zeta^{(N),\Psi}_{f,{\rm w}} $ exists, and setting $\zeta_{f,{\rm w}}^\Psi=\lim_{N\to+\infty} \zeta^{(N),\Psi}_{f,{\rm w}}  $, one has   ${\sigma}_f(H)=(\zeta_{f,{\rm w}}^\Psi)^*(H)$ for all $H\in I$. 
\end{enumerate}
\end{definition}

\begin{remark}\label{remMFF}
(1) In the increasing part of $\sigma_f$, i.e. when  $H\le (\zeta_f)'(0^-)$, item (1) of the previous definition coincides with the multifractal formalism associated with wavelet leaders considered by Jaffard (see~\cite{JaffardPSPUM} for instance).

%

\medskip
\noindent
(2) Contrarily to   \eqref{UPF}, in general, even if there exists such a subsequence $(j_k)_{k\in\N}$ making it possible to define $\zeta_{f,{\rm w}}^\Psi$, one cannot get the a priori inequality $\sigma_f\le (\zeta_{f,{\rm w}}^\Psi)^*$.  This justifies  the terminology ``weak''. Nevertheless, the existence of $\zeta_{f,{\rm w}}^\Psi$ emphasizes  that  the sequences $(\zeta^{ {N },\Psi}_{f,j}(t))_{j\in\N}$ converge along the same subsequence for all $N$ and $t$.  This property is typical in $\widetilde {B}^{\mu,p}_{q}(\R^d)$, and holds simultaneously for countably many $\Psi$'s.

\medskip
\noindent
(3) If the WWMF holds on $I$ relatively to   $\Psi$ and $\widetilde \Psi$ in $\mathcal F_r$, then $\zeta_{f,{\rm w}}^\Psi = \sigma_f^*=\zeta_{f,{\rm w}}^{\widetilde\Psi} $ on the interval $\bigcup_{H\in I}   \partial {\sigma_f}_{|I}(H)$ ($\partial {\sigma_f}_{|I}$ is  the subdifferential of the concave function ${\sigma_f}_{|I}$). 
 
\end{remark}

  Theorem~\ref{main}  can now be completed by the following result on the validity of the MF. Recall \eqref{deftaumup} and \eqref{rmu} for the definitions of $\zeta_{\mu,p}$ and $s_\mu$ respectively, as well as Remarks~\ref{remark-inclusions}\,(2) and~\ref{rem-spectra}.
\begin{theorem}[Validity of the multifractal formalism]
\label{validity}
Let $\mu\in \MDs $. 
\begin{enumerate} 
\item For all $f\in \widetilde {B}^{\mu,p}_{q}(\R^d)$, one has ${\zeta_f}_{|\R_+}\ge {\zeta_{\mu,p}}_{|\R_+}$. 

\item Typical functions  $f\in \widetilde {B}^{\mu,p}_{q}(\R^d)$ satisfy the WMF  on  $[\zeta_{\mu,p}'(+\infty),\zeta_{\mu,p}'(0^+)]$ (i.e. in the increasing part of $\sigma_f$), and ${\zeta_f}_{|\R_+}={\zeta_{\mu,p}}_{|\R_+}$.  

\item 
\begin{itemize} \item [(i)]Let $\Psi\in\mathcal F_{s_\mu}$. Typical functions  $f\in \widetilde {B}^{\mu,p}_{q}(\R^d)$ satisfy  the WWMF  on 
${\rm dom}(\sigma_f)=[\zeta_{\mu,p}'(+\infty),\zeta_{\mu,p}'(-\infty)]$ relatively to~$\Psi$, with $\zeta_{f,{\rm w}}^\Psi={\sigma}_f^*=\zeta_{\mu,p}$. Moreover,  if $q<+\infty$, the property ${\zeta^\Psi_f}_{|\R^*_-}=-\infty$ is typical as well. 

\item [(ii)] Given a countable subset $\mathcal F$ of $\mathcal F_{s_\mu}$, typical functions  $f\in \widetilde {B}^{\mu,p}_{q}(\R^d)$ satisfy the  WWMF  on the  interval ${\rm dom}(\sigma_f)$ 
relatively to any $\Psi\in \mathcal F$, with $\zeta_{f,{\rm w}}^\Psi={\sigma}_f^*=\zeta_{\mu,p}$, and ${\zeta^\Psi_f}_{|\R^*_-}=-\infty$ if $q<+\infty$. 
\end{itemize}
\end{enumerate}
\end{theorem}

In other words, when $\mu\in\mathscr E_d$, for typical functions in $  \widetilde {B}^{\mu,p}_{q}(\R^d)$, the   WMF  holds in the increasing part of the spectrum, while the WWMF  holds in the stronger form stated in Theorem~\ref{validity}(3)(ii) on the whole domain of the spectrum. Also, it is not possible to substitute ${\zeta^\Psi_f}$ to $\zeta_{f,{\rm w}}$, at least when $q<+\infty$, since  ${\zeta^\Psi_f}_{|\R^*_-}=-\infty$. 

\begin{remark}\label{BesMF}  For   standard Besov spaces (when viewed as  ${B}^{\mu,p}_{q}(\R^d)$ with $\mu$ is a positive power of  Lebesgue measure), Theorem \ref{validity} shows that if $q<+\infty$, generically a function $f$ in such a space satisfies   ${\zeta_f}_{|\R_+}={\zeta_{\mu,p}}_{|\R_+}$ and ${\zeta_f}_{|\R_-^*}=-\infty$. 
\end{remark}

\subsection{Solutions to the Frisch-Parisi conjecture. Proof of Theorem~\ref{solution}} \label{fpc} 
  
It is worth recalling the results  by  Jaffard in \cite{JAFF_FRISCH}. Consider an increasing continuous and concave  function $\eta:\R_+ \to\R^+$, with positive slope $\eta'(+\infty)$ at $\infty$, such that $\eta (0)\in [0,d]$, and  $\eta^*$ takes values in $[-d,0]$ over its domain. Setting $\zeta=\eta-d$, Jaffard seeks for a Baire space in which the increasing part of the typical  singularity spectrum is given by $\zeta^*$. He works with the so-called homogeneous Besov spaces $\dot{B}^{s,p}_q(\R^d)$, introduced   the Baire space $V=\bigcap_{\epsilon>0}\bigcap_{t>0} \dot{B}^{(\eta(t)-\epsilon)/t,t}_{t,{\rm loc}}(\R^d)$  \cite{JAFF_FRISCH}
and proved that for  typical functions $f \in V$,  $\sigma_f = \widetilde\zeta^*$ , where
$$
\widetilde \zeta(t)=\begin{cases}
d(t/t_c-1)&\text{if } t<t_c\\
 \zeta(t)&\text{if } t\ge t_c
 \end{cases},
 $$
$t_c$ being the unique solution of $\zeta (t_c)=0$. In particular, $\sigma_f$ is necessarily increasing, with domain $[\zeta'(+\infty),d/t_c]$, and with an affine part over the interval $[\zeta'(t_c+),d/t_c]$. Also, $\sigma_f$ coincides with $\zeta^*$  over $[\zeta'(+\infty), \zeta'(t_c+)]$. 

In addition,  in the multifractal formalism used in \cite{JAFF_FRISCH}, the scaling function $\zeta_f (t)$ is defined as $\sup\{s\ge 0: f\in \dot B^{s/t,t}_{\infty,\mathrm{loc}}(\R^d)\} -d$ for $t>0$, and with this definition typical functions in $V$  satisfy $\zeta_f=\zeta$. Thus the associated multifractal formalism holds on $[\zeta'(+\infty), \zeta'(t_c+)]$ only. However, it can be checked that the  WMF does hold for $f$ with $\zeta_f=\widetilde\zeta$ on $[\zeta'(+\infty), d/t_c]$.  

Hence, although this approach was a substantial progress, it allowed to reach only increasing singularity spectra, necessarily  composed   by  an affine part followed by a concave part. Up to now, no better solution  to Conjecture~\ref{FP} has been proposed.
 
\medskip

Combining the previously stated   results (Theorems \ref{th-construct-measures}, \ref{main} and \ref{validity}), we are now able to prove Theorem~\ref{solution}, and bring a positive answer to Conjecture \ref{FP}. The solutions provided in this proof are of the form $\widetilde {B}^{\mu,+\infty}_{q}(\R^d)$. Solutions of the form $\widetilde {B}^{\mu,p}_{q}(\R^d)$ with $1\le p<+\infty$ are considered in Theorem~\ref{exhaustion} below. 



\begin{proof}[Proof of Theorem~\ref{solution}] {\bf (Solutions of the form $\widetilde {B}^{\mu,+\infty}_{q}(\R^d)$)}
Let $\sigma \in \SDs$. Let  $\sigma_{\mathcal M}=\sigma(\cdot/s)$, where $s$ is the unique positive real number such that $\sigma(\cdot/s) \leq \text{Id}_{\R}$ and there exists at least one $H$ such that   $\sigma  (H/s)=H$. In other words,  $s$ is the unique  real number such that $\sigma^*(s)=0$. In particular,   $\sigma_{\mathcal M} \in  \SD$. By Theorem \ref{th-construct-measures}, there exists  $\mu\in \mathcal M_d$  such that $\tau_\mu=\sigma_{\mathcal M}^*$.
 
Now, we apply Theorems \ref{main} and \ref {validity} with the capacity $\mu^{s}\in\mathscr{E}_d$: in the Baire space  $\widetilde {B}^{\mu^{s},+\infty}_{q}(\R^d)$, typical functions have $\sigma$ as  singularity spectrum,   satisfy the  WMF  in the increasing part, and  also satisfy the WWMF  over ${\rm dom}(\sigma)$ relatively to any $\Psi$ in a countable family of elements of $ \mathcal F_{s_{\mu^s}}$. 

 Hence, for any $q\in[1,+\infty]$, the space $\widetilde {B}^{\mu^{s},+\infty}_{q}(\R^d)$  provides  a solution to the Conjecture~\ref{FP} with initial data $\sigma$.  
\end{proof}

It is natural to seek for other solutions using the spaces $\widetilde {B}^{\mu,p}_{q}(\R^d)$ with  $1\le p<+\infty$, $q\in[1,+\infty)$ and $\mu\in\mathscr{E}_d$. 

A first observation is the following: Suppose $\nu\in\mathscr{E}_d$, $1\le p<+\infty$ and $1\le q\le \infty$. Let $\sigma\in \SDs$ be the typical  singularity spectrum in $\widetilde {B}^{\nu,p}_{q}(\R^d)$ given by~\ref{main}. Considering $\mu\in\mathcal M_d$ as in the previous proof yields for all $q'\in[1,+\infty]$ the space $\widetilde {B}^{\mu^{s},+\infty}_{q'}(\R^d)$ in which the typical multifractal structure is the same as in $\widetilde {B}^{\nu,p}_{q}(\R^d)$. 
However, much more can be said (the proof of the following result is given in Section~\ref{proof of exhaustion}).

\begin{theorem}[{\bf Solutions of the form $\widetilde {B}^{\mu,p}_{q}(\R^d)$ with $p<+\infty$}]\label{exhaustion}
Let $\sigma\in\SDs$ and denote its domain by $[H_{\min},H_{\max}]$.  
\begin{enumerate}
\item If  $\sigma$ is the typical singularity spectrum in $\widetilde {B}^{\mu,p}_{q}(\R^d)$ for some $1\le p<+\infty$, $q\in[1,+\infty]$ and $\mu\in\mathscr{E}_d$,  then $\sigma(H_{\min})=0$ and $\sigma'(H_{\min}^+)\le p$. 

\smallskip

\item  Suppose $\sigma(H_{\min})=0$ and $\sigma'(H_{\min}^+)<+\infty$. For all $p\in \big [\max (1,\sigma'(H_{\min}^+)),+\infty\big )$ there exists  $\mu\in\mathscr{E}_d$ such that for all $q\in[1,+\infty]$, $\sigma$ is the singularity spectrum of the typical elements of $\widetilde {B}^{\mu,p}_{q}(\R^d)$. Moreover, if $\sigma(H_{\max})>0$ and $\sigma'(H_{\max}^-)=-\infty$, there are infinitely many such $\mu\in\mathscr{E}_d$, with distinct singularity spectra. 

Also, for each such $\mu$, typical functions in $\widetilde {B}^{\mu,p}_{q}(\R^d)$  satisfy the  WMF in the increasing part of $\sigma$ and   the WWMF  over $[H_{\min},H_{\max}]$ relatively to any $\Psi$ in a countable family of elements of $ \mathcal F_{s_{\mu}}$. 

 \end{enumerate}
\end{theorem}

Let us make a final remark.  Like for Besov spaces, one can let $p$ or $q$ take values in $(0,+\infty]$ in the definition of Besov spaces in multifractal environment, and all our results are valid, the only change to make being to take $p\in \big [\sigma'(H_{\min}^+)),+\infty\big )$ in Theorem~\ref{exhaustion}(2). This provides a larger set of solutions to the inverse problem~\ref{FP}.

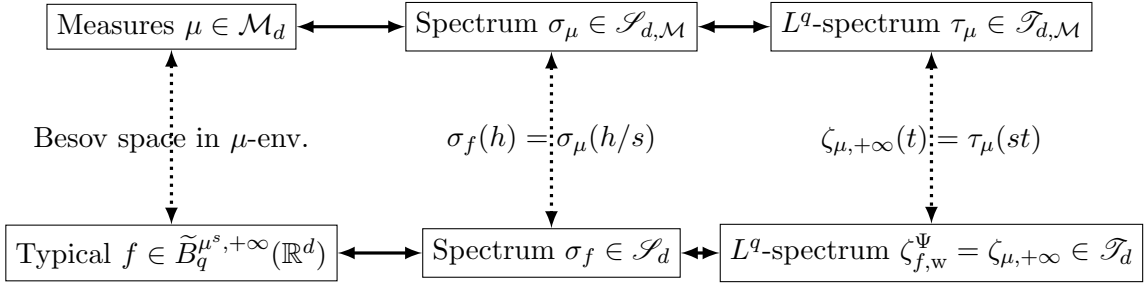
\begin{figure}
\hspace{-7mm}
\begin{tikzpicture}
\tikzstyle{quadri}=[rectangle,draw, ,text=black]
\tikzstyle{cir}=[ellipse,draw,fill=red!10,text=blue]
\tikzstyle{fleche}=[<->,very thick,>=latex]
\tikzstyle{fleche3}=[->,dotted,very thick,>=latex]
\tikzstyle{fleche2}=[<->,dotted,very thick,>=latex]

\node[quadri] (L) at (-4,3) {Measures $\mu\in\mathcal{M}_d$};
\node[quadri] (F) at (-4 ,0) {Typical $f \in \wb^{\mu^s,+\infty}_q(\R^d) $}; 
\draw[fleche2] (L)--(F);

\node[quadri]  (SP) at (1,3) {Spectrum $\sigma_\mu \in \SD$}; \draw[fleche] (L)--(SP);
 \node[quadri] (SPF) at (1 ,0) {Spectrum  $\sigma_f \in \SDs$}; \draw[fleche] (F)--(SPF);
\draw[fleche2] (SP)--(SPF);

\node (Q) at  (-4,1.5) {Besov space in $\mu$-env.};
\node (Q) at  (1,1.5) { $\sigma_f(h) = \sigma_\mu(h/s)$};

\node (Q) at  (6,1.5) {$\zeta_{\mu,+\infty}(t) = \tau_\mu(st)$ };

\node[quadri] (SF) at (6,3) {$L^q$-spectrum   $\tau_\mu \in \TD$}; \draw[fleche] (SP)--(SF); 
\node[quadri] (SFF) at (6,0) {$L^q$-spectrum  $\zeta^\Psi_{f,{\rm w}}=\zeta_{\mu,+\infty}\in \TDs$}; \draw[fleche] (SPF)--(SFF); 
\draw[fleche2] (SF)--(SFF);
 
\end{tikzpicture}
\label{fig-correspondence}
\caption{Scheme of the correspondence between the various  objects and sets  in the case where $p=+\infty$, for some parameter $s>0$.}
\end{figure}

  
\subsection{Organization of the rest of the paper} 
Section \ref{wqb} is dedicated to the construction of the class of measures $\MD$ (Definitions \ref{defM1} and \ref{defMd}) with prescribed multifractal behavior as described in Theorem \ref{th-construct-measures}. There,    some properties relating these measures (and associated auxiliary measures) with  the dyadic approximation   in $\R^d$ are proved.  In Section \ref{sec_integral2},  the wavelet characterization of the space $\widetilde B^{\mu,p}_q$ is established when $\mu$ is an almost doubling capacity satisfying property (P) (Theorem \ref{th_equivnorm_2}).
The various shapes of $\zeta_{\mu,p}$ and $\zeta^*_{\mu,p}$ are investigated in Section~\ref{sec-description}, where $\zeta^*_{\mu,p}$ is explicitly expressed in terms of $\tau_\mu^*$; this  expression   turns out to be very useful in the proof of the   WMF's validity  for typical functions. Next, in section \ref{sec-upperbound}, the upper bound for the singularity spectrum of all functions in $\widetilde  B^{\mu,p}_q(\R^d)$ is obtained (part (1) of Theorem \ref{main}), as a consequence of part (1) of Theorem~\ref{validity} which is also proved there. Part (2) of Theorem~\ref{main} is shown in Section \ref{sec-typical}. It consists first in building a specific function whose singularity spectrum  is    typical, and then tin building a dense $G_\delta$-set included in $   \widetilde B^{\mu,p}_q(\R^d)$ in which all functions share the same \ml spectrum. Parts (2) and (3) of Theorem~\ref{validity} are established in Section~\ref{sec-formalism}. Finally, the proof of Theorem~\ref{exhaustion} is given in Section~\ref{proof of exhaustion}.


\section{Measures with prescribed multifractal behavior}
\label{wqb}
In Section~\ref{AddMF}, additional general properties associated with multifractal formalism for capacities are recalled. Section~\ref{constr.meas.} is a preparation to the construction of the measures satisfying the requirements of Theorem~\ref{th-construct-measures}. The construction is achieved when $d=1$ in Section~\ref{constrmu}. Then, in Sections~\ref{musatP} to~\ref{muspectrum}  the conclusions of Theorem~\ref{th-construct-measures} are obtained. The construction is extended to the case $d\ge 2$ in Section~\ref{dimensiond}. This yields the desired family of measures $\mathcal M_d$. Finally, in Sections~\ref{ubi} and ~\ref{secdiop} we investigate some connections between the elements of $\mathcal{M}_d$ and metric number theory:   a ubiquity theorem associated with $\mu\in \mathcal{M}_d$ and the family of dyadic vectors is established, and it is proved that auxiliary measures associated with $\mu\in \mathcal{M}_d$   are supported on the set of points which are badly approximated by dyadic vectors. All these properties are key for  the multifractal analysis of typical elements of   $\widetilde B^{\mu,p}_q(\R^d)$.  
\subsection{Additional notions related to the multifractal formalism for capacities}\label{AddMF}

Let us introduce the notations for $\ep>0$, $\alpha\in \R$, and $I=[a,b]$ an interval:
\begin{align}
\label{defpmep}
 \alpha\pm \ep  =[ \alpha-\ep, \alpha+\ep] \ \ \ \mbox{ and } \ \ \ 
 I\pm\ep  = [a-\ep,b+\ep].
\end{align}

Also, the   convention  $\log(0)=-\infty$ is adopted.

Next propositions complete the properties mentioned in Section~\ref{constr.meas.} about multifractal analysis of capacities.  Recall that the   Legendre  spectrum $\alpha\mapsto {\tau_\mu^*}(\alpha)$ is increasing on  the interval   $ \alpha  \le \tau_\mu'(0^-)$, and is decreasing on $\alpha\geq  \tau_\mu'(0^+)$.  The following Propositions \ref{fm00}  and \ref {fm}  are  easily deduced from any of the following sources \cite{Barralinverse,BRMICHPEY,Olsen,JLVVOJAK,BBP}.  

\begin{proposition}
\label{fm00}   
Let $\mu \in \mathcal{C}(\zu^d)$ such that $\supp(\mu)\neq\emptyset$.  
For $\alpha\in \R$,   let 
\begin{equation*} 
 \underline E^{\leq}_\mu( \alpha)  =\{x \in\supp(\mu): \underline h_\mu(x) \leq  \alpha\}   \ \ \mbox{ and } \ \ 
\overline E^{\geq}_\mu( \alpha)  =\{x \in \supp(\mu): \overline h_\mu(x) \geq  \alpha \} . 
 \end{equation*}
 One has:
\begin{enumerate}
\item  For every $ \alpha  \le \tau_\mu'(0^-)$,    $\dim\, \underline{
E}^{\leq}_\mu( \alpha) \le {\tau_\mu^*}( \alpha).$

\smallskip \item  For every $ \alpha\ge \tau_\mu'(0^+)$,  $\dim\,
\overline{E}^{\geq}_\mu( \alpha) \le   {\tau_\mu^*}( \alpha).$
\end{enumerate}  
\end{proposition}

The distribution of a capacity at small scales can be described through its large deviations spectrum.

\begin{definition}  
Let $\mu\in \mathcal{C}(\zu^d)$ with ${\rm supp}(\mu)\neq\emptyset$. 
For $I\subset \R$ and  $j\in\N^*$ define
$$
\mathcal{D}_\mu(j,I) =
  \left\{ \lambda\subset [0,1]^d,\, \lambda\in\mathcal D_j: \frac{\log_2 \mu(\lambda)}{{-j}} \in
I \right\}.
$$

The lower and upper large deviations spectra of $\mu$ are defined  respectively as
\begin{eqnarray*} 
\underline {\sigma} ^{\mathrm{LD}}_\mu:\alpha\in\R&\mapsto &  \lim_{\ep\to 0}\liminf_{j\to \infty} \frac{\log_2 \#\mathcal{D}_\mu(j, \alpha\pm\ep) }{j}  \\
\mbox{ and } \ \  \   \overline{\sigma} ^{\mathrm{LD}}_\mu:\alpha\in\R & \mapsto& \lim_{\ep\to 0}\limsup_{j\to + \infty} \frac{\log_2 \#\mathcal{D}_\mu(j, \alpha\pm\ep)}{j}.
\end{eqnarray*}
  \end{definition}

\begin{proposition}
\label{fm}   
Let $\mu \in \mathcal{C}(\zu^d)$ with $\supp(\mu)\neq\emptyset$, such that  $\mu$   obeys the SMF (Definition \ref{def_formalism_measure}). One has  $\mathrm{dom}(\tau_\mu^*) = \{\alpha\in\R : \tau_\mu^*(\alpha)\geq 0\}$, and:
\begin{enumerate}
\smallskip\item  For every $ \alpha \in\R$, one has 
$$\sigma_\mu(\alpha)=\dim  E_\mu(\alpha)=\dim \underline E_\mu(\alpha)=\dim \overline E_\mu(\alpha)= \underline {\sigma} ^{\mathrm{LD}}_\mu( \alpha)=\overline {\sigma} ^{\mathrm{LD}}_\mu( \alpha)=\tau_\mu^*(\alpha).$$ 
 
\item  For every $ \alpha  \le \tau_\mu'(0^-)$,    $\dim\, \underline{
E}^{\leq}_\mu( \alpha) = {\tau_\mu^*}( \alpha).$

\smallskip \item  For every $ \alpha\ge \tau_\mu'(0^+)$,  $\dim\,
\overline{E}^{\geq}_\mu( \alpha) =  {\tau_\mu^*}( \alpha).$

 \sk\item
For every $\eta>0$ and every interval $ I \subset \mathrm{dom}(\tau_\mu^*)$,  there exists $\ep_0>0$ and $J_0\in\N$ such that for every $\ep\in (0,\ep_0)$ and $j\geq J_0$,   for $\widetilde I\in\{I,I\pm\ep\}$, 
$$\left| \frac{ \log_2  \#\mathcal{D}_\mu(j,\widetilde I) }{ j } - 
 \sup_{ \alpha\in I} 
 {\tau_\mu^*}( \alpha) \right| \leq  \eta.  $$
 
 \sk\item
If $\mathrm{dom}(\tau_\mu^*)$ is compact, then $\mathrm{dom}(\tau_\mu^*)=[\tau_\mu'(+\infty),\tau_\mu'(-\infty)]$ and there exists a positive decreasing sequence  $(\ep_j)_{j\geq 0}$ tending to 0 when $j \to +\infty$, such that for all $j\in\N$ and $\lambda\in \mathcal D_j$,
$$ \tau_\mu'(+\infty)- \ep_j \leq \frac{\log_2 \mu(\lambda)}{-j}\leq \tau_\mu'(-\infty)+\ep_j .  $$
\end{enumerate}
\end{proposition}

The   properties listed  in the following remarks will be also used.

\begin{remark}\label{Linearisation} Properties of the   Legendre transform  also  imply that:\begin{itemize}
\item
 if $t_\infty:=(\tau_\mu^*)'(\tau_\mu'(+\infty)^+ )<+\infty$, then $t_\infty=\inf\{t: \tau_\mu'(t)=\tau_\mu'(+\infty)\}$, 
 \item for all $t\ge t_\infty$ one has $\tau_\mu(t)=\tau_\mu'(+\infty)t-\tau_\mu^*(\tau_\mu'(+\infty))$,
 \item
  if $\tau_\mu$ is linear over the interval $[p,+\infty)$, then $\tau_\mu'(+\infty) \le p$. 
\end{itemize}
Similarly, \begin{itemize}
\item
if $t_{-\infty}:=(\tau_\mu^*)'(\tau_\mu'(-\infty)^-)>-\infty$, then $t_{-\infty}=\sup\{t: \tau_\mu'(t)=\tau_\mu'(-\infty)\}$, 
\item
for all $t\le t_{-\infty}$ one has $\tau_\mu(t)=\tau_\mu'(-\infty)t-\tau_\mu^*(\tau_\mu'(-\infty))$,
\item
 if $\tau_\mu$ is linear over the interval $(-\infty,p)$, then $\tau_\mu'(-\infty)\ge p$. 
\end{itemize}
The previous properties hold when   $\mu$ is replaced by the capacity $\mu^s$, for any $s>0$. 
\end{remark}
 
\begin{remark}\label{justif} When $\mu$ is a positive measure,   $\tau_\mu^*(\alpha)=\alpha$ if and only if $\alpha\in[\tau_\mu'(1^+),\tau_\mu'(1^-)]$ \cite{Ngai}. This justifies that when $\mu$ obeys the MF, there must exist $D$ such that $\sigma_\mu(D)=D$. Moreover, it is also clear that if $\mu$ obeys the MF and $\mu$ is fully supported, any $D'\in [\tau_\mu'(0^+),\tau_\mu'(0^-)]$ is such that $\sigma_\mu(D')=\tau_\mu^*(D')=-\tau_\mu(0)=d$. 

\end{remark} 
 
We now prove Theorem~\ref{th-construct-measures} in the case $d=1$.

\subsection{A family of probability vectors associated with $\si\in \mathscr{S}_{1,\mathcal M}$} 
\label{sec-proof-step1}

Fix $ \si \in \mathscr{S}_{1,\mathcal M}$ (recall Definition~\ref{def2.5}). In this section,  a sequence of probability vectors $(p_N)_{N\geq 1}$ (where $p_N \in \zu^{2^N}$),which constitutes the core of the construction of a measure $\mu$   satisfying both $(P)$ and the MF   with $\tau_\mu^*= \si$, is defined.  For this, write $\dom(\si)=[\alpha_{\min},\alpha_{\max}]$. When  $\alpha_{\min}=1=\alpha_{\max}$,    the Lebesgue measure on ${[0,1]}$ yields a solution to the inverse problem studied in this Section~\ref{wqb}. So it is  assumed  that $\alpha_{\min}< \alpha_{\max}$.

Let us start by introducing two parameters $D,D'$ defined as follows:
\begin{itemize}
\item
if  $\sigma(1)=1$, set $D=D'=1$.
\item
if   $\sigma(1)\neq 1$, let $0<D<1<D'$ be such that  $\sigma(D)=D$ and   $\sigma(D')=1$.
\end{itemize}
The presence of      the exponents $D$ and $D' $ is justified  in Remark \ref{justif} above.

\smallskip

Then, fix an integer $N_0 $ large enough so that for all $N\ge N_0$, setting $\ep_N= 2\log_2(N)/N$, there exists a  finite subset $A_N=\{\alpha_{N,i}:i=1,...,2m_N\}$ of $[\alpha_{\min},\alpha_{\max}]$ satisfying:
 \begin{itemize}
\item
$\ep_{N_0} \leq \alpha_{\min}/4$;
\item
$m_N \le  2 N (\alpha_{\max}-\alpha_{\min})$;
 \item  $D,D'\in A_N$;
 \item
 for every $i\in \{1,...,m_N-1\}$, $(4N)^{-1} < \alpha_{N,i+1} - \alpha_{N,i}< N^{-1}$;
 \item the following inclusions hold:
 \begin{equation}\label{AN}
A_N \subset \si ^{-1}\Big (\Big [\frac{1 }{N}+\ep_N ,1\Big ]\Big )\subset \bigcup_{i=1}^{m_N} \Big[\alpha_{N,i}-\frac{1}{N} ,\alpha_{N,i}+\frac{1}{N}\Big];
\end{equation}
\item
for every $i\in \{m_N+1,..., 2m_N\}$, $\alpha_{N,i}=\alpha_{N,2m_N-i+1}$;
\item if $\si (\alpha_{\min})>0$, then   $\alpha_{N,1}=\alpha_{\min}$.
\end{itemize}
The continuity of $\sigma$ is used to get \eqref{AN}, and   when $D\neq D'$ the above conditions impose that $|D-D'| \geq (4N)^{-1}$.


\sk

Denote by $i_N$ (resp. $i_N'$) the index  in $[1,m_N]$   such that $D=\alpha_{N,i_N}$ (resp. $D'=\alpha_{N,i'_N}$). Note that $i_N=i'_N$ if and only if $D=D'=1$.

For each $1\le i\le m_N$  such that $i\not\in\{i_N,i'_N\}$, set 
\begin{equation}\label{RNi}
R_{N,i}=\left\lfloor 2^{N(\si (\alpha_{N,i})-\ep_N)-1}\right \rfloor,
\end{equation}
which implies that     for every $i$, $1\leq R_{N,i}  \leq 2^{N-1} N^{-2} $.

 \sk
 
 When $D=D'$,    $i_N=i'_N$ and   set 
 $$
 R_{N,i_N}=2^{N-1}-\sum_{i=1,\, i\neq i_N}^{m_N} R_{N,i}.
 $$  
 
 \sk 
 
 When  $D<D'$,    $i_N < i'_N$ and in this cas one sets 
 \begin{equation}\label{RNibis}
 R_{N,i_N}=\lfloor 2^{N \si(\alpha_{N,i_N}) -1}\rfloor=\lfloor 2^{ND-1}\rfloor\quad\text{and}\quad  R_{N,i'_N}=2^{N-1}- \sum_{i=1,\, i\neq i'_N}^{m_N} R_{N,i}.
 \end{equation}
 
  In all cases, by construction 
$$
\sum_{i=1,\, i\neq i'_N}^{m_N} R_{N,i}\le m_N 2^{N-1} N^{-2} + \mathbf{1}_{\{D\neq D'\}} 2^{ND-1}=o(2^{N-1}) \quad\text{as $N\to\infty$},
$$ 
since the term $\mathbf{1}_{\{D\neq D'\}} 2^{ND-1} $ appears if and only if $D<1$. This also implies that 
\begin{equation}
\label{minoration1}
R_{N,i'_N} =   2^{N-1 } (1+o(1)) .
\end{equation}
Without restriction, we choose  $N_0$ large enough so  that  
\begin{equation}
\label{defN0}
\mbox{for all $N\ge N_0$, } \ \ \sum_{i=1,\, i\neq i'_N}^{m_N} R_{N,i}\le 2^{N-2}.
\end{equation}
Finally, for  $N\ge N_0$ and $m_N<i\le 2m_N$, set $R_{N,i}=R_{N,2m_N-i+1}$, so that 
$$
\sum_{i=1}^{2m_N} R_{N,i} = 2^N.
$$

\begin{definition}  The collection of exponents $(\beta_{N,i})_{0\le i\le 2^N-1}$ is defined  as follows:  
\begin{equation}
\label{defbeta}
\mbox{for all   $1\le j\le 2m_n$, } \ \ \ \ \beta_{N,i}=\alpha_{N,j}\quad \text{if } \sum_{k=1}^{j-1}R_{N,k}\le  i < \sum_{k=1}^{j}R_{N,k}.
\end{equation}
\end{definition}
In other words,  $(\beta_{N,i})_{0\le i\le 2^N-1}$ is obtained by repeating  $R_{N,1}$ times the value $\alpha_{N,1}$, $R_{N,2}$ times   $\alpha_{N,2}$, and so on, until repeating $R_{N,2m_N} $ times  $\alpha_{N,2m_N}=\alpha_{N,1}$.

\begin{lemma}
Let  $p_N=(p_{N,i}) _{0\le i \le 2^N-1}$ be the probability vector defined by
$$
p_{N,i}= \frac{2^{-N\beta_{N,i}}}{\sum_{j=0}^{2^N-1} 2^{-N\beta_{N,j}}}.
$$
One has $p_{N,0}=p_{N,2^N-1}$, and if $|i-i ' | \leq 1$, then 
\begin{equation}\label{control pNvoisin}
 \frac{p_{N,i}}{p_{N,i'}}\in [2^{-1},2].
 \end{equation}
In addition,  for  $N$ large enough,  
\begin{equation}\label{controlpN}
p_{N,i} 2^{N\beta_{N,i}} = 1+\varepsilon_{N,i},
\end{equation}
where $\varepsilon_{N,i}=O(N^{-1})$ uniformly in $0\le i<2^N$.   \end{lemma}

\begin{proof} By definition,
$$
p_{N,i}2^{N\beta_{N,i}}=  \frac{1}{2\sum_{j=1}^{m_N} 2^{-N\alpha_{N,j}} R_{N,j}}
.$$
In order to estimate $p_{N,i}2^{N\beta_{N,i}}$ uniformly in $i$, recall that $\si\leq \mbox{Id}_\R$, so that using the definition of $R_{N,i}$ and $\ep_N$, one gets
\begin{align*}
\lfloor 2^{ND-1}\rfloor     2^{-ND}&  =  R_{N,i_N} 2^ {- N\alpha_{N,i_N} }   \le  \sum_{i=1}^{m_N} 2^{-N\alpha_{N,i} } R_{N,i}\\
&\le  \sum_{1\le i\neq i_N,i\neq i'_N\le m_N}  \!\!\!\! 2^{N( \si(\alpha_{N,i})-\alpha_{N,i} -\ep_N)}    +  R_{N,i_N} 2^ {- ND} +   \mathbf{1}_{D\neq D'}  R_{N,i_N '} 2^ {- ND' } \\
&\le m_N N^{-2} +2^{-ND}\lfloor 2^{ND-1}\rfloor  +\mathbf{1}_{D\neq D'} 2^{N(1-D')}.
\end{align*}
Also, recall that when $D \neq D'$, $D<1$ and $D'>1$. Consequently,  since $\lfloor 2^{ND-1}\rfloor     2^{-ND} = 1/2+o(1)$,  m the sequence of  inequalities just above give  \eqref{controlpN}. 

The fact that  \eqref{control pNvoisin} holds  when $0\le  i,i' \le 2^N-1$ and $|i-i ' | \leq 1$ follows from  $p_{N,i}/p_{N,i'} = 2^{-N(\beta_{N,i}-\beta_{N,i'})} $ and the fact that $|\beta_{N, i'}-\beta_{N,i}|\le N^{-1}$.

Finally,  $p_{N,0}=p_{N,2^N-1}$ by definition of these parameters.
\end{proof}
 
\subsection{Construction of  the  measure $\mu_\si $ associated with $\si\in \mathscr{S}_{1,\mathcal M}$} \label{constrmu}

A Moran measure $\mu_\si $ is iteratively constructed as      concatenation of pieces of   Bernoulli product measures associated with the probability vectors $(p_N)_{N\ge N_0}$. The sequence $(p_N)_{N\ge N_0}$ has been built so that when $N \to +\infty$,  the singularity spectrum of the Bernoulli product measure associated with $p_N$ pointwise converges   to $\si$. Indeed,  each $p_N$ has been chosen so that, heuristically, there are $2^{N\si(\alpha_{N,i})}$ weights of order $2^{-N\alpha_{N,i}}$ and the $\alpha_{N,i}$ tend to be more or less uniformly distributed in the domain of $\si$. 

\medskip

Further ingredients are introduced:
\begin{itemize}
\item
 For $N\ge N_0$,      an integer $\ell_N\ge N^2$ is fixed, such that $(\ell_N)_{N\geq N_0}$ forms an increasing sequence;
\item   consider the product space 
$$\Sigma=\prod_{N=N_0}^\infty   \{0,\cdots ,2^{N}-1\}^{\ell_N};$$

\item
for $N\geq N_0$, if  $g= \ell+ \sum_{n=N_0}^{N-1} \ell_n  $ with $1\le \ell\le\ell_N$,  a word  of generation (or length) $g$
$$(J_{N_0}, J_{N_0+1},\ldots, J_{N-1},J_{N})\in \Sigma_g:=\Big (\prod_{n=N_0}^{N-1}   \{0,\ldots 2^{n}-1\}^{\ell_n}\Big) \times \{0,\ldots 2^{N}-1\}^{\ell}$$
is also denoted $J_{N_0}\cdot  J_{N_0+1}\cdots J_{N}$; then  the cylinder consisting of those elements in $\Sigma$ with common prefix   $J_{N_0}\cdot  J_{N_0+1}\cdots J_{N}$ is denoted $[J_{N_0}\cdot  J_{N_0+1}\cdots J_{N}]$, and the set of such cylinders  of generation $g$ is denoted $\mathcal C_g$;
\smallskip

 \item
the space $\Sigma$ is endowed with the $\sigma$-field $\mathcal B$ generated by the cylinders.

\end{itemize}

\begin{definition}
The probability measure $\nu_\si $ on $(\Sigma,\mathcal B)$ is defined as follows. For all  $N\geq N_0$, for  all $1\le \ell\le \ell_N$, for $g=\ell+ \sum_{n=N_0}^{N-1} \ell_n$   and  $[J_{N_0}\cdot  J_{N_0+1}\cdots J_{N}]\in \mathcal C_g$,    set
\begin{align}
\label{defnu}
\nu_\si  ([J_{N_0}\cdot  J_{N_0+1}\cdots J_{N}])
=\Big (\prod_{n=N_0}^{N-1} \prod_{k=1}^{\ell_n} p_{n, j_{n,k}}\Big) \prod_{k=1}^{\ell} p_{N,k},
\end{align}
where  :
\begin{itemize} 
\item
for every $n\geq N_0$, for every $i \in \{1, \ldots, \ell_n\}$, $j_{n,i} \in  \{0,\ldots, 2^n-1\}$,
\item 
for $N_0\le n\le N-1$, $J_n=j_{n,1}\cdots j_{n,\ell_n} \in \{0,\cdots ,2^{n}-1\}^{\ell_n}$,
\item
 $J_N=j_{N,1}\cdots j_{N,\ell}\in \{0,\cdots 2^{N}-1\}^{\ell}$.
\end{itemize}
\end{definition}

\begin{remark}
\label{rk35}
Formula \eqref{defnu} could be written 
$$\nu_\si  ([J_{N_0}\cdot  J_{N_0+1}\cdots J_{N}])=\prod_{n=N_0}^N \mu_n(J_n),
$$
where $\mu_n$ is the Bernoulli measure associated with the parameters $p_n=(p_{n,i})_{i=0,...,2^n-1}$.  \end{remark}

It is immediate to check that   \eqref{defnu} is consistent, in the sense that$\nu_\si  (\Sigma)=1$ and  for every integers $g' > g\geq 1$, for every cylinder $J\in \mathcal{C}_g$, $\nu_\si (J) = \sum_{J' \in \mathcal{C}_{g'}, J'\subset J} \nu_\si  (J')$. 


Using \eqref{controlpN}, one sees  that there exists  $C>0$ such that for each $N\ge N_0$ and $(J_n)_{N_0\le n\le N}\in \prod_{n=N_0}^N   \{0,\ldots, 2^{n}-1\}^{\ell_n}$, 
$$
\nu_\si  ([J_{N_0}\cdot  J_{N_0+1}\cdots J_{N}]) \le \prod_{n=N_0}^{N} \big ((1+C/n)2^{-n\alpha_{\min}}\big )^{\ell_n}.
$$
Hence $\nu_\si $ is atomless since the right-hand side tends to $0$ as $N$ tends to infinity. 

\medskip

Every $g\in\N^*$ is decomposed  in a unique way under the form $g=\ell+ \sum_{n=N_0}^{N-1} \ell_n$ with $N\ge N_0$ and $1\le\ell\le\ell_N$ (when $N=N_0$, the sum $\sum_{n=N_0}^{N-1} \ell_n$ is 0). Using this decomposition, one can define the     mapping 
\begin{align}
\label{defgammag}
\gamma :  \begin{cases}  \N^*    \to    \N^*\\
  g   \mapsto    \gamma(g) :=  N \ell   + \sum_{n=N_0}^{N-1} n \ell _n. \end{cases}
\end{align}

The space $\Sigma$ provides a natural coding of $[0,1]$. Indeed, considering the coding map 
\begin{equation}
\label{defpi}
\pi: x=\left((x_{N,k})_{k=1}^{\ell_N}\right )_{N\ge N_0}\in\Sigma\mapsto \sum_{N=N_0}^\infty  2^{ -\sum_{n=N_0}^{N-1}n\ell_n} \sum_{k=1}^{\ell_N} x_{N,k}2^{ -k N} \in \zu,
\end{equation}
   for each $g\in \N^*$, $\pi$ maps bijectively the elements of $\mathcal C_g$ onto the set of closed dyadic subintervals of generation $\gamma(g)$ of $[0,1]$.

\begin{definition}
\label{defM1}
For  every $\si \in \mathscr{S}_{1,\mathcal M}$, consider   the Borel probability  measure    on $\zu $   
$$
\widetilde \mu_\si =\nu_\si \circ\pi^{-1},
$$ 
where $\nu_\si $ is the measure constructed above in \eqref{defnu}. 
Then,   $\mu_\si$  is defined as the natural periodized version of $\widetilde \mu_\si $, i.e. the $\mathbb Z$-invariant measure  
$$
\mu_\sigma:B\in\mathcal B(\R)\mapsto \sum_{k\in\mathbb Z} \mu ((B\cap [k,k+1)) - k).
$$
Finally, set   
$$
\mathcal{M}_1=\{\mu_\sigma:\si \in \mathscr{S}_{1,\mathcal M}\}\subset \mathcal M(\R).
$$
The measures $\mu_\si $ and $\widetilde \mu_\si$ are said to be  associated with $\si \in \mathscr{S}_{1,\mathcal M}$.
\end{definition}

\begin{proposition}\label{propmu-p}
Every $\mu\in \mathcal{M}_1$  satisfies the property (P) of  Definition~\ref{mildcond1}.

Moreover, if $\mu$ is associated with $\si\in \mathscr{S}_{1,\mathcal M}$, then $\mu_{|[0,1]}$ has $\si$ as multifractal spectrum, and it   obeys the SMF on $\R_+$.
\end{proposition}

Observe that since $\nu_\si $ is atomless and $\pi$ is $1$-to-$1$ outside a countable set of points, for any closed dyadic subinterval $\la$ of $[0,1]$ of generation $n\in\gamma(\N^*)$, one has  $\mu_\si (\la)=\nu_\si  ([w])$, where $[w]$ is the unique  cylinder of generation $\gamma^{-1}(n)$ such that $\pi([w])=\la$.  

 Next sections are devoted to the proofs of the various properties of $\mu_\si $, which, in particular, yield  Proposition~\ref{propmu-p}.

\mk

For the rest of this section,   $\si \in  \mathscr{S}_{1,\mathcal M} $ is fixed,  and we simply call $\mu $ and $\nu$ the measure  $\mu_\si$ and $\nu_\si $ associated with $\si$.

\subsection{The measure $\mu $ satisfies property (P)} \label{musatP}

\begin{lemma} 
\label{lemdoubling}
The measure $\mu$ is almost doubling.
\end{lemma}
 
 \begin{proof}
 
Let $g\in\N^*$ and write it under the form  $g=\ell+ \sum_{n=N_0}^{N-1} \ell_n\in \N$ with $N\geq N_0$  and $1\le\ell\le\ell_N$.

First,    note that if $g$, hence $N$, is   large enough, the term $1+\ep_{N,i}$ in \eqref{controlpN} is greater than 1/2  and smaller than 3/2. Hence, for any $1\le i\le 2m_N$,
\begin{equation}
\label{controlpN2}
2^{-N(\alpha_{\max}+ \widetilde \ep_N )} \leq p_{N,i} \leq 2^{-N(\alpha_{\min}-\widetilde\ep_N)},
\end{equation} 
where $(\widetilde\ep_N)_{N\geq 1}$ is a non-increasing sequence (independent of $i$)  converging to 0.

\mk

We start by dealing with the dyadic intervals of generation~$\gamma(g)$.

Consider two   closed dyadic subintervals $\la $ and $ \lah$ of $[0,1]$ of generation $\gamma (g)$   such that $\la$ is the left neighbor of $\lah$. By construction,  $\la$ and $\lah $ are the images by $\pi$ of two cylinders $[J]$ and $[\widetilde J]$ in $\mathcal C_g$ such that, denoting by $u$ the longest common prefix of the words $J$ and $\widetilde J$, there exist $N_1\ge N_0$ and $0\le j<2^{N_1}-2$ such that $J=u\cdot j \cdot v$ and $\widetilde J = u\cdot (j+1)\cdot  \widetilde v$, where:
\begin{enumerate}
\item
there is $ 1\leq \tilde \ell_1 \leq \ell_{N_1}$ such that 
$u\in  \prod_{N=N_0}^{N_1-1}  \{0,\cdots ,2^{n}-1\}^{\ell_n} \cdot \{0,\cdots ,2^{N_1}-1\}^{\tilde \ell_{1}}$; 
\item
$v$ and $\tilde v$ belong to $   \{0,\cdots ,2^{N_1}-1\}^{\ell_{N_1}-1-\tilde \ell_{1}} \cdot  \prod_{N=N_1}^{N -1}  \{0,\cdots ,2^{n}-1\}^{\ell_n} \cdot     \{0,\cdots ,2^{N}-1\}^{\ell }$, and :
\begin{enumerate}
\item
 either $v$ and $\widetilde  v$ are empty words, 
 \item
 or   all letters of $\widetilde v$ are 0 and all letters of $v$ are as large as possible.
  \end{enumerate} \end{enumerate}
  
  From \eqref{control pNvoisin}, \eqref{defnu} and  the fact that  $p_{n,0}=p_{n,2^n-1}$ for every $n\geq N_0$,  one deduces that 
\begin{equation} 
\label{doubling}
2^{-1}\le \frac{\mu (\la)}{\mu(\lah )}\le 2.
\end{equation}

Consider now two neighboring  intervals $\la$ and $\lah$   of generation $j$, where   $\gamma(g)<  j\le\gamma(g+1)$. Let $  \la'$ and $  \lah'$ be the elements of $\mathcal D_{\gamma (g)}$ which contain $\la$ and $\lah$ respectively. These intervals are either equal or neighbors. By construction, when  $N$ is large enough, one has by \eqref{control pNvoisin} and \eqref{defnu}
\begin{equation}
\label{encadre1}
 2^{- N(\alpha_{\max}+\ep )} \leq \frac{\mu(\la)}{\mu(  \la' )} \leq 1,
 \end{equation}
 and if $j=\gamma(g+1)$
 \begin{equation}
\label{encadre1'}
 2^{- N(\alpha_{\max}+\ep )} \leq \frac{\mu(\la)}{\mu(  \la' )} \leq 2^{- N(\alpha_{\min}-\ep )},
 \end{equation}
 where $\ep \leq \alpha_{\min}/2$ for instance (this is due to the choice of  $N_0$ with $\ep_{N_0}\leq \alpha_{\min}/4$).
 
 \begin{remark}
 \label{remg}
 Observe that by construction, when $g$ gets larger, the $\ep$ in \eqref{encadre1'} can be taken very small and converges to zero when $g\to+\infty$, again because of the construction of $p_N$. This remark is used in Remark~\ref{rem-philambda} below.
 \end{remark}
 
 The same property as \eqref{encadre1} holds true for $\lah$ and $\lah'$, hence 
\begin{equation}\label{comparison}
 \frac{\mu(  \la')}{\mu(\lah')} 2^{-(\alpha_{\max}+\ep )  N}\le \frac{\mu(\la)}{\mu(\lah)}= \frac{\mu(\la)}{\mu(  \la')} \frac{\mu(  \la')}{\mu(\lah')}  \frac{\mu(\lah')}{\mu(\lah)}\le  \frac{\mu(  \la' )}{\mu(\lah')} 2^{ (\alpha_{\max}+\ep) N}.
\end{equation}
Let 
\begin{equation}
\label{theta1}
\phi (j)=  \begin{cases} 0 \ \mbox{  if }0\le j<N_0\ell_{N_0}\\ (1+(\alpha_{\max} +\ep)  )N \ \mbox{ if }\gamma(g)\le j<\gamma(g+1). \end{cases}
\end{equation}
Note that  $\phi (j)/j\le (1+(\alpha_{\max} +\ep) )N/\gamma(g)$ which tends to 0 as $j\to\infty$.  This follows from the fact that 
\begin{equation}
\label{gtoinfty}
\gamma(g) \geq \sum_{n=N_0}^{N-1}n\ell _{n}    >\!\!>   N^2
\end{equation}
  as $N\to\infty$ since $\ell_n \ge  n^2$ for all $n\ge N_0$.   Hence $\phi\in \Phi$.

Applying \eqref{doubling} 
 to $  \la'$ and $\lah'$, and using \eqref{comparison}, there exists $\tilde J$ such that for $j\geq \tilde J$, 
  \begin{equation}
 \label{doubling2}
  2^{- \phi (j)} \le \frac{\mu(\la)}{\mu(\lah)} \leq 2^{ \phi (j)},
  \end{equation}

Upon adding a constant to $\phi$ (to take into account the small generations $j \leq \tilde J$), one  concludes that $\mu_{|\zu}$ is almost doubling in the sense of Definition~\ref{defAD}.

To prove that $\mu$ is almost doubling on $\R$,  it is enough to observe that   by symmetry of the coefficients ($p_{n,0}=p_{n,2^n-1}$), for any $g\in \N$   $\mu_{|[1-2^{-\gamma(g)},1]}(\cdot +1-2^{-\gamma(g)})=\mu_{|[0,2^{-\gamma(g)}]}$, and then to use   the periodicity of $\mu$.
  \end{proof}

\begin{lemma}
\label{propertyP}
The measure $\mu$  satisfies (P).
\end{lemma}

 \begin{proof} First, consider subintervals of   $[0,1]$. 
  
Let $\ep>0$.  For $N\geq N_0$ and $g=\ell+ \sum_{n=N_0}^{N-1} \ell_n$ with $1\le\ell\le\ell_N$, any dyadic interval $\la \in \mathcal D_j$ with $\gamma(g)\le j<\gamma(g+1)$ satisfies, if $N$ is large enough 
$$
2^{-(\gamma(g)+N)(\alpha_{\max}+\ep/2) }\le  \mu(\la )\le 2^{-\gamma(g)(\alpha_{\min}-\ep/2)},
$$
(use \eqref{controlpN2} for instance). By our choice for $\ell_N$ and \eqref{gtoinfty},  for $\gamma(g)\le j<\gamma(g+1)$, $N/j$ converges to 0 as $j\to+ \infty$. Hence, for $j$ large enough
\begin{equation}
\label{encadr2}
2^{-j(\alpha_{\max}+\ep) }\le  \mu(\la )\le 2^{-j(\alpha_{\min}-\ep)}.
\end{equation}
So, \eqref{minmaj1} is satisfied with $s_2= \alpha_{\max}+\ep$ and   $s_1 = \alpha_{\min}-\ep$, and some constant $C>0$. This yields property (P$_1$).

Let us move to (P$_2$).

Let $g,g'\in\N^*$, $j,j'\in\N^*$ with $j'> j$ and  $N'\geq N\ge N_0$ such that:
\begin{itemize}
\item
  $g=\ell+\sum_{n=N_0}^{N-1}\ell_n$ with $1\le \ell\le \ell_N$, and  
  $\gamma(g)\le j<\gamma(g+1)$,
\item
   $g'=\ell'+\sum_{n=N_0}^{N'-1}\ell_n$  with     $1\le \ell'\le \ell_{N'}$, and     $\gamma(g')\le j'<\gamma(g'+1)$.
\end{itemize}

Consider two neighboring dyadic intervals $\la, \lah \in\mathcal D_j$,  and  an interval $\la'\in\mathcal D_{j'}$ such that $\la'\subset \la$.

Due to the doubling property of $\mu$ applied to $\la$ and $\lah$, we have 
\begin{equation}\label{aaa}
2^{-\phi(j)}\frac{\mu(\la)}{\mu(\la')} \le \frac{\mu(\lah)}{\mu(\la')}  = \frac{\mu(\lah)}{\mu(\la)}   \frac{\mu(  \la)}{\mu(\la')}     \le 2^{\phi(j)} \frac{\mu(\la)}{\mu(\la')}.
\end{equation}
For $J\le j$, denote by $\lambda_{|J}$ the unique element of $\mathcal D_J$ which contains $\lambda$, and for $j<J\le j'$ denote by $\lambda_{|J}$ the unique element $\widetilde \lambda$ of $\mathcal D_J$ such that $\lambda'\subset \widetilde \lambda\subset \lambda$. We have 
$$
 \frac{\mu(\la_{|\gamma(g)+N})}{\mu(\la'_{|\gamma(g')})} \le \frac{\mu(\la)}{\mu(\la')}\le  \frac{\mu(\la_{|\gamma(g)})}{\mu(\la'_{|\gamma(g')+N'})}.
 $$
It is easily seen that $N+N'=o(j)+o(j'-j)$ as $j,j'\to + \infty$. Consequently, using the multiplicative structure of $\mu$ and \eqref{controlpN2} yields  a function $\widetilde\phi\in\Phi$, as well as a constant $C\ge 1$,  depending on $\mu$ only, such that 
\begin{align}\label{bbb}
C^{-1}2^{-j\widetilde \phi(j)} 2^{(j'-j) (\alpha_{\min} -\ep) } \le \frac{\mu(\la)}{\mu(\la')}\le C 2^{j\widetilde \phi(j)}2^{(j'-j) (\alpha_{\max}+\ep)}.
\end{align}
Incorporating \eqref{bbb} in  \eqref{aaa} shows that (P$_2$) holds  with the same exponents $s_1$ and $s_2$ as in (P$_1$). 

\medskip

Finally, the same arguments as  those developed at the end of the proof of Lemma \ref{lemdoubling} ensure that the property  true on $\zu$ extends to $\R$. 
\end{proof}

\begin{remark}
\label{rem-philambda}
 
The previous estimates  and Remark \ref{remg}  show that for every $\ep>0$, there exists $j_\ep\in\N$ such that  for all $j'\ge j\ge j_\ep $, for all  $\lambda, \widetilde \lambda\in\mathcal D_j$  such that $\partial\lambda\cap \partial  \widetilde \lambda\neq\emptyset$,  and all $\lambda'\in\mathcal D_{j'}$ such that $\lambda'\subset \lambda$, one has 
\begin{equation}\label{loulou}
\mu(\lambda')\le \mu(\widetilde \lambda) 2^{j\ep}2^{-(j'-j)(\alpha_{\min}-\ep)}.
\end{equation} 
 
Also, from the  construction of $\mu$, for all integers $j,j'\ge 0$ and $\lambda\in\mathcal D_j$, one has  
\begin{equation}\label{lala}
\mu( \lambda\cdot [0,2^{-j'}]^d )=\mu(\la) 2^{-\phi_\lambda } 2^{-j'\alpha_{\min}+ \tilde\phi_{\la}( j')},
\end{equation}
where:
\begin{itemize}
\item
   $\lambda \cdot [0,2^{-j'}]^d $ is the concatenation of $\la$ and $[0,2^{-j'}]^d $, meaning that $\lambda \cdot [0,2^{-j'}]^d$ is  the image of  $[0,2^{- j'}]^d $ by the canonical similarity which maps $[0,1]^d$ onto  $ \la$, 

\item
  $\phi_\la \in \R$ and  $ \tilde\phi_\la\in \Phi$ are uniform $o(j) $  in the sense that  
\begin{equation}
\label{majphiphitilde}
\lim_{j\to+\infty} \sup \left\{ \frac{ | \phi_\la  |}{j} :  \lambda\in \mathcal D_j\ \right\} =  \lim_{j'\to+\infty} \sup \left\{  \frac{|\tilde\phi_{ \la}(j')| }{j'}:  \lambda\in \bigcup_{j\in\N}\mathcal D_j \right\}=0 .
\end{equation}
\end{itemize}

These inequalities are key to prove the optimal  upper bound  for the   singularity spectrum  of typical functions in $\widetilde B^{\mu,p}_q(\R^d)$. 
\end{remark}


\subsection{The $L^q$-spectrum  of $\mu_{|[0,1]}$ equals $\si^*$.}\label{sectaumu} Let $\tau=\si^*$. Since $\si\in \mathscr{S}_{1,\mathcal{M}} $,   $\tau \in \mathscr{T}_{1,\mathcal{M}} $. 

For simplification, denote $\mu_{|[0,1]}$ by $\mu$. For all $j\in\N$, let 
$$
\mathcal D_j^0=\{\lambda\in\mathcal D_j: \lambda\subset [0,1]^d\}.
$$ 

Fix  $t\in \R$ and $g=\ell+ \sum_{n=N_0}^{N-1} \ell_n$  with $N\ge N_0$ and  $1\le \ell\le \ell_{N}$.   Assume that $g$ is so large that \eqref{encadre1} holds for every $j\geq \gamma(g)$.

\medskip

First,  remark that, for the integers $j$ such that  $ \gamma(g) \le  j < \gamma(g+1)$,  for every $ \la \in \mathcal D^0_{\gamma(g)}$, by \eqref{encadre1} one has
$$ 2^{(j-\gamma(g))}2^{- N|t|(\alpha_{\max}+\ep)} \leq \sum_{ \la '\in \mathcal D_{j} ,  \la ' \subset  \la } \frac{\mu( \la ')^t}{\mu(\lambda)^{t}} \leq 2^{(j-\gamma(g))}2^{ N|t|(\alpha_{\max}+\ep)}.$$
Since  $N+(j-\gamma (g))=o(\gamma (g))$ as $g\to+\infty$, one deduces that  
\begin{equation}\label{fromgammagton}
 { \sum_{ \la \in\mathcal D^0_{j}}\mu( \la )^t} = 2^{o(\gamma(g))} {\sum_{ \la \in\mathcal D^0_{\gamma (g)}}\mu(\la)^t}  . 
 \end{equation}

This shows that it is enough to study  $  \liminf_{g\to+\infty} \frac{1}{-\gamma(g)} \log_2 \sum_{I\in\mathcal D^0_{\gamma (g)}}\mu(I)^t$ to find the value  $\tau_\mu(t)$ (actually, $\tau_\mu$ will be proved to be a limit, not only a liminf).

\medskip
  
 $\bullet$ Let us start with the lower bound for $\tau_\mu(t)$.

The multiplicative structure defining $\nu$ and $\mu$ using  concatenation of pieces of Bernoulli product measures yields   
\begin{equation}
\label{decomp1}
\sum_{ \la \in\mathcal D^0_{\gamma (g)}}\mu( \la )^t =\Big (\prod_{n=N_0}^{N-1}\Big (\sum_{i=0}^{2^n-1}p_{n,i}^ t \Big )^{\ell_n}\Big )\cdot \Big (\sum_{i=0}^{2^N-1}p_{N,i}^t \Big )^\ell.
\end{equation}
For each $n\ge N_0$, using \eqref{controlpN}, one has 
\begin{equation}
\label{defcnt}
C_{n,t}^{-1} 2^{-nt\beta_{n,i}} \le p_{n,i} ^t \leq  2^{-nt\beta_{n,i}} C_{n,t},
\end{equation}
 where $C_{n,t}$ tends to 1   when $n\to +\infty$ (and does not depend on $i\in \{0,...,2^n-1\}$). Hence, using \eqref{defbeta}, the definition of the $R_{n,i}$ and the inequality $2R_{n,i}\le 2^{n\sigma(\alpha_{n,i})}$ which follows from \eqref{RNi},  one gets   
\begin{align*}
\sum_{i=0}^{2^n-1}p_{n,i}^t & \le     C_{n,t}\sum_{i=0}^{2^n-1}2^{-tn \beta_{n,i}}  \le     C_{n,t} \sum_{i=0}^{m_n} 2 R_{n,i} 2^{-tn \alpha_{n,i}}   \le   C_{n,t}\sum_{i=1}^{m_n} 2^{n(\si(\alpha_{n,i})-t\alpha_{n,i})}\\
& \le  C_{n,t}m_n 2^{-n\inf\{t\alpha-\si (\alpha):\, \alpha\in \mathrm{dom}(\sigma)\}} =  C_{n,t}m_n 2^{-n\tau(t) } .
\end{align*}
 Consequently,  
$$
\sum_{ \la  \in\mathcal D^0_{\gamma (g)}}\mu( \la )^t \le 2^{-\gamma (g)  \tau (t)  } \cdot \Big (\prod_{n=N_0}^{N-1} (C_{n,t}m_n)^{\ell_n}\Big )\cdot (C_{N,t}m_N)^\ell.
$$

Since $\log(m_n)= o(n)$, one sees that  $\ell\log(m_N)+\sum_{n=N_0}^{N-1}  \ell_n \log m_n  = o(\gamma(g))$. Combining this with the fact that $(C_{n,t})_{n\ge N_0}$ converges to $1$ when $n$ tends to infinity, one deduces that   $\Big (\prod_{n=N_0}^{N-1} (C_{n,t}m_n)^{\ell_n}\Big )\cdot (C_{N,t}m_N)^\ell = 2^{o(\gamma(g))}$ and 
$$ \tau_\mu(t) =  \liminf_{g\to + \infty} \frac{-1}{\gamma(g)} \log_2\sum_{ \la  \in\mathcal D^0_{\gamma (g)}}\mu( \la  )^t  \geq \tau(t) .$$

 $\bullet$  Let us now estimate $\limsup_{g\to + \infty} \frac{-1}{\gamma(g)} \log_2\sum_{ \la  \in\mathcal D^0_{\gamma (g)}}$. 

Suppose first  that $\si(\tau'(t^+))>0$. By construction, one can fix    $N'_0\ge N_0$ such that for all $n\ge N'_0$, there exists  an integer $1\le i_{n,t} \le m_n$,  such that   $|\alpha_{n,i_{n,t}} -\tau'(t^+)|\le 1/n$,  $ i_{n,t}\neq i_n$ and $ i_{n,t}\neq i'_n$. The  Legendre transform   $\sigma=\tau^*$ implies then    that  $t\tau'(t^+)-\tau(t)=\sigma(\tau'(t^+))$. 

In addition, by  continuity of $\sigma$,    $\lim_{n\to + \infty}\eta_n=0$, where  $\eta_n= \si (\alpha_{n,i_{n,t}})-t\alpha_{n,i_{n,t}}+\tau(t)$. Bounding  from below the sums in \eqref{decomp1} by the sum only over those integers $j$ such that $\beta_{n,j}=\alpha_{n,i_{n,t}}$ (see \eqref{defbeta}), and recalling \eqref{defcnt} and  the definition  \eqref{RNi} of $R_{n,i}$,  
\begin{align*}
\sum_{ \la  \in\mathcal D^0_{\gamma (g)}}\mu( \la )^t&\ge \Big (\prod_{n=N_0}^{N_0'-1}\Big (\sum_{i=0}^{2^n-1}p_{n,i}^ t \Big )^{\ell_n}\Big )\cdot \Big (\prod_{n=N'_0}^{N-1} \Big (C_{n,t}^{-1}\lfloor 2^{n(\si (\alpha_{n,i_{n,t}}) -\ep_n)}\rfloor 2^{- t n \alpha_{n,i_{n,t}} }\Big )^{\ell_n} \Big ) \\
&\qquad\qquad\qquad \cdot \Big (C_{N,t}^{-1}\lfloor 2^{N(\si (\alpha_{N,i_{N,t}} )-\ep_N)}\rfloor 2^{-tN \alpha_{N,i_{N,t}}}\Big )^{\ell}.
\end{align*}
Recalling  that $\ep_n=\frac{2\log _2 n}{n}$,  and setting $C_t=\prod_{n=N_0}^{N_0'-1}\Big (\sum_{i=0}^{2^n-1}p_{n,i}^ t \Big )^{\ell_n}$, one obtains \begin{align*}
\sum_{ \la \in\mathcal D^0_{\gamma (g)}}\mu( \la )^t&\ge C_t \Big (\prod_{n=N'_0}^{N-1} \Big (C_{n,t}^{-1}\frac{2^{n(\si (\alpha_{n,i_{n,t}})  -t  \alpha_{n,i_{n,t}} )  }}{4n^2}\Big )^{\ell_n} \Big)  \cdot \Big (C_{N,t}^{-1}\frac{2^{N  (\si (\alpha_{N,i_{N,t}})  -t  \alpha_{N,i_{N,t}} )    }}{4N^2}\Big )^{\ell}\\
&= C_t 2^{-\gamma (g) \tau (t)}  \Big (\prod_{n=N'_0}^{N-1} \Big (C_{n,t}^{-1}\frac{2^{n\eta_n}}{4n^2}\Big )^{\ell_n} \Big)  \cdot \Big (C_{N,t}^{-1}\frac{2^{N\eta_N}}{4N^2}\Big )^{\ell}\\
&= 2^{-\gamma (g) (\tau (t)+o(1))}
\end{align*}
as $g\to+ \infty$, where we used that  $\log(C_{n,t})+n\eta_n+\log (4n^2)=o(n)$ (recall that $C_{n,t}\to 1$ when $n\to +\infty$ uniformly in $t$).
The  last lines imply that 
$$
\limsup_{g\to + \infty} \frac{-1}{\gamma(g)}\log_2 \sum_{ \la \in\mathcal D^0_{\gamma(g)}}\mu( \la )^t \le  \tau(t)   .
$$

This equation and the lower bound already obtained  for $\tau_\mu(t)$  show that $\tau_\mu(t) = \tau(t)$.

\medskip

It remains us to consider the extremal case $\si(\tau'(t^+))=0$, which may happen only if $\tau'(t^+)\in\{\alpha_{\min},\alpha_{\max}\}$. 

Suppose that $\tau'(t^+)=\alpha_{\min}$ and $\si(\alpha_{\min})=0$. One has $0=\sigma(\alpha_{\min})=\tau^*(\alpha_{\min})=t^+\tau'(t^+)-\tau(t)$, so  $\tau(t) = t\alpha_{\min}$, and $t_0=\min\{t\in\R: \tau(t)=\alpha_{\min} t\}<+\infty$. In addition, $t_0>0$ since $\tau(0) <0$. Also,   for $t\in [0,t_0)$,   $\si (\tau'(t^+))\in (0,1]$ and we know from the first part of this proof that $\tau_\mu(t)=\tau(t)$ on this interval $[0,t_0)$. To conclude, it is thus enough to show that this last  equality holds over the whole interval $[t_0,+\infty)$ as well. 

At first, for all $t\ge t_0$, $\ep\in (0,t_0)$ and $n\in\N$, by subadditivity of $x\ge 0\mapsto x^{t/(t_0-\ep)}$,  
$$
\sum_{ \la \in\mathcal D^0_{ \gamma(g) }}\mu( \la )^t \le\Big ( \sum_{ \la \in\mathcal D^0_{\gamma(g)}}\mu( \la )^{t_0-\ep}\Big )^{t/(t_0-\ep)},
$$
so
\begin{equation}\label{tauqplus}
\tau _\mu(t) = \liminf_{g\to\infty} -\frac{1}{\gamma(g)}\log_2 \sum_{ \la \in\mathcal D^0_{\gamma(g)}}\mu( \la )^t \ge \frac{t}{t_0-\ep}\tau(t_0-\ep).
\end{equation}
On the other hand, consider   the interval $ [0,2^{-\gamma(g)}]$ in $\mathcal D_{\gamma (g)}$. Its $\mu$-mass is by construction $2^{-\gamma(g)(\alpha_{\min}+o(1))}$ as $g\to +\infty$, so 
$$
\limsup_{g\to + \infty} -\frac{1}{\gamma(g)}\log_2 \sum_{ \la \in\mathcal D^0_{\gamma(g)}}\mu( \la )^t \le \limsup_{g\to + \infty} -\frac{1}{\gamma(g)  }\log_2 2^{-t\gamma(g)(\alpha_{\min}+o(1))}  =  \alpha_{\min }t. 
$$
Letting $\ep\to 0$  in \eqref{tauqplus}, and using that $\alpha_{\min}= \tau(t_0)/ t_0$, one concludes that  $
\tau_\mu(t)=\alpha_{\min }t=\tau(t).$

\medskip

The case $\tau'(t^+)=\alpha_{\max}$ and $\si(\alpha_{\max})=0$ works similarly by considering $t_0=\max\{t\in\R: \tau(t)=\alpha_{\max} t\}\in (-\infty,0)$, and the element of $\mathcal D^0_{\gamma (g)}$ whose $\mu$-mass is minimal, i.e. equal to $2^{-\gamma(g)(\alpha_{\max}+o(1))}$.

\subsection{The  SMF holds  for $\mu$ with $\sigma_\mu=\si$.} 
\label{muspectrum}

The facts that $E_\mu(\alpha)=\emptyset$ for   $\alpha\not\in [\alpha_{\min},\alpha_{\max}]$ and    $\dim  E_\mu(\alpha)\le \si(\alpha)$ for $\alpha \in [\alpha_{\min},\alpha_{\max}]$, follow  from Proposition \ref{fm00}   and  Section \ref{sectaumu}  where it is  proved that $\tau_\mu^*=\si$ (so $\tau_\mu^* (\alpha)=-\infty$ if $\alpha \notin [\alpha_{\min}, \alpha_{\max}]$).

\medskip

 Further, it follows from the construction and the choice of the weights $p_{n,i}$ that there exist  real numbers  $x$ at which $h_\mu(x)=\alpha_{\min} $,  and other 
real numbers  $x$ at which $h_\mu(x)=\alpha_{\max} $. Hence, $ \sigma_\mu(\alpha_{\min}) \ge 0$ and $ \sigma_\mu(\alpha_{\max}) \ge 0$.
  
  In particular, if $\si(\alpha_{\min}) =0$ ({\em resp.} $\si(\alpha_{\max}) =0$), then $ \sigma_\mu(\alpha_{\min}) = 0$ ({\em resp.} $ \sigma_\mu(\alpha_{\max}) =0$) and  the SMF holds   at $\alpha_{\min}$  ({\em resp.} $\alpha_{\max}$).

\medskip

Now, fix $\alpha\in [\alpha_{\min},\alpha_{\max}]$ such that $\sigma(\alpha)>0$.  For each $N\ge N_0$, let
\begin{equation}
\label{defjalpha}
\mathcal J_{N,\alpha} = \left\{ j\in\{0,\ldots 2^N-1\} :  \mbox{ $j$ is odd and } \ |\beta_{N,j} -\alpha|\leq N^{-1}\right \}.
\end{equation}

Let $\ep>0$. Recalling the definitions of Section \ref{sec-proof-step1} we first observe that  the exponents  $\beta_{N,j}$ considered in the definition of $\mathcal J_{N,\alpha}$ correspond to at most nine distinct exponents $\alpha_{N,i}$ (since $\alpha_{N,i}-\alpha_{N,i-1}\leq (4N)^{-1}$). This observation, together with the continuity of $\sigma$ and the definition of the numbers $R_{N,i}$ imply that for $N$ large enough,
  \begin{equation}
\label{cardinalj}
 2^{N(\si (\alpha)-\ep)}   \leq \# \mathcal J_{N,\alpha}  \leq   2^{N(\si (\alpha)+\ep)} .
\end{equation}

Consider the measure $\nu_\alpha$ supported on 
\begin{equation*}
\label{defsigmaalpha}
\Sigma_\alpha=\prod_{n=N_0}^\infty\mathcal  J_{n,\alpha}^{\ell_n}  \ \ \subset \  \Sigma 
\end{equation*}
 defined by setting, for each $N\ge N_0$, $0\le \ell<\ell_N$ and  for every word $J_{N_0}\cdot  J_{N_0+1}\cdots J_{N}\in \Big (\prod_{n=N_0}^{N-1}   \{0,\cdots 2^{n}-1\}^{\ell_n}\Big) \times \{0,\cdots 2^{N}-1\}^{\ell}$:
$$
\nu_\alpha([J_{N_0}\cdots J_{N}])= \begin{cases}  ( \# \mathcal J_{N,\alpha} )^{-\ell}\prod_{n=N_0}^{N-1}( \# \mathcal J_{n,\alpha} )^{-\ell_n}&\text{if } [J_{N_0}\cdots J_{N}]\cap \Sigma_\alpha\neq\emptyset,\\
0&\text{otherwise}.
\end{cases}
$$
 One  easily checks that this last formula is consistent, and the measure $\nu_\alpha$  is well-defined and  atomless.  

\begin{proposition}
\label{propmualpha}
The measure $\mu_\alpha=\nu_\alpha\circ\pi^{-1}$ is  defined as the push-forward measure of $\nu_\alpha$ on the interval $\zu$ (recall \eqref{defpi}). This measure  is supported by $\pi(\Sigma_\alpha)$, and for every $x\in \pi(\Sigma_\alpha)$, $h_\mu(x) = \alpha$ and $h_{\mu_\alpha}(x) = \sigma(\alpha)$.
\end{proposition}

\begin{proof}
For all $\omega\in\Sigma$, denote by $[\omega_{|g} ]$ the cylinder of generation $g\in\N$ which contains $\omega$. From the definition of $ \mathcal J_{N,\alpha}$, for every $\omega \in \Sigma_\alpha$ one has 
 $$
\alpha-\ep \leq  \liminf_{g\to + \infty} -\frac{1}{\gamma(g) }\log\big ( \mu(\pi([\omega_{|g}]))\big ) \leq  \limsup_{g\to +\infty} -\frac{1}{\gamma(g)}\log\big ( \mu(\pi([\omega_{|g}]))\big )\leq  \alpha+\ep.$$

 Since this holds for every choice of $\ep>0$,  
 $$
 \lim_{g\to+ \infty} -\frac{1}{\gamma(g) } \log\big ( \mu(\pi([\omega_{|g}]))\big )=\alpha. $$

Moreover, $\lim_{g\to= \infty} \frac{\gamma(g+1)}{\gamma (g)}=1$ and $\mu$ is almost doubling, so $\pi(\Sigma_\alpha)\subset E_\mu(\alpha)$.

On the other hand,  from \eqref{cardinalj} one deduces that for every $\omega \in \Sigma_\alpha$
$$ \si(\alpha)- \ep\leq  \liminf_{g\to+ \infty}  \frac{-1}{\gamma(g)}\log\big ( \mu_\alpha(\pi([\omega_{|g}]))\big ) \leq  \limsup_{g\to + \infty} \frac{-1}{\gamma(g)}\log\big ( \mu_\alpha(\pi([\omega_{|g}]))\big ) \leq \si(\alpha)+\ep.
 $$
 
 Again,  this holds for every choice of $\ep>0$, hence 
 $$\lim_{g\to +\infty} -\frac{1}{\gamma(g) } \log\big ( \mu_\alpha(\pi([\omega_{|g}]))\big )=\si(\alpha).
 $$
 Since $\lim_{g\to+  \infty} \frac{\gamma(g+1)}{\gamma (g)}=1$, the measure $\mu_\alpha$, which is supported by $\pi(\Sigma_\alpha)$, is exact dimensional with dimension $\si(\alpha)$, so $\dim( \Sigma_\alpha) \geq \si(\alpha)$.
 
 The combination of the last two facts imply that $\sigma_\mu(\alpha) = \dim E_\mu(\alpha)\ge \si(\alpha)$. Since the converse inequality holds true by the \ml formalism, the proof is complete.  
\end{proof}

\subsection{The case $d\ge 2$} 
\label{dimensiond}
If $\si\in \SD $, then the map  $  \widetilde \si  : \alpha\in\R \mapsto d^{-1} \si  (d\cdot \alpha)$  belongs to $\mathscr{S}_{1,\mathcal M}$. Let   $ \widetilde\mu_{\widetilde\sigma}$ be the measure associated with $\widetilde\si$  as built  in the previous sections in dimension 1. Then, it is easily checked that the tensor product measure  $\mu= ( \widetilde\mu _{\widetilde\sigma} ) ^{\otimes d}$ possesses all the required properties.

In addition, for all $\alpha\in\mathrm{dom}(\si )$, if $\widetilde\nu_{d^{-1}\alpha}$  and $\widetilde\mu_{d^{-1}\alpha}$  are  the measures built in Section \ref{muspectrum} associated with the exponent $d^{-1}\alpha$, then the measure $\mu_\alpha:=  ((\widetilde\mu_{d^{-1}\alpha} )^{\otimes d})  $ satisfies the same properties as   $\mu_\alpha$  (described in Proposition \ref{propmualpha}).
 
\begin{definition}
\label{defMd}
Set $\mathcal{M}_d=\{\mu^{\otimes d}:\mu\in  \mathcal{M}_1\}$.
\end{definition}

By construction, for any $\mu \in \mathcal{M}_d$ and its associated auxiliary measures $\nu_\alpha$, the inequalities \eqref{encadr2}, \eqref{loulou} and \eqref{cardinalj} and all those of Section \ref{muspectrum} still hold true.

\subsection{A conditioned ubiquity property associated with the elements of $\mathcal {M}_d $}\label{ubi}\ The property  established in this section  plays a key role in determining the singularity spectrum of typical elements in $\widetilde B^{\mu,p}_q(\R^d)$ when $p<+\infty$ and $\sigma_\mu(\alpha_{\min})>0$.

Let  $\mu \in \mathcal{M}_d$. In this section, we measure the size of the set of    those points $x\in \R^d$ which are infinitely often close to dyadic vectors $2^{-j}k \in \R^d$ such that the order of magnitude of $\mu(\lambda_{j,k})$ is $2^{-j\alpha_{\min}}$. 

\begin{definition}
\label{defirreducible}
A dyadic vector $k 2^{-j}$, $k\in \Z^d$, $j\in \N$ is  {\em  irreducible} when $k\in \Z^d\setminus (2\Z)^d$.

The  irreducible representation of  a dyadic element  $ k 2^{-j}$ with $k\in \Z^d$, $j\in \N$,  is  the unique irreducible dyadic number $\ki 2^{-\ji }$ such that $k2^{-j}=\ki 2^{-\ji }$.

If $\lambda=2^{-j}(k+[0,1]^d) \in \mathcal{D}_j$, then its associated irreducible cube   is $ \lai := 2^{-\ji}(\ki +[0,1]^d) \in \mathcal{D}_{\ji}$, where $\ki 2^{-\ji} $ is the irreducible representation of $ k 2^{-j}$.
\end{definition}
Observe that  $\la$ is the  dyadic cube of generation $j$  located at the ``bottom-left" corner of   $\lai$. We  can write  $\lambda= \lai \cdot [0,2^{-(j-\ji)}]^d $,  with the notations defined in Remark \ref{rem-philambda}.  

\begin{definition}
For $\delta>1$  and $j\geq 1$,  let $(j)_\delta$ be the largest integer in $\gamma(\N)\cap [0,j/\delta]$ (recall the definition \eqref{defgammag} of the mapping $\gamma$).

For any positive sequence  $ \eta=(\eta_j)_{j\geq 1}$, let us define the set
$$
X_j(\delta,\eta)=\left\{k2^{-(j)_\delta}\in[0,1]^d: 
\begin{cases} 
k\in \Z^d\setminus 2\Z^d,\\
\mu\big (2^{-(j)_\delta}(k+[0,1]^d\big )\ge 2^{-(j)_\delta(\alpha_{\min}+\eta_{(j)_\delta}), \ }\\
\mu\big (2^{-(j)_\delta}k+2^{- j}[0,1]^d\big )\ge 2^{-j (\alpha_{\min}+ \eta_{j})}
\end{cases}\!\!\!\!\!\!\!\!
\right \}.
$$
\end{definition}

Recall that by construction and \eqref{encadr2}, 
\begin{align*}
\mu\big (2^{-(j)_\delta}(k+[0,1]^d ) \big ) & \le 2^{-(j)_\delta(\alpha_{\min}-\ep)} \\
\mu\big (2^{-(j)_\delta}k+2^{- j}[0,1]^d\big ) & \le 2^{-j (\alpha_{\min}-\ep)},
\end{align*} 
which are complementary to the inequalities used to defined $X_j(\delta,\eta)$. Hence,  $X_j(\delta,\eta)$ contains irreducible dyadic vectors of generation $(j)_\delta$ whose $\mu$-mass is controlled both at generation $(j)_\delta$ and at generation $j$ by the exponent $\alpha_{\min}$ (note that $(j)_\delta\sim j/\delta$).

\begin{definition}
For  any positive sequence $\eta=(\eta_j)_{j\ge 1} $ and any increasing sequence of integers $(j_n)_{n\ge 1}$, set 
$$
S(\delta,\eta,(j_n)_{n\ge 1})=\bigcap_{N\ge 1} \ \bigcup_{n\ge N} \  \bigcup_{k2^{-(j_n)_\delta} \in X_{j_n}(\delta,\eta)} (k2^{-(j_n)_\delta}+2^{-j_n}[0,1]^d).
$$ 
\end{definition}

An element   $y\in S(\delta,\eta,(j_n)_{n\ge 1}) $ satisfies $|y-k2^{-(j_n)_\delta}| \leq  2^{-j_n}  \sim   2^{- \delta  \cdot (j_n)_\delta} $ for infinitely many  dyadic vectors of the form $k2^{-(j_n)_\delta } \in X_{j_n}(\delta,\eta)$: we say that $y$ is approximated at rate $\delta$  by the elements of the sets $X_{j_n}(\delta,\eta) $, $n\ge 1$ (around which $\mu$-mass is locally controlled by $\alpha_{\min}$ at generations $j_n$ and $j_n(\delta)$).

\mk Recall that the  lower Hausdorff dimension of  a  Borel probability measure $\nu$ on $\R^d$ is the infimum of the Hausdorff dimension of the Borel sets of positive $\nu$-measure (see~\cite{Fan1994} for instance).  

\begin{proposition}\label{ubiquity}
Suppose that $\si_\mu(\alpha_{\min})>0$.  

There is a positive sequence $\eta=(\eta_j)_{j\geq 1} $ converging to 0 when  $j \to + \infty$ such that for any $\delta>1$, for any  increasing sequence of integers $(j_n)_{n\ge 1}$, there exists a Borel probability measure $\nu$ on $\R^d$ of lower Hausdorff dimension larger than or equal to $\si_\mu(\alpha_{\min})/\delta$, and  such that $\nu (S(\delta,\eta,(j_n)_{n\ge 1}))=1$. 

In particular, $\dim S(\delta,\eta,(j_n)_{n\ge 1})) \geq \si_\mu(\alpha_{\min})/\delta$. 
\end{proposition}

\begin{remark} When $\mu$ is the Lebesgue measure, Proposition \ref{ubiquity} is proved in \cite{JAFF_FRISCH}.

\end{remark}

\begin{remark} 
\label{RemED}
Proposition \ref{ubiquity} is proved for a measure $\mu \in \mathcal{M}_d$, but it is extended without any difficulty to powers of $\mu$, i.e. to environments $\mu\in \mathcal{E}_d$ (recall  \eqref{defmd*}).  
\end{remark}

\begin{proof} 

We first deal with the case $d=1$. For simplicity,   $\sigma_\mu$  is denoted by  $\sigma$.

\medskip

{\bf Preliminary  observation.}  Recall the construction of the measure $\mu$ and the notations of Section \ref{sec-proof-step1}.  
\begin{definition}
\label{pag} Let $g=\ell+\sum_{n=N_0}^{N-1} \ell_n \in \N^*$, with $N\ge N_0$ and $1\le \ell\le \ell_N$. A real number   $x\in [0,1]$ satisfies property $P(\alpha_{\min},g)$ when there exists a word $w\in\Sigma_g$ such that $x\in \pi([w])$ and   writing  $w = J_{N_0}\cdot J_{N_0+1}\cdots J_{N-1} \cdot J$ with $J_n = j_{n,1}\cdots j_{n,\ell_n} \in \{0,...,2^{n}-1\} ^{\ell_n}$ for $n \in \{N_0,\ldots, N-1\} $ and $J_N = j_{N,1}\cdots j_{N,\ell} \in \{0,...,2^{N}-1\} ^\ell $, then all the $j_{n,i}$ are such that $\beta_{n,j_{n,i}} =  \alpha_{\min}$.
\end{definition}

It is direct to see  that  there exists a sequence $(\eta_j)_{j\ge 1}$ such that for all $x\in [0,1]$, for all $g\ge 1$, if $x$ satisfies property $P(\alpha_{\min},g)$, then for all $1\le j\le \gamma (g)$, one has 
$$
\mu(\lambda_j(x))\ge 2^{-j(\alpha_{\min}+\eta_j)}.
$$
Fix such a sequence $\eta=(\eta_j)_{j\ge 1}$. 

\mk

Fix $\delta> 1$ and    an increasing sequence of integers $(j_n)_{n\ge 1}$. We are going to construct a Cantor subset $\mathcal K$ included in  $S(\delta,\eta,(j_n)_{n\ge 1})$ and a Borel probability measure $\nu$ supported on $\mathcal K$ such that for all closed dyadic subcubes $\lambda$ of $[0,1]^d$ of generation $j\ge 0$,  one has $\nu(\lambda)\le 2^{-j(\delta^{-1}\si(\alpha_{\min})-\psi(j))}$, where the function $\psi: \N\to (0,+\infty)$ tends to 0 as $n\to \infty$. The mass distribution principle (see \cite{Fa1}) allows then to conclude that  $\dim S(\delta,\eta,(j_n)_{n\ge 1})) \geq \si(\alpha_{\min})/\delta$. 

\medskip

We   proceed in three steps. Notations and definitions of Section~\ref{muspectrum} are adopted.
 
  \mk

{\bf Step 1}: Construction of a family of measures $(\nu^\la)_{\la \in \mathcal{D}}$.

\mk

A family of auxiliary measures indexed by the closed dyadic subintervals of $[0,1]$ is built  in a very similar way  as $\mu_{\alpha_{\min}} $   in~Section~\ref{muspectrum}. 

Let us introduce a notation: for $j\in\N^*$, set 
$$
N(j)=\begin{cases}N_0&\text{ if $1\le j\leq \ell_{N_0} N_0$},\\
N&\text{ if $j>\ell_{N_0} N_0$ and $\gamma (\sum_{n=N_0}^{N-1} \ell_n)<j\le \gamma(\sum_{n=N_0}^{N} \ell_n)$}.
 \end{cases}
 $$
 Observe that 
 \begin{equation}
 \label{nj}
 \lim_{j\to +\infty} \frac{N(j)}{j} =0.
 \end{equation}

 Let  $N\geq N_0+1$, $1\le \ell\le \ell_{N}$, and  $g =\ell+ \sum_{n=N_0}^{N-1} \ell_n$. Let  $J$ be an integer such that
$\gamma (g-1)<  J\le   \gamma(g)$. Note that $J\ge j_0:= \ell_{N_0}N_0 +1$.

Fix $\lambda \in \mathcal D_J$,  and construct a measure $\nu^\lambda$ supported on $\lambda$ as follows.

 For each $n\ge N=N(J)$,  consider 
\begin{equation}
\label{defjalpha2}
\mathcal J_{n,\alpha_{\min}} = \{ j\in\{0,\ldots 2^n-1\} :  \mbox{ $j$ is odd and } \ \beta_{n,j} = \alpha_{\min} \}.
\end{equation}
Using \eqref{RNi} and \eqref{RNibis},  one sees  that  for an $n\geq N$, 
\begin{equation}\label{RNiter}
\#\mathcal J_{n,\alpha_{\min}} \ge 2^{n (\sigma(\alpha_{\min})-2\ep_n)}.
\end{equation}   
Writing $\lambda=K2^{-J}+2^{-J}[0,1]$,   denote by $\lambda_g \subset \la $ the   dyadic subinterval $K2^{-J}+2^{-\gamma(g)}[0,1]$ and $[w_{\lambda_g}]$ the unique cylinder such that $\pi ([w_{\lambda_g}])=\lambda_g$. Observe that $[w_{\lambda_g}]\in \mathcal{C}_g$,  the set of cylinders  of generation $g$ in $\Sigma$.  Then,    consider the set 
$$
\Sigma^\lambda= \{w_{\lambda_g}\} \times (\mathcal J_{N,\alpha_{\min}})^{\ell_N-\ell}\times\prod_{n=N+1}^\infty (\mathcal J_{n,\alpha_{\min}})^{\ell_n} \ \subset \Sigma,
$$ 
and for each $n\ge N$ and  $w\in \Sigma_g\times \{0,\ldots,2^N-1\}^{\ell_N-\ell}\times\prod_{k=N+1}^n \{0,\ldots,2^k-1\}^{\ell_k}$ set 
$$
\rho^\lambda([w])
=\begin{cases}
(\# \mathcal J_{N,\alpha_{\min}})^{-\ell_N+\ell}\prod_{k=N+1}^n (\#  \mathcal J_{k,\alpha_{\min}} )^{-\ell_k}&\text{if } [w]\cap \Sigma^\lambda\neq\emptyset\\
0&\text{otherwise.}
\end{cases}
$$
This yields an atomless  measure $\rho ^\la $ whose   support is $\Sigma ^\lambda$. Finally, the measure $\nu^\lambda=\rho^\lambda\circ\pi^{-1}$ is  a probability measure   supported on $ \la_g \subset \zu $.

\medskip

By construction of $\nu^\lambda$, using \eqref{RNiter}, for $g'\ge g$ and $\lambda'\in\mathcal D_{\gamma(g')}$, one has either $\nu^\lambda(\lambda')=0$, or $\lambda'\cap \pi(\Sigma^\lambda)\neq\emptyset$ and 
\begin{align*}
\nu^\lambda(\lambda')\le 2^{-(\gamma(g')-\gamma (g))( \si (\alpha_{\min})-2\ep_{N(J)})} & \le 2^{-(\gamma(g')-J)( \si(\alpha_{\min})-2\ep_{N(J)})} 2^{N(J)\si(\alpha_{\min})}.
\end{align*}
Consequently, for every $g'\geq g$ and   $\gamma(g')< j\le \gamma(g'+1)$, for  $\lambda'\in \mathcal D_j$ one has
\begin{equation}\label{nulambda}
\nu^\lambda(\lambda')\le 2^{-(j-J)(\si (\alpha_{\min})-2\ep_{N(J)})} 2^{2N(j)\si(\alpha_{\min})}.
\end{equation}
This inequality extends easily to all integers $j$ such that $J\le j\le \gamma(g)$ and $\lambda'\in \mathcal D_j$. 

\begin{remark}\label{irreduc/alphamin}By construction, since only odd integers $j$ are considered in the definition of the sets $\mathcal{J}_{n,\alpha_{\min}}$, if $ \widehat \lambda\varsubsetneq \lambda$ and $\nu^\lambda(\widehat \lambda)>0$, then $\widehat \lambda=\lambda_{\widehat j,\widehat k}$ with $\widehat k 2^{-\widehat j}$ irreducible. Moreover, writing $\gamma(\widehat g)<\widehat j\le \gamma(\widehat g+1)$, if property $P(\alpha_{\min},g)$ of Definition~\ref{pag}  holds for all $x\in\lambda$, then $P(\alpha_{\min},\widehat g)$ holds for all $x\in \widehat \lambda$. 
\end{remark}

\medskip

We finally set $\nu^{\lambda}=\nu^{[0,2^{-j_0}]}$ if $\lambda\in\bigcup_{j=1}^{j_0-1}\mathcal D_j$ and $\lambda\subset [0,1]$.

 \mk

{\bf Step 2}: Construction of a Cantor set $\mathcal K\subset S(\delta,(\eta_{j})_{j\geq 1} ,(j_n)_{n\ge 1})$ and a Borel probability measure $\nu $ supported on $\mathcal K$.

\medskip

Recall that $j_0=N_0\ell_{N_0}+1$. Define  $n_1=0$, $\mathcal G_1=\{[0,2^{-j_0}]\}$ and  a set function $\nu$ on $\mathcal G_1$ by $\nu([0,2^{-j_0}])=1$. Note that $\gamma(\ell_{N_0})<j_0\le \gamma(\ell_{N_0}+1)$, and that for all $x\in [0,2^{-j_0}]$, property $P(\alpha_{\min},\ell_{N_0})$ holds (recall Definition \ref{pag}).

Let $p$ be a positive integer. Suppose that   $p$ families $\mathcal G_1,\ldots, \mathcal G_p$ of closed dyadic intervals, as well as $p$ integers $0=n_1<n_2 < \cdots <n_p  $ are constructed such that:

\begin{enumerate}

\item[(a)] for every $k\in \{2,...,p\}$, 
  $(j_{n_k})_\delta \ge j_0$;
   
  \item[(b)]
  for every $k\in \{2,...,p\}$, $\mathcal G_k \subset \{x+2^{-j_{n_k}}[0,1]^d: x\in X_{j_{n_k}}(\delta,\eta)\} \subset \mathcal{D}_{j_{n_k}}$;
  
  \item[(c)] for every $k\in \{1,...,p\}$,  writing $\gamma(g_k )<j_{n_k}\le \gamma(g_k+1)$ for some integer $g_k$, every $x\in\mathcal{G}_k $  satisfies property $P(\alpha_{\min},g_k)$;

  \item[(d)]    for every $k\in \{2,...,p\}$,  the   irreducible intervals $\{ \lai : \lambda\in \mathcal G_k\} $ are pairwise disjoint;
  
  \item[(e)]  for every $k\in \{2,...,p\}$ and every element of $\lambda\in \mathcal G_k$, there is a unique $\lau  \in \mathcal{G}_{k-1}$ such that $\lambda\subset \lai \subset \lau$;

  \item[(f)]  the measure $\nu$ is defined on the $\sigma$-algebra generated by the elements of $\bigcup_{k=1}^p\mathcal G_k$ by the following formula: for all $2\le  k\le p$ and $\lambda\in \mathcal G_k$, 
$$
\nu (\lambda):=\nu (\lau) \nu^{\lau} (\lai);
$$
  \item[(g)]
  for all $2\le  k\le p$ and $\lambda\in \mathcal G_k$, 
\begin{equation}
\label{eqnula}
\nu (\lambda)\le 2^{-j_{n_k} (\delta^{-1}\sigma(\alpha_{\min})-3\ep_{N(j_{n_{k-1}})}}.
\end{equation}
\end{enumerate}

Let us explain how to build $n_{p+1} $ and $\mathcal{G}_{p+1}$.

Write  $\gamma(g_k )<j_{n_p}\le \gamma(g_k+1)$, where  $g_k =\ell+ \sum_{n=N_0}^{N-1} \ell_n\in \N$ with  $N\geq N_0$ and  $1\le \ell\le \ell_{N}$.

 Fix $n_{p+1}$ so that $\gamma(g_k+1)\le     (j_{n_{p+1}})_\delta$ (other constraints on $n_{p+1}$ will be given a few lines below). 
  
 Consider  $ \lau \in \mathcal G_{p}$.  For every $\lah\in \mathcal{D}_{  (j_{n_{p+1}})_\delta}$ with $\lah\subset \lau$ and   $\nu^{\lau}(\lah)>0$,  \eqref{nulambda} gives
$$
\nu(\lau) \nu^{\lau}(\lah) \le \nu(\lau) 2^{-(  (j_{n_{p+1}})_\delta -j_{n_p})(\sigma(\alpha_{\min})-2\ep_{N(j_{n_p})})} 2^{2N(  (j_{n_{p+1}})_\delta)\sigma(\alpha_{\min})}.
$$
By \eqref{eqnula} applied to $\nu(\lau)$,  and then   \eqref{nj},  choosing $n_{p+1}  $  large enough  yields that 
$\nu(\lau) \nu^{\lau}(\lah) \le 2^{-j_{n_{p+1}} (\delta^{-1}\sigma(\alpha_{\min}) -3\ep_{N(j_{n_p})})}$  (the equivalence $(j_{n_{p+1}})_\delta \sim j_{n_{p+1}}/\delta$ was used). 

Further, one    sets
\begin{equation}
\label{defgpp1}
\mathcal{G}_{p+1}=\bigcup_{\lau\in \mathcal G_p} \left\{k2^{-(j_{n_{p+1}})_\delta}+2^{-j_{n_{p+1}}}[0,1]:\begin{cases}  \lah =k2^{- (j_{n_{p+1}})_\delta}+2^{- (j_{n_{p+1}})_\delta}[0,1] \subset \lau\\ \nu^{\lau}(\lah)>0 \end{cases} \vspace{-3mm} \right\}.
\end{equation}

By construction, $\mathcal{G}_{p+1} \subset \mathcal{D}_{j_{n_{p+1}}}$, and each interval  $\la \in \mathcal{G}_{p+1}$ is the left-most interval inside the corresponding interval $\lah \in \mathcal{D}_{(j_{n_{p+1}})_\delta}$. It follows from this, (c) and Remark~\ref{irreduc/alphamin} that property (c) holds at generation $p+1$ as well. 

Next, for every  $\la \in \mathcal{G}_{p+1}$ associated with  $\lah \in \mathcal{D}_{(j_{n_{p+1}})_\delta}$ and $\lau\in \mathcal{G}_{p}$, one finally sets  $\nu(\lambda)=\nu(\lau) \nu^{\lau}(\lah) $.

\sk

The previous construction and the above remarks show that  all the items (a)-(g) above hold with $p+1$ as well.

\sk

Finally, we define 
$$\mathcal K=\bigcap_{p\ge 1}\bigcup_{\lambda\in \mathcal G_p} \lambda,$$
 and the set function $\nu$ defined on the elements of $\bigcup_{p\ge 1}\mathcal G_p$ extends to a Borel probability measure on $[0,1]$, whose topological support is $\mathcal K $. It is direct to check that $\nu$ is atomless, and that due to property (d) and the preliminary observation, $\mathcal K\subset S(\delta,\eta,(j_n)_{n\ge 1})$. 

\mk

{\bf Step 3}: Let us study the H\"older properties of  $\nu$ to get a lower bound for its  lower Hausdorff dimension. 

\mk

Fix a closed dyadic subinterval $\lambda$  in  $[0,1]$ of generation $j\ge j_{n_2}$ such that the interior of $\lambda$ intersects $\mathcal K$. 
Let $p\geq 2 $ be the smallest integer  such that  the interior of $\lambda $ intersects at least two elements of $\mathcal G_p$. Necessarily,  $j\le j_{n_p}$. 

 Let $  \lau $ be  the unique element of $\mathcal G_{p-1}$ such that the interior of $\lambda$ intersects $  \lau$. Since $\nu$ is atomless,   $\nu(\lambda)\le \nu  (\lau) $. In addition, $\nu(\lambda)=\nu( \lau)\nu^{  \lau}(\lah)$ where $\lah$ is associated with $\la$ as in \eqref{defgpp1}.

 Consequently, for every $p$, denoting $\ep_{N(j_{n_p})}$ simply by $\widetilde \epsilon_p$, if $j\le j_{n_{p-1} }$ then
$$
\nu (\lambda)\le  \nu (\lau )\le 2^{-j_{n_{p-1}} (\delta^{-1}\sigma(\alpha_{\min}) -3\widetilde\epsilon_{p-2})}\le2^{-j (\delta^{-1}\sigma(\alpha_{\min}) -3\widetilde\epsilon_{p-2})},
$$
and if $j>j_{n_{p-1}}$, then by \eqref{nulambda} and \eqref{eqnula}, one has
\begin{align*}
\nu(\lambda)&=\nu(\lau )\nu^{\lau}(\lah)\\
&\le  2^{-j_{n_{p-1}} (\delta^{-1}\sigma(\alpha_{\min}) -3\widetilde\epsilon_{p-2})} 2^{-(j-j_{n_{p-1}})(\sigma(\alpha_{\min})-2\widetilde\epsilon_{p-1})} 2^{2N(j)\sigma(\alpha_{\min})}\\
&=2^{-j(\delta^{-1}\sigma(\alpha_{\min})-\varphi(\lambda))},
\end{align*}
where
$$
\varphi(\lambda)= 3\widetilde\epsilon_{p-2}+ \frac{(j-j_{n_{p-1}})(\sigma(\alpha_{\min})(\delta^{-1}-1)+3\widetilde\epsilon_{p-2}-2\widetilde \epsilon_{p-1})+ 2 N(j)\sigma(\alpha_{\min})}{j}.
$$

Pay attention to the fact that in the formula above, $p$ depends  a priori on $\lambda$ and $j$. However, this dependence can be uniformly controlled. Indeed, observe that $\varphi(\lambda)\le 6\, \widetilde\epsilon_{p-2}+\frac{2 N(j)\sigma(\alpha_{\min})}{j}$, and that when $j$ tends to $+\infty$, 
$$ \min\{p\geq 2: \mbox{$ \exists \,\lambda\in\mathcal{D}_j$ such that  the interior of $\lambda $ intersects at least 2 elements of $\mathcal G_p$ }\}$$ also  tends to $+\infty$. Consequently, $\widetilde\epsilon_{p -2}$ converges uniformly to 0 over $\{\lambda \in \mathcal D_j,\ \mathrm{Int}(\lambda)\cap \mathcal K\neq\emptyset\}$ as $j\to + \infty$. Thus, remembering that \eqref{nj} holds as well,  one concludes that there exists a function $\psi: \N\to(0,+\infty)$ such that $\lim_{j\to + \infty} \psi(j)=0$ and for every $\lambda \in \mathcal{D}^0_j$,
$$
\nu(\lambda)\le 2^{-j(\delta^{-1}\sigma(\alpha_{\min})-\psi(j))}.
$$
In particular the lower Hausdorff dimension of $\nu$ is greater than $\sigma(\alpha_{\min})/\delta$. 
Since   $\mathcal K\subset S(\delta,\eta,(j_n)_{n\ge 1})$,  $\nu ( \mathcal K)=1$, we get $\dim S(\delta,\eta,(j_n)_{n\ge 1}) \geq \delta^{-1}\sigma(\alpha_{\min})$, and the conclusions of Proposition \ref{ubiquity} holds in dimension 1.

\mk \smallskip

For the case $d\ge 1$, we know  by Section \ref{dimensiond} that a measure $\mu\in\mathcal{M}_d$ is equal to $\mu_1^{\otimes d}$  for some $\mu \in \mathcal{M}_1$. Hence, with the definitions and notations introduced earlier in this section,  the tensor product measure  $\nu^{\otimes d}$ of the measure $\nu$ associated above  with the measure $\mu_1$ satisfies the conclusions of Proposition \ref{ubiquity} in any dimension $d$.
\end{proof}

 \subsection{The set of badly approximated points supports the auxiliary measures $\mu_\alpha$}
\label{secdiop}

The measures $\mu_\alpha  $ described in Proposition \ref{propmualpha} are supported on the set of points which are badly approximated by dyadic vectors, as stated by the following lemma. This property is  key for the study of typical singularity spectra, in Section \ref{sec_saturation2}.

 \begin{lemma}
 \label{lemdiop}
 For every $x$,  call   $ \overline{\lambda_{j}(x) } \in \mathcal{D}_{ \overline{j(x)} }$   the  irreducible representation of $\lambda_{j}(x)$.  For every $\alpha\in [\alm,\alpha_{\max}]$ such that $\tau_\mu^*(\alpha)>0$, for $\mu_\alpha$-almost every $x$, one has $\lim_{n\to+ \infty} \frac{\overline{j_{n(x) }}}{j_n}=1$. 
 \end{lemma}
 \begin{proof}
 
Fix $\alpha\in [\alm,\alpha_{\max}]$ and $\delta>1$. For $j\in\N^*$, let $E_\mu(\alpha,\delta,j)=\{x\in E_\mu(\alpha): \frac{\overline{j(x)}}{j}\le \delta^{-1}\}$   
and
$$E_\mu(\alpha,\delta):= \Big\{x\in E_\mu(\alpha): \liminf_{j\to + \infty} \frac{\overline{j(x)}}{j}\le \delta^{-1}\Big\}= \limsup_{j\to+ \infty}E_\mu(\alpha,\delta,j).$$
For $\ep>0$,  let 
$$F_\mu(\alpha, j,\ep)=\{x\in [0,1]^d: \forall \ j'\ge j, \ 2^{-j'(\alpha+\ep)}\le \mu(\lambda_{j'}(x))\le 2^{-j'(\alpha-\ep)}\}.$$ 
Setting $j_\delta=\lfloor j/\delta\rfloor$, the following inclusion holds : 
 $$
E_\mu(\alpha,\delta)\subset \bigcap_{\ep>0} \bigcap_{J\ge 1} \bigcup_{j\ge J} \bigcup_{\substack{\lambda_{j_\delta,k}  \in \mathcal D_{j_\delta} :\\ \lambda_{j_\delta,k}\cap F_\mu(\alpha, j_\delta ,\ep)\neq\emptyset}} B(k2^{-j_\delta}, 2^{-j}).
 $$
 Using Proposition~\ref{fm}(1) or (4), for every fixed $\ep>0$, one sees that  the cardinality of $\{ \lambda_{j_\delta,k}  \in \mathcal D_{j_\delta} : \lambda_{j_\delta,k}\cap F_\mu(\alpha, j_\delta ,\ep)\neq\emptyset  \}$ is less than $2^{j_\delta (\tau_\mu^*(\alpha)+\ep)}$ when $j$ is large.
 
 Combining this with the previous embedding,   coverings of $E_\mu(\alpha,\delta)$ are obtained using  sets of the form  $\bigcup_{j\ge J} \bigcup_{\substack{\lambda_{j_\delta,k}\in \mathcal D_{j_\delta} :\\ \lambda_{j_\delta,k}\cap F_\mu(\alpha, j_\delta ,\ep)\neq\emptyset}} B(2^{-j_\delta}k, 2^{-j})$, and it is easily seen that $\dim E_\mu(\alpha,\delta)\le \tau_\mu^*(\alpha)/\delta$. This implies that  $\mu_\alpha(E_\mu(\alpha,\delta))=0$, again because   $\mu_\alpha$ may give a positive mass to a set $E$ only when $\dim E\geq \tau_\mu^*(\alpha)$.

 Since this holds for all $\delta>1$,  $\liminf_{j\to\infty} \frac{\overline{j(x)}}{j}=1$ for $\mu_\alpha$-almost every $x$, and in particular $\lim_{n\to\infty} \frac{\overline{j_n(x)}}{j_n}=1$.
 \end{proof}

\section{Wavelet characterization of $B^{\mu,p}_{q}(\R^d)$ and $\widetilde B^{\mu,p}_{q}(\R^d)$}
\label{sec_integral2}

After some definitions  and  two basic lemmas in Section~\ref{implication0},   Theorem~\ref{th_equivnorm_2} is proved when $ p \in [1,+\infty)$ in Section~\ref{caspfini}. The much simpler case  $p=+\infty$ is left to the reader who can easily adapt the lines used to treat the case $p<+\infty$.

\subsection{Preliminary definitions and observations}
\label{implication0}

We start by  extending the definition of the moduli of smoothness \eqref{defomegat} and \eqref{defomega}     to all  Borel sets $\Omega\subset \R^d$.

\begin{definition}
Let $\Omega \subset \R^d$. For $h\in \R^d$, let 
\begin{equation}
\label{defomegahn}
\Omega_{h,n} = \{x\in \Omega: x+kh\in\Omega,\, k=1,\ldots, n \}.
\end{equation}
 Then, for $f:\R^d \to\R$, $\mu\in \mathcal{H}(\R^d)$, $t>0$ and  $n\geq 1$ set
\begin{align}
\label{defomegat'}
  \omegat_n(f,t, \Omega) _p &= \sup_{t/2\leq |h|\leq t} \| \Delta ^{\mu,n}_hf \|_{L^p(\Omega_{h,n})}\\
  \mbox{ and } \ \ \ \ \label{defomega'}
  \omega_n(f,t, \Omega) _p &= \sup_{0\leq |h|\leq t} \| \Delta^n_hf \|_{L^p(\Omega_{h,n})}.
\end{align}
  \end{definition}

Let $\mu \in \mathcal{C}(\R^d)$ be  an almost doubling capacity   such that property (P) holds  with exponents $0<s_1 \leq s_2$.
Let $n\geq r=\lfloor s_2+ \frac{d}{p}\rfloor+1$ and $\Psi = (\phi,\{\psi^{(i)}\}_{i=1,...,2^d-1}) \in \mathcal F_r$ (see Definition~\ref{fr}). 

Also, recall that for $\lambda= (i,j,k) \in \Lambda_j$, $\psi_{\lambda}(x) =   \psi^{(i)}(2^jx-k)$. It follows from the construction of $\Psi$ (see \cite[Section 3.8]{Meyer_operateur}) that there exists an integer $N_\Psi\in\N^*$ such that ${\rm supp}(\phi)$ and ${\rm supp}(\psi^{(i)})$ are included in $N_\Psi[0,1]^d$. Our proofs will use some estimates established in \cite{cohen-book}. These estimates require to associate with each $\lambda= (i,j,k) \in \Lambda_j$ a larger cube $\widetilde \lambda$ described in the following definition. 
%
\begin{definition}
\label{defsupport} 
For each $\lambda= (i,j,k) \in \Lambda_j$, set 
$$\widetilde {\lambda}= \lambda_{j,k}+2^{-j}({\rm supp}(\phi)-{\rm supp}(\phi)).$$
\end{definition}
 Note that $\lambda_{j,k}\subset \supp(\psi_\lambda)\subset \widetilde {\lambda}\subset 3N_\Psi\lambda_{j,k}$, the second embedding coming from the construction of compactly supported wavelets (see \cite[Section 3.8]{Meyer_operateur}). 

\medskip

For every $j\in\N$, the cubes $(\widetilde\la)_{\la\in \Lambda_j}$ do not overlap too much, in the sense that \begin{equation}
 \label{overlap}
 K_\Psi := \sup_{j\in\N}\sup_{\lambda\in\Lambda_j}\#\{\la'\in \Lambda_j : \mbox{$\widetilde \lambda\cap\widetilde \lambda'\neq\emptyset$} \}<+\infty.
\end{equation}

\begin{lemma}\label{integral estimate} Let $p\in [1,+\infty)$ and $n\in\N^*$. There exists a constant $C_{d,n,p}$ (depending on $p$, $n$, and $d$ only) such that  for all $f\in L^p_{\rm{loc}}(\R^d)$, $t>0$ and $\lambda\in\Lambda$, the following inequality holds:
$$
\omega_n(f,t, \widetilde \lambda)^p _p\le C_{d,n,p}\, t^{-d}  \int_{t\le |y|\le 4nt} \int_{{\widetilde\lambda}+B(0,2nt)}|\Delta^n_yf(x)|^p\,\mathrm{d}x\mathrm{d}y.
$$
\end{lemma} 
\begin{proof} The approach  follows the lines of the proof of \cite[inequality (3.3.17)]{cohen-book}, where a similar inequality is proved.

Fix $f$, $t$ and $\lambda$ as in the statement. For any $h,y\in\mathbb{R}^d$, recall  the following   equality (see (3.3.19) in \cite{cohen-book}):
$$
\Delta^n_h f(x)=\sum_{k=1}^n(-1)^k\binom{n}{k}\big [\Delta^n_{ky}f(x+kh)-\Delta^n_{h+ky}f(x)\big ].
$$
Integrating $|\Delta^n_h f|^p$ over $\widetilde \lambda_{h,n}$ (recall formula \eqref{defomegahn}),     one sees that for some constant $C_{n,p}>0$, when  $|h|\leq t$,
\begin{align*}
\|\Delta^n_h f\|^p_{L^p(\widetilde \lambda_{h,n})}&\le C_{n,p}\sum_{k=1}^n\|\Delta_{ky}^n f(\cdot+kh)\|^p_{L^p(\widetilde \lambda_{h,n})}+\|\Delta_{h+ky}^n f\|^p_{L^p(\widetilde \lambda_{h,n})}\\ 
&\le C_{n,p}\sum_{k=1}^n\|\Delta_{ky}^n f\|^p_{L^p({\widetilde\lambda}+B(0,2nt))}+\|\Delta_{h+ky}^n f\|^p_{L^p({\widetilde\lambda}+B(0,2nt))}.\end{align*}
Then, defining $C_d=\mathcal L^d (B(0,3)\setminus B(0,2))$,  an integration with respect to  $y$ over $B(0,3t)\setminus B(0,2t)$ yields
$$
C_d t^{d}\|\Delta^n_h f\|^p_{L^p(\widetilde \lambda_{h,n})} \le C_{n,p}\sum_{k=1}^n \int_{2t\le |y|\le 3t} \int_{{\widetilde\lambda}+B(0,2nt)}|\Delta_{ky}^n f(x)|^p+|\Delta_{h+ky}^n f(x)|^p\, \mathrm{d}x\mathrm{d}y.
$$
Further,   operating the change of variable $y'=ky$ in each term of the sum yields
\begin{align*}
t^d \|\Delta^n_h f\|^p_{L^p(\widetilde \lambda_{h,n})}&\le C_d^{-1}C_{n,p}  \sum_{k=1}^n \int_{2kt\le |y|\le 3kt}\int_{{\widetilde\lambda}+B(0,2nt)}|\Delta_{y}^n f(x)|^p+|\Delta_{h+y}^n f(x)|^p\, \mathrm{d}x\mathrm{d}y\\
&\le  2nC_d^{-1}C_{n,p}  \int_{t\le |y|\le 4nt}\int_{{\widetilde\lambda}+B(0,2nt)}|\Delta_{y}^n f(x)|^p\, \mathrm{d}x\mathrm{d}y.
\end{align*}
where one used that $t\le |h+y |\leq 4nt$ when $|h|\leq t $ and $|y|\geq 2t$.
The previous upper bound being independent of $h\in B(0,t)$, one concludes that 
$$ \omega_n(f,t, \widetilde\lambda) ^p_p = \sup_{0\le  |h|\leq t} \| \Delta ^n_hf \|^p_{L^p(\widetilde\lambda_{h,n})} \leq  \frac{2nC_d^{-1}C_{n,p}}{t^d}  \int_{t\le |y|\le 4nt}\int_{{\widetilde\lambda}+B(0,2nt)}|\Delta_{y}^n f(x)|^p\, \mathrm{d}x\mathrm{d}y,$$
as desired.
\end{proof}

\begin{lemma}
\label{lemdouble} 
Let $\ep>0$ and $\mu\in \mathcal{C}(\R^d)$ that satisfies Property (P) with exponents $s_1$ and $s_2$.

There exists a constant $C = C(\ep,n,\mu)  \geq 1$ such that for every $j\in\N$ and  $\lambda\in \Lambda_j$,   for every  $x\in \widetilde {\lambda}+B(0,2n2^{-j})$ and $y\in\R^d$ such that $ 2^{-j} \leq |h| \leq  4n 2^{-j}$, for every $f:\widetilde\lambda\to \R$,   one has 
$$  \frac{|\Delta_h^nf(x)|} {\mu(\lambda)}   \leq C   \frac{|\Delta_h^nf(x)|} {\mu^{(+\ep)}(B[x,x+nh])}.$$
\end{lemma}

\begin{proof}
Observe first  that under the assumptions of the Lemma, the inequality 
$$\frac{\mu(B[x,x+ny] )} {\mu( {\lambda}) }  \leq  C  (n|y|)^{-\ep},$$ 
  follows easily from the definition of the almost doubling property \eqref{ad}. Then,   Lemma \ref{lemdouble} is   deduced from  last inequality and  the definition of~$\mu^{(+\ep)}$. 
\end{proof}


\subsection{Proof of Theorem~\ref{th_equivnorm_2} when $1\le p<+\infty$}\label{caspfini}

Let us now explain our approach to get Theorem~\ref{th_equivnorm_2} when $p\in[1,+\infty)$.  Recall that ${B^{\mu,p}_q(\R^d)}$  is   defined via  $L^p$ moduli of smoothness of order $n\geq r=\lfloor s_2+d/p+1\rfloor$, and that $\Psi$ belongs to $\mathcal{F}_{r}$. The purpose  of this theorem is to establish relations between this definition \eqref{defbpqmuosc} and the wavelet-based one \eqref{defbesovwavelet}.

\sk

In Section~\ref{implication1}, it is shown   that,   for any $\ep \in (0,1) $,  when  ${B^{\mu^{(+\ep)},p}_q(\R^d)}$  is   defined via the  $L^p$ modulus of smoothness of order $n$, then    \eqref{eqnorm} holds for  any $\Psi\in\mathcal F_n$.
It is only a partial proof of the statement, since one wants to obtain \eqref{eqnorm}  for any $\Psi\in\mathcal F_{r}$.

\sk

Then, in Section~\ref{implication2},   \eqref{eqnorm2} is completely proved to hold for any $\ep\in(0,1)$ and any $\Psi\in\mathcal F_{r}$ when ${B^{\mu,p}_q(\R^d)}$  is   defined via the  $L^p$ modulus of smoothness of order exactly equal to $r$.
 Since   $\mathcal F_{n}\subset \mathcal F_{r}$, the statement also holds for $\Psi\in  \mathcal F_{n}$.

\sk

Finally, from the two preceding observations, we conclude  that \eqref{eqnorm} holds  for any $\ep\in(0,1)$ and any $\Psi\in\mathcal F_{r}$,  by applying:\begin{itemize}
\item
 first  \eqref{eqnorm} with  the environment $\mu$, the $n$-th order difference operator, $\ep/3$ and any wavelet $  \widetilde \Psi\in\mathcal F_n$, \
 \item
  then  \eqref{eqnorm2} with the environment $\mu^{(+\ep/3)}$, the $r$-th order difference operator, $\ep/3$ and the same $\widetilde \Psi\in\mathcal F_n$, 
  \item
   finally   \eqref{eqnorm} with the environment $\mu^{(+2\ep/3)}$, the $r$-th order difference operator, $\ep/3$ and $  \Psi\in\mathcal F_r$.
 \end{itemize}
 

\subsubsection{Proof of inequality  \eqref{eqnorm} in Theorem \ref{th_equivnorm_2}}
\label{implication1}

Assume that $\Psi\in\mathcal F_n\subset\mathcal F_r$. Fix $\ep>0$, $f\in L^p(\R^d)$ and $j\in\N$. 

For every $\la = (i,j,k) \in \Lambda_j$, since   $\psi_\lambda$ is orthogonal to any polynomial $P$ of degree  $\leq n$, the wavelet coefficient $c_\la$ can be written     
\begin{equation*} { c_{\lambda}  } 
 =    {\ds 2^{jd}\int_{\R^d} (f(x)-P (x))\psi_{\lambda} (x) dx }.
\end{equation*} 

Due to the local approximation of $f$ by polynomials (equation (3.3.13) in \cite{cohen-book}), there exists  a polynomial $P_\lambda$ of degree $\leq n$ such that 
$$
 \|f-P_\lambda\|_{L^p(\widetilde {\lambda})} \leq  C \omega_n(f,2^{-j},  \widetilde {\lambda}) _p,
$$
where $C$ depends on $n$ and $p$ only.   Recall that $\supp(\psi_\lambda)\subset \widetilde {\lambda}$.

The last   inequalities, together with   H\"older's inequality,  yield  
 \begin{align}
\nonumber
\frac{|c_{\lambda}| }{\mu({\lambda})} 
& \leq  2^{jd }   \frac{\|\psi _{\lambda}\|_{L^{p'}(\R^d)}  \|f-P_\lambda\|_{L^p(\widetilde {\lambda})} }{\mu( {\lambda}) }   \leq   C  2^{jd }   \frac{    2^{-jd/p'} \|\psi^{(i)} \| _{L^{p'}(\R^d)} \omega_n(f,2^{-j},\widetilde {\lambda}) _p   }{\mu( {\lambda}) }       \\
  & \le \widetilde C 2^{jd/p}  \frac{  \omega_n(f,2^{-j},\widetilde {\lambda}) _p  }{\mu( {\lambda})} 
 \label{besov_ineg1},
 \end{align}
where $\widetilde C=C \sup\left\{ \|\psi^{(i)}\|_{L^{p'}(\R^d)}: \, 1\le i\le 2^d-1\right \}$.

Then, Lemma~\ref{integral estimate}   gives
$$
\Big (\frac{|c_{\lambda}| }{\mu({\lambda})}\Big )^p\le C_{d,n,p} \widetilde C^p 2^{2dj}  \int_{2^{-j}\le |y|\le 4n2^{-j}} \int_{{\widetilde\lambda}+B(0,2n2^{-j})}\frac{|\Delta^n_yf(x)|^p}{\mu(\lambda)^p}\,\mathrm{d}x\mathrm{d}y,
$$
and by  Lemma~\ref{lemdouble}, there exists $ C' $ depending on $(\ep,n,p,\Psi)$ such that 
\begin{align*}
\Big (\frac{|c_{\lambda}| }{\mu({\lambda})}\Big )^p&\le C_{d,n,p} (C')^p 2^{2dj}  \int_{2^{-j}\le |y|\le 4n2^{-j}} \int_{{\widetilde\lambda}+B(0,2n2^{-j})}|\Delta^{\mu^{(+\ep)},n}_yf(x)|^p\,\mathrm{d}x\mathrm{d}y\\
&\le C_{d,n,p} (C')^p  \sum_{k=0}^{j_n} 2^{2dj}  \int_{2^{-j+k}\le |y|\le 2^{-j+k+1}} \int_{{\widetilde\lambda}+B(0,2n2^{-j})}|\Delta^{\mu^{(+\ep)},n}_yf(x)|^p\,\mathrm{d}x\mathrm{d}y, 
\end{align*}
where  $j_n=\lfloor \log_2(4n)\rfloor$. By  \eqref{overlap},  there exists a constant  $K_{\Psi,n}>0 $ depending on  $(\Psi,n)$ only  such that any $\lambda\in\Lambda_j$ is covered by at most $K_{\Psi,n} $  sets of the form   $\widetilde \lambda'+B(0,2n2^{-j})$ with $\lambda'\in \Lambda_j$. It follows that  
\begin{align*}
\sum_{\lambda\in\Lambda_j}   \Big (\frac{|c_{\lambda}| }{\mu({\lambda})}\Big )^p&\le K_{\Psi,n}C_{d,n,p} (C')^p  \sum_{j' =0}^{j_n} 2^{2dj}  \int_{2^{-j+j'}\le |y|\le 2^{-j+j'+1}} \int_{\R^d}|\Delta^{\mu^{(+\ep)},n}_yf(x)|^p\,\mathrm{d}x\mathrm{d}y.\end{align*}
Recalling the definition \eqref{defomegat} of $\omega^{\mu^{(+\ep)}}_n(f,t,\R^d)$, for every $j'$ the  double integral above is bounded by $2^{d(-j+j'+1)} \omega^{\mu^{(+\ep)}}_n(f,2^{-j+j'+1},\R^d)^p_p$. Since $2^{d(j'+1)} \leq 2^{d(j_n+1)} \leq (8n)^d$, one has 
\begin{align*}
\sum_{\lambda\in\Lambda_j}   \Big (\frac{|c_{\lambda}| }{\mu({\lambda})}\Big )^p&\le  C_1^p  \sum_{j'=0}^{j_n} 2^{dj}\omega^{\mu^{(+\ep)}}_n(f,2^{-j+j'+1},\R^d)^p_p,
\end{align*}
where   $C_1=((8n)^dK_{\Psi,n}C_{d,n,p})^{1/p} C'$ does not depend on $f$ or $j$.

\medskip

Suppose  now that $q\in[1,+\infty)$ (the case $q=+\infty$ is obvious).  The previous estimates together with the subadditivity of $t\ge 0\mapsto t^{1/p}$ and  the convexity of $t\ge 0\mapsto t^q$ yield
$$
\Big \|\Big (\frac{c_{\lambda} }{\mu({\lambda})}\Big )_{\lambda\in\Lambda_j}\Big \|^q_{\ell^p(\Lambda_j)}\le C^q_1 (j_n+1)^{q-1} \sum_{j'=0}^{j_n} \big (2^{dj/p}\omega^{\mu^{(+\ep)}}_n(f,2^{-j+j'+1},\R^d)_p\big )^q.
$$

Summing the last inequality over $j \in \N$ gives
$$
 \sum_{j\geq 0} \Big \|\Big (\frac{c_{\lambda} }{\mu({\lambda})}\Big )_{\lambda\in\Lambda_j}\Big \|^q_{\ell^p(\Lambda_j)}\le C^q_1 (j_n+1)^{q-1} \sum_{j=-j_n-1}^ {+\infty}  K_j  \big ( \omega^{\mu^{(+\ep)}}_n(f,2^{-j },\R^d)_p\big )^q,
$$
where  $K_j =\begin{cases} \sum_{ j'= j+1}^{j +j_n +1}  2^{qdj'/p}  & \mbox{when } j\geq -1 \\  \sum_{ j'= 0}^{j +j_n +1} 2^{qdj'/p}  & \mbox{when }  -j_n-1 \leq j  \leq -2   \end{cases}$. 
It is easily seen that there is a constant $C_2=C_2(n,q,d)$ such that $C_1^q(j_n+1)^{q-1}K_j \leq C_2 2^{qdj/p}$, so 
$$
 \sum_{j\geq 0} \Big \|\Big (\frac{c_{\lambda} }{\mu({\lambda})}\Big )_{\lambda\in\Lambda_j}\Big \|^q_{\ell^p(\Lambda_j)}\le C_2\sum_{j=-j_n-1}^ {+\infty}    \big ( 2^{dj/p}  \omega^{\mu^{(+\ep)}}_n(f,2^{-j },\R^d)_p\big )^q.
$$

Observe that there is $C_3 \ge 1 $ depending on $n$  such that for $ -j_n-1\le j \le 0 $, $2^{dj/p}\omega^{\mu^{(+\ep)}}_n(f,2^{-j},\R^d)_p\le C_3 \|f\|_{L^p(\R^d)}$. This follows from the fact that for such a $j$: 
\begin{itemize}
\item $ 2^{dj/p} \leq 1$;
\item
by periodicity of $\mu$, $\mu(B[x,x+n2^{-j}] ) \geq   \mu(\zu^d)=1 $, so $   \frac{|\Delta_y^nf(x)|} {\mu(\lambda)}  \leq |\Delta_y^nf(x)|$, and thus for some constant $C'''$ 
$$\omega^{\mu^{(+\ep)}}_n(f,2^{-j},\R^d)_p \leq  2^{j\ep} \omega _n(f,2^{-j},\R^d)_p \leq  ( \omega _n(f, 8n ,\R^d)_p  \leq C_3 \|f\|_{L^p(\R^d)}.$$ 
\end{itemize}

Consequently, for some constant $C$ independent of $f$, 
$$
\sum_{j\ge 0} \Big \|\Big (\frac{c_{\lambda} }{\mu({\lambda})}\Big )_{\lambda\in\Lambda_j}\Big \|^q_{\ell^p(\Lambda_j)}\le C  \Big (\|f\|_{L^p(\R^d)}^q+\sum_{j\ge 0}\big (2^{dj/p}\omega^{\mu^{(+\ep)}}_n(f,2^{-j},\R^d)_p\big )^q\Big ),
$$
which implies that $ \|f\|_{L^p(\mathbb R^d)}+|f|_{\mu,p,q} \leq C(  \|f\|_{L^p(\mathbb R^d)}+ |f|_{{B}^{\mu^{(+\ep)},p}_{q}(\mathbb R^d)  })$. Hence,  \eqref{eqnorm} holds when   $\Psi\in\mathcal F_n$.

\subsubsection{Proof of inequality \eqref{eqnorm2} in Theorem \ref{th_equivnorm_2}} 
\label{implication2}

Fix $\ep>0$ and  $f\in L^p(\R^d)$. 

Define the partial sums    $f_j=\sum _{\lambda\in \Lambda_j} c_\lambda \psi_\lambda$, for all $j\geq 0$.

The following lemma is needed.

\begin{lemma}
\label{besov_lemma2} Let $s\in \left (s_2+\frac{d}{p}, s_2+\frac{d}{p}+1\right )$. 
There exist a constant $C>0$  and a sequence $(\widetilde\ep_m)_{
m\in\N}\in \ell^q(\mathbb N)$ bounded by 1, independent of $f$,  such that  for all $j,J\ge 0$, 
\begin{equation}\label{ineq4.5}
  \omegat_n(f_j, 2^{-J} ,\R^d)_p    \leq   C 2^{-jd/p} \min\big ( 1,2^{(j-J)(s-s_2)}\widetilde \ep_{J-j}\big ) \left(\sum_{\lambda\in \Lambda_j}  \left(\frac{  |c_{\lambda} | }{\mu ^{(+\ep)}(\lambda)}  \right) ^p \right) ^{1/p},
\end{equation}
with the convention that $\widetilde\ep_m =1$ when $m< 0$.
\end{lemma}

\begin{proof}
Inspired by the proof of  \cite[Theorem 3.4.3]{cohen-book},   two cases are separated:

\medskip

\noindent{\bf Case 1:} 
$J<j$.  
In order to prove \eqref{ineq4.5},  let us begin by   writing that 
\begin{align}
 \omegat_n(f_j, 2^{-J} ,\R^d)_p^p  
\label{intermomegamu}  = \sup_{2^{-J-1}\le |h|\le 2^{-J}}  \sum_{\lambda'\in \mathcal D_J}   \int_{\lambda'} \frac{\left|\sum_{\lambda\in \Lambda_j}  c_\lambda\, \Delta_h^n\psi_\lambda (x)\right |^p}{\mu(B(x,x+nh))^p}\,{\rm d}x.
  \end{align}
  
Consider $\lambda\in\Lambda_j$, $x\in \R^d$,  and $h\in \R^d$   such that $2^{-J-1}\le |h|\le 2^{-J}$. Then: 
  
  \begin{itemize}
  \item [(i)]
If $x\not\in \bigcup_{k=0}^n\supp(\psi_\lambda)-kh$, then $\Delta_h^n\psi_\lambda (x)=0$;

\item[(ii)]
Let $\la' = \la_J(x)$ the unique cube of generation $J$ that contains $x$. 

\noindent
There exists an integer $N =N(n,\Psi)$   such that if $x\in \bigcup_{k=0}^n\supp(\psi_\lambda)-kh$, then necessarily   $\lambda\subset N\lambda'$. 

\item [(iii)] By the almost doubling property of $\mu$, there exists a constant $C=C(\mu,n,\Psi,\ep)$    such that for every   $x\in \bigcup_{k=0}^n\supp(\psi_\lambda) - kh$, 
\begin{equation}\label{muepla}
\mu^{(+\ep)}(\lambda)  =2^{-j\ep} \mu(\la) \leq   2^{-j\ep}  \mu(N\lambda') \leq C  {\mu(B(x,x+nh))}. 
\end{equation}
\end{itemize}

From the equality $\Delta_h^n\psi_\lambda=\sum_{k=0}^n (-1)^k\binom{n}{k} \psi_\lambda(\cdot+(n-k)h)$, the three items (i)-(iii) and \eqref{intermomegamu},  one obtains  that
\begin{align*} 
 \omegat_n(f_j, 2^{-J} ,\R^d)_p^p 
   &\le C^p\sup_{2^{-J-1}\le |h|\le 2^{-J}}  \sum_{\lambda'\in \mathcal D_J} 
  \frac{2^{j\ep p}}{ \mu(N\lambda')^p}\int_{\lambda'} \Big |\sum_{\lambda\in \Lambda_j,\, \lambda\subset N\lambda'}  c_\lambda\, \Delta_h^n\psi_\lambda (x)\Big |^p\,{\rm d}x\\
  &\le C^p  \sum_{\lambda'\in \mathcal D_J}    \frac{2^{j\ep p}}{ \mu(N\lambda')^p}  T_{j,J,\la',\la,\psi}, 
  \end{align*}
  where 
 \begin{align*}    T_{j,J,\la',\la,\psi} =  \sup_{2^{-J-1}\le |h|\le 2^{-J}}  \int_{\R^d} \Big |\sum_{k=0}^n (-1)^k\binom{n}{k} \sum_{\lambda\in \Lambda_j,\, \lambda\subset N\lambda'} c_\lambda\, \psi_\lambda (x+(n-k)h)\Big |^p\,{\rm d}x.
 \end{align*}
The  convexity inequality $(\sum_{k=0}^n|z_k|)^p\le (n+1)^{p-1} \sum_{k=0}^n|z_k|^p$  and       $\binom{n}{k} \leq 2^n$  give 
 \begin{align*} 
   T_{j,J,\la',\la,\psi}     &\le 2^{np}   (n+1)^{p-1} \sup_{2^{-J-1}\le |h|\le 2^{-J}}      \sum_{k=0}^n    \int_{\R^d} \Big |\sum_{\lambda\in \Lambda_j,\, \lambda\subset N\lambda'}  c_\lambda\, \psi_\lambda (x+(n-k)h)\Big |^p\,{\rm d}x.
   \end{align*}
 Observe that  the property (ii) above allows to  bound each integral by the same term   $\int_{\R^d} \Big |\sum_{\lambda\in \Lambda_j,\, \lambda\subset N\lambda'}  c_\lambda\, \psi_\lambda (x)\Big |^p\,{\rm d}x$. Moreover, according to \cite[Ch. 6, Prop. 7]{Meyer_operateur}, there exists $C'>0$ depending on $\Psi$ only such that 
 $$
 \int_{\R^d} \Big |\sum_{\lambda\in \Lambda_j,\, \lambda\subset N\lambda'}  c_\lambda\, \psi_\lambda (x)\Big |^p\,{\rm d}x\le C'^p2^{-jd} \sum_{\lambda\in \Lambda_j,\, \lambda\subset N\lambda'}  |c_\lambda|^p.
 $$
Consequently, 
%
 using the first inequality of \eqref{muepla},  
$$
 \omegat_n(f_j, 2^{-J} ,\R^d)_p^p \le  (CC)'^p  (n+1)^{p}2^{np}  \sum_{\lambda'\in \mathcal D_J} 2^{-jd} \sum_{\lambda\in \Lambda_j,\, \lambda\subset N\lambda'}  \Big (\frac{|c_\lambda|}{\mu^{(+\ep)}(\lambda)}\Big )^p.
 $$
Finally,  the number of dyadic cubes  $\lambda'\in\mathcal D_J$  such that $N\la'$ intersects  a given $\la \in \La_j$ is bounded uniformly with respect of  $j$ and $J$, so
$$  \omegat_n(f_j, 2^{-J} ,\R^d)_p    \leq   C 2^{-jd/p}  \left(\sum_{\lambda\in \Lambda_j}  \left(\frac{  |c_{\lambda} | }{\mu ^{(+\ep)}(\lambda)}  \right) ^p \right) ^{1/p}$$
for some  constant $C$ that depends on $n$, $p$ and other constants. This yields  \eqref{ineq4.5}.
 
\medskip

\noindent{\bf Case 2:} $J\geq j$. Let us start with a few observations. First, by assumption, $\psi^{(i)} \in B^{s,p}_{q}(\R^d)$, hence  $$
  \omega_n(  \psi^{(i)} ,2^{j-J} , \R^d) _p \leq  2^{(j-J)s} \widetilde \ep^{(i)}_{J-j},
$$
 where $(\widetilde \ep^{(i)}_{m})_{m\geq 1} \in \ell^q(\mathbb N^*)$ and  $\|\widetilde \ep^{(i)}\|_{\ell^q(\mathbb N^*)}\leq \|\psi^{(i)}\|_{B^{s,p}_q}$.
Consequently,  for all $\lambda\in\Lambda_j$  
\begin{align}\label{omeganpsila}
 \omega_n(  \psi_\lambda ,2^{-J} , \R) _p \leq   2^{(j-J)s}2^{-jd/p}\widetilde\ep_{J-j},  
  \end{align}
where $\widetilde\ep_{J-j}=\sup_i\widetilde\ep^{(i)}_{J-j}$. 
 
Next,  observe  the following facts: 
\begin{itemize}
\item [(i)]
There exists an integer $N$ independent of $j$ and $J$ such that for all $x\in \R^d$ and $h\in\R^d$ such that  $2^{-J-1}\le |h|\le 2^{-J}$,  $B[x,x+nh]\subset  N\lambda_j(x)$. Also,  $\Delta_h^n\psi_\lambda(x)\neq 0 $   only if $\lambda \subset N\lambda_j(x)$ (recall that  $\la=(i,j,k)  \subset E$ means    $\lambda_{j,k}\subset E$). 
\sk \item [(ii)]
There exist two  dyadic cubes $\lambda' \in \mathcal{D}_{J+3}$  and $\la''\in  \mathcal{D}_j$ such that $\la' \subset B(x,x+nh)$  and $\lambda'\subset\lambda'' \subset N\lambda_j(x)$. By construction, for all $\La_j\ni \lambda \subset N\lambda_j(x)$, one has 
$$
\mu(B[x,x+nh])^{-1}\le \mu(\lambda')^{-1}=  \frac{\mu(\lambda'')}{\mu(\lambda')}\frac{\mu(\lambda )}{\mu(\lambda'')}
\mu(\lambda )^{-1}.$$
Hence, using property (P$_2$) to control from above  $\frac{\mu(\lambda )}{\mu(\lambda'')}$ by $O(2^{N^d\phi(j)})$ and $\frac{\mu(\lambda'')}{\mu(\lambda')}$ by $O(2^{\phi(j)} 2^{(J-j)s_2})$, as well as the fact that  $2^{\phi(j) (N^d+1)}  \leq |\lambda |^{-\ep}$ since $\phi\in \Phi$, there exists  a constant $C$ depending on $(\mu,n,\ep)$ only such that 
$$
\mu(B[x,x+nh])^{-1}\le C2^{(J-j)s_2}(\mu^{(+\ep)}(\lambda )^{-1}.
$$
\end{itemize}

The two previous observations  yield    
\begin{align*}
  \omegat_n(f_j, 2^{-J} ,\R^d)_p^p &  = \sup_{2^{-J-1}\le |h|\le 2^{-J}}  \sum_{\lambda'\in \mathcal D_J}   \int_{\lambda'} \frac{\left|\sum_{\lambda\in \Lambda_j}  c_\lambda\, \Delta_h^n\psi_\lambda (x)\right |^p}{\mu(B(x,x+nh))^p}\,{\rm d}x \\
  &  \leq  C^p2^{(J-j)s_2 p}\sup_{2^{-J-1}\le |h|\le 2^{-J}}  \int_{\R^d}\Big (\sum_{\lambda\in\Lambda_{j},\lambda\subset N\lambda_j(x)}\frac{|c_\lambda|}{\mu^{(+\ep)}(\lambda)} |\Delta_h^n\psi_\lambda(x)|\Big )^p\, {\rm d}x.
  \end{align*}
Since  $\#\{\lambda\in\Lambda_{j}:\, \lambda\subset N\lambda_j(x)\} \leq (2Nd)^d$,   for each $x\in\R^d$ one has
\begin{align*} 
\Big (\sum_{\lambda\in\Lambda_{j},\lambda\subset N\lambda_j(x)}\frac{|c_\lambda|}{\mu^{(+\ep)}(\lambda)} |\Delta_h^n\psi_\lambda(x)| \Big )^p 
\le (2Nd)^{d(p-1)} \sum_{\lambda\in\Lambda_{j},\lambda\subset N\lambda_j(x)}\Big (\frac{|c_\lambda|}{\mu^{(+\ep)}(\lambda)}\Big )^p |\Delta_h^n\psi_\lambda(x)|^p. 
\end{align*}
Also, each $\la \in \mathcal D_j$ intersects at most $(2N)^d$ cubes  $N\lambda'$ with $\lambda'\in\mathcal D_j$, so
\begin{align*}
   \int_{\R^d}   \sum_{\lambda\in\Lambda_{j},\lambda\subset N\lambda_j(x)}\Big (\frac{|c_\lambda|}{\mu^{(+\ep)}(\lambda)}\Big )^p |\Delta_h^n\psi_\lambda(x)| ^p \, {\rm d}x   & \le   (2N)^{dp}  \sum_{\lambda\in\Lambda_j} \Big (\frac{|c_\lambda|}{\mu^{(+\ep)}(\lambda)}\Big )^p \int_{\R^d}  |\Delta_h^n\psi_\lambda(x)| ^p \, {\rm d}x .
   \end{align*}
Finally,  taking the supremum over $h \in [2^{-J-1},2^{-J}]$ in the last inequalities gives 
\begin{align*}
 \omegat_n(f_j, 2^{-J} ,\R^d)_p^p  & \le  C^p (2N)^{dp} 2^{(J-j)s_2p}\sum_{\lambda\in\Lambda_j} \Big (\frac{|c_\lambda|}{\mu^{(+\ep)}(\lambda)}\Big )^p  \omega_n(\psi_\lambda, 2^{-J} ,\R^d)_p^p,
\end{align*}
hence the conclusion by  \eqref{omeganpsila}.

\sk

By changing the constant $C$ in  \eqref{omeganpsila} into $C \|(\widetilde\ep_j)_{j\geq 0}\|_{\ell^\infty}$,   one gets   $\widetilde \ep_j \leq 1$. 
 \end{proof}

We are now in position to prove \eqref{eqnorm2}. Fix $\ep\in(0,1)$. Setting  $\widetilde f = f-\sum_{j=0}^{\infty} f_j$,  the triangle inequality yields
\begin{align}\label{triangle}
 \omegat_n(f, 2^{-J},\R^d)_p  \leq \omegat_n(\widetilde f , 2^{-J} ,\R^d)_p + \sum_{j=0}^{\infty} \omegat_n(f_j, 2^{-J} ,\R^d)_p  ,
\end{align}
and our goal is to control the $\ell_q$ norms of the   sequences $(u_J:=2^{Jd/p} \omegat_n(\widetilde f , 2^{-J} ,\R^d)_p)_{J\in \N}$ and $( v_J:=2^{Jd/p}  \sum_{j=0}^{\infty} \omegat_n(f_j, 2^{-J} ,\R^d)_p)_{J\in \N}$.

\medskip

$\bullet$ 
The terms $(u_J)_{J\geq 1}$  correspond  to low frequencies, and can be controlled as follows. Using property (P), one has $\mu(B(x,2^{-J})) \geq 2^{-J (s_2+\ep)}$ for every $x\in \R^d$, and so 
 \begin{equation}\label{trianglebis}
 u_J \le 2^{J(s_2+\ep+d/p)} \omega_n(\widetilde f , 2^{-J} ,\R^d)_p.
 \end{equation}
Observe that, since $\widetilde f$ is obtained by removing from $f$  the high frequency terms,  $\widetilde f \in B^{s',p}_{q}(\R^d)$ for all $s'\in(d/p,r)$ and $q\in [1,+\infty]$, as can be checked using \eqref{normbesovbeta}. In addition,   $|\widetilde f |_{(\mathcal L^d)^{\frac{s'+\ep+d/p}{d}-\frac{1}{p}},p,q}=|\widetilde f |_{(\mathcal L^d)^{\frac{s_1+\ep+d/p}{d}-\frac{1}{p}},p,q}=0$ since   the wavelet coefficients $c_\la(\widetilde f)$ of $\widetilde f$ vanish for all  $\la\in \La_j$, $j\geq 1$.

 Recalling  the decomposition  \eqref{decompf},   one notes  that the wavelet coefficients $( \beta(k))_{k\in \Z^d}$   in the wavelet expansions of  $f$ and $\widetilde f$ are identical. Hence, using  the equivalence of norms recalled after \eqref{normbesovbeta}, there is a constant $\widetilde C$ depending on $(d,\ep,\mu,p,q)$  (that may change from line to line) such that 
\begin{align*}
\|(u_J)_{J\in \N}\|_{\ell^q(\N)} \le \|\widetilde f \|_{B^{s_2+\ep+d/p,p}_q(\R^d)} &\le \widetilde C (\|\widetilde f\|_{L^p(\R^d)}+ |\widetilde f|_{(\mathcal L^d)^{\frac{s_2+\ep+d/p }{d}-\frac{1}{p}},p,q})\\
&=\widetilde C(\| \widetilde f\|_{L^p(\R^d)}+ |\widetilde f|_{(\mathcal L^d)^{\frac{s_1+\ep+d/p }{d}-\frac{1}{p}},p,q})\\
&\le \widetilde C(\| \beta(k)\|_{\ell_p(\Z^d)}+ |f|_{(\mathcal L^d)^{\frac{s_1+\ep+d/p }{d}-\frac{1}{p}},p,q})\\
&\le \widetilde C(\|f \|_{L^p(\R^d)}+ |f|_{(\mathcal L^d)^{\frac{s_1+\ep+d/p }{d}-\frac{1}{p}},p,q})\\
&\le  \widetilde C  (\|f\|_{L^p(\R^d)}+ |f|_{\mu^{(+\ep)},p,q}),
\end{align*}
where the last inequality is a consequence of property (P$_1$) (which implies that $\mu(\la) \leq C 2^{-js_1} = C\mathcal{L}^d(\la)^{ \frac{s_1 +d/p}{d}-1/p} $ for all $j\in\N$ and $\lambda\in \mathcal D_j$).

\medskip

$\bullet$ Next  the $\ell^q$ norm of $ (v_J)_{J\geq 1}$ is controlled.   Set $A_j =  \left\|  \left( \frac{  |c_{\lambda} | }{\mu^{(+\ep)} ({\lambda} )}  \right)_{\lambda\in \Lambda_j}\right\|_{\ell^p(\Lambda_j)}$. 

By Lemma \ref{besov_lemma2},   when $j\leq J$ one has $\omegat_n(f_j, 2^{-J} ,\R^d)_p  
 \leq C 2^{-jd/p} 2^{(j-J)(s-s_2)p }   A_j$, while   when $j> J$,  one has $\omegat_n(f_j, 2^{-J} ,\R^d)_p  
 \leq C 2^{-j d/p}    A_j $. Consequently,
 \begin{align*}
v_J   & \leq  C2^{Jd/p}  \sum_{j=0}^{J} 2^{-jd/p+(j-J)(s-s_2)} A_j  +  C 2^{Jd/p}\sum_{j=J+1}^{\infty} 2^{-jd/p}  A_j ,
\end{align*}
which implies that  $ \| ( v_J )_{J\geq 0} \| _{\ell^q(\N)}   \leq C  (\| (\alpha_J)_{J\geq 0}   \|_{\ell^q(\N)} +\| (\beta_J)_{J\geq 0}   \|_{\ell^q(\N)} )$, 
where 
$$\alpha_J:= 2^{Jd/p}  \sum_{j=0}^{J} 2^{-jd/p+(j-J)(s-s_2)} A_j  \ \ \mbox{ and } \ \ \beta_J:=    2^{Jd/p}\sum_{j=J+1}^{\infty} 2^{-jd/p} A_j .$$

Recall now the two following  Hardy's inequalities (see, e.g. (3.5.27) and (3.5.36) in \cite{cohen-book}): let $q\in [1,+\infty]$ as well as $0<\gamma<\delta$. There exists a constant $K>0$ such that :
\begin{itemize}
\item 
if  $(a_j)_{j\in\N}$ is a non negative sequence and for $J\in\N$ one sets $b_J= 2^{-\delta J}\sum_{ j=0}^{J} 2^{j\delta} a_j$, then  $\| (2^{\gamma J}b_J)_{J\geq 1}\|_{\ell^q(\N)} \le K \| (2^{\gamma J}a_J)_{J\geq 0}\|_{\ell^q(\N)}$.

\item
if $(a_j)_{j\in\N}$ is a non negative sequence and for $J\in\N$ one sets $b_J=\sum_{j\geq J} a_j$, then $\| (2^{\gamma J}b_J)_{J\in\N}\|_{\ell^q(\N)} \le K \| (2^{\gamma J}a_J)_{J\geq 0}\|_{\ell^q(\N)}$.  
\end{itemize} 

Let $\delta=s-s_2$ and $\gamma=d/p$. The first Hardy's inequality with $a_j=2^{-jd/p}  A_j$  yields 
$$\| (\alpha_j)_{j\in\N}   \|_{\ell^q(\N)} \leq K \| (A_j)_{j\in\N}\|_{\ell^q(\N)},$$ 
while   the second one with $a_j=2^{-jd/p}  A_j  $ and $\gamma =d/p$ gives
$$\| (\beta_J)_{J\in\N}   \|_{\ell^q(\N)} \leq K \| (A_j)_{J\in\N}\|_{\ell^q(\N)}.$$

Since $\| (A_J)_{J\in\N}\|_{\ell^q(\N) } =|f|_{\mu^{(+\ep)},p,q}  $, one concludes that 
\begin{align*}
 \|   (v_J)_{J\geq 1}  \|_{\ell^q(\N)}   \le   2CK (\|f\|_{L^p(\R^d)}+ |f|_{\mu^{(+\ep)},p,q}),
 \end{align*}
which, together with the control of  $\|   (u_J)_{J\geq 1}  \|_{\ell^q(\N)}$,  implies \eqref{eqnorm2}.
    
    \medskip
  
  Although we do not elaborate on this in this paper, it is certainly worth investigating the relationship between the Besov spaces in multifractal environment and the following analog of Sobolev space in multifractal environment.
 
\begin{definition}
\label{defWspmu} Let $\mu$ be a probability measure on $\R^d$, $s>0$, $p\geq 1$.
A function $f$ belongs to $ W^{\mu,s}_{p}(\R^d)$ if and only if $\|f\|_{W^{\mu,s}_{p}(\R^d)}<+\infty$, where   
$\|f\|_{W^{\mu,s}_{p}(\R^d)} = \|f\|_{L^{p}(\R^d)}+ |f|_{W^{\mu,s}_{p}(\R^d)}$ and 
$$ |f|_{W^{\mu,s}_{p}(\R^d)}:=\iint _{(\zud)^2} \frac{|f(x)-f(y)|^{p}}{\mu(B[x,y])^{sp} |x-y|^{2d}}\mathrm{d}x\mathrm{d}y <+\infty.$$
\end{definition}
  

\section{Main features of  the typical  singularity spectrum in $\widetilde B^{\mu,p}_{q}(\R^d)$}
\label{sec-description}

Given $\mu\in\mathscr E_d$, Theorem~\ref{main}(2) claims that the singularity spectrum of  typical functions  in $\widetilde B^{\mu,p}_{q}(\R^d)$ equals the Legendre transform $\zeta_{\mu,p}^*$ of   $\zeta_{\mu,p}$, which is explicitly given by \eqref{deftaumup} in terms of $\tau_\mu$.  In this section, we find an explicit formula for $\zeta_{\mu,p}^*$ in terms of $\tau_\mu^*$ ($=\sigma_\mu$) (Proposition~\ref{lemtetap}), and we discuss the possible shapes and features of $\zeta_{\mu,p}^*$ and  $\zeta_{\mu,p}$ (Sections~\ref{features} and~\ref{proofof5.1}).  
 
\subsection{Preliminaries and statements}

To express $\zeta_{\mu,p}^*$ in terms of $\tau_\mu^*$,   the following continuous and concave mapping  $\theta_p$ is introduced: 
\begin{equation}
\label{thetap}
\theta_p:\alpha\in [\tau_\mu'(+\infty),\tau_\mu'(-\infty)]\longmapsto \alpha+\frac{\tau_\mu^*(\alpha)}{p},
\end{equation}
see Figure \ref{fig-theta}. Notice that $\theta_\infty$ is just the identity map.  

\medskip

The concave sub-differential of a continuous concave function $g$ whose domain is a non trivial interval is well defined as the opposite $-\partial (-g)$ of the sub-differential $\partial (-g)$ of the convex function $-g$, and  is  denoted by $\overset{\,{_\smallfrown}}{\partial}g$.  

\medskip 

Let us briefly describe the variations of $\theta_p$, see Figure \ref{fig-theta} for an illustration.

\medskip
If $\tau_\mu'(+\infty)=\tau_\mu'(-\infty)$, then $\theta_p$ is constant and we set $\alpha_p=\tau_\mu'(-\infty)$. 

If $[\tau_\mu'(+\infty),\tau_\mu'(-\infty)]$ is non trivial, using the concavity of $\tau_\mu^*$, it is easily seen that the mapping $\theta_p$ is concave and reaches is maximum at $\alpha_p$, where 
$$
\alpha_p=
\begin{cases}
\min\big\{\alpha\in [\tau_\mu'(+\infty),\tau_\mu'(-\infty)] :\, -p\in \overset{\, _\smallfrown}{\partial}(\tau_\mu^*)(\alpha)\big\}&\text{if }-p \in \overset{\, _\smallfrown}{\partial}(\tau_\mu^*)\\
\tau_\mu'(-\infty)&\text{otherwise} 
\end{cases}
$$
(when  $\tau_\mu^*$ is differentiable and strictly concave,  $\alpha_p$ is the unique exponent $\alpha$ at which $(\tau_\mu^*)'(\alpha)=-p$ whenever it exists). Moreover,  $\theta_p$ is increasing on $[\tau_\mu'(+\infty),\alpha_p]$ and if $\alpha_p<\tau_\mu'(-\infty)$, then $\theta_p$ is constant over $[\alpha_p,\alpha'_p]$ and decreasing on $[\alpha'_p, \tau_\mu'(-\infty)]$, where $\alpha'_p=\max\{\alpha\in [\tau_\mu'(+\infty),\tau_\mu'(-\infty)] :\, -p\in \overset{\, _\smallfrown}{\partial}(\tau_\mu^*)(\alpha)\}$. Also, necessarily  $ \alpha_p \ge \tau_\mu'(0^+)$ since $\tau_\mu^*$ is increasing over the interval $[\tau_\mu'(+\infty),\tau_\mu'(0^+))$, and by Legendre duality, if $-p \in \overset{\, _\smallfrown}{\partial}(\tau_\mu^*)$, then  $\tau_\mu(-p)=(\tau_\mu^*)^*(-p)=-\alpha_p p-\tau_\mu^*(\alpha_p)=-p\theta_p(\alpha_p)$.

\medskip

Thus, in any case, the range of $\theta_p$ restricted to the interval $[\tau_\mu'(+\infty),\alpha_p]$ is the interval $ [\theta_p(\tau_\mu'(+\infty)),\theta_p(\alpha_p)]$, where

$$
\theta_p(\alpha_p)=\begin{cases}
\displaystyle \frac{\tau_\mu(-p)}{-p}&\text{if } -p \in \overset{\, _\smallfrown}{\partial}(\tau_\mu^*),\\
\displaystyle\tau_\mu'(-\infty)+\frac{\tau_\mu^*(\tau_\mu'(-\infty))}{p}&\text{otherwise}.
\end{cases}
$$

Note that according to Remark~\ref{Linearisation}, if  $-p \notin \overset{\, _\smallfrown}{\partial}(\tau_\mu^*)$, then $(\tau_\mu^*)'(\tau_\mu'(-\infty))>-\infty$ so that $\tau_\mu$ is linear near $-\infty$. This is also the case for $\zeta_{\mu,p}$, with the formula $\zeta_{\mu,p}(t)=(\tau_\mu'(-\infty)+\frac{\tau_\mu^*(\tau_\mu'(-\infty))}{p})t- \tau_\mu^*(\tau_\mu'(-\infty))
$.

%
\begin{figure}     \begin{tikzpicture}[xscale=1.7,yscale=1.4]
    {\tiny
\draw [->] (0,-0.2) -- (0,3.1) [radius=0.006] node [above] {$\theta_p(H)$};
\draw [->] (-0.2,0) -- (3.,0) node [right] {$H$};

\draw [thick, domain=0:5]  plot ({-(exp(\x*ln(1/5))*ln(0.2)+exp(\x*ln(0.8))*ln(0.8))/(ln(2)*(exp(\x*ln(1/5))+exp(\x*ln(0.8)) ) )} , {(-\x*( exp(\x*ln(1/5))*ln(0.2)+exp(\x*ln(0.8))*ln(0.8))/(ln(2)*(exp(\x*ln(1/5))+exp(\x*ln(0.8))))+ ln((exp(\x*ln(1/5))+exp(\x*ln(0.8))))/ln(2))/1-(exp(\x*ln(1/5))*ln(0.2)+exp(\x*ln(0.8))*ln(0.8))/(ln(2)*(exp(\x*ln(1/5))+exp(\x*ln(0.8)) ) )});

\draw [thick, domain=0:5]  plot ({-( ln(0.2)+ ln(0.8))/(ln(2)) +(exp(\x*ln(1/5))*ln(0.2)+exp(\x*ln(0.8))*ln(0.8))/(ln(2)*(exp(\x*ln(1/5))+exp(\x*ln(0.8)) ) )} , {-( ln(0.2)+ ln(0.8))/(ln(2)) +(exp(\x*ln(1/5))*ln(0.2)+exp(\x*ln(0.8))*ln(0.8))/(ln(2)*(exp(\x*ln(1/5))+exp(\x*ln(0.8)) ) )+(-\x*( exp(\x*ln(1/5))*ln(0.2)+exp(\x*ln(0.8))*ln(0.8))/(ln(2)*(exp(\x*ln(1/5))+exp(\x*ln(0.8))))+ ln((exp(\x*ln(1/5))+exp(\x*ln(0.8))))/ln(2))/1});  
\draw [fill] (-0.1,-0.10)   node [left] {$0$}; 

\draw [fill] (0,2.62) circle [radius=0.03] node [left] {$\theta_p(\alpha_p) \ $} [dashed]   (0,2.62) -- (1.9,2.62)  [fill] (1.9,2.62) circle [radius=0.03]   (1.9,2.62) -- (1.9,0)  [fill] (1.9,0) circle [radius=0.03]  node [below] {$\alpha_p $} ;

\draw [fill] (0,0.32)  circle [radius=0.03] node [left] {$\tau_\mu'(+\infty) \ $} [dashed]   (0,0.32 )  -- (0.32,0.32)   [fill] (0.32,0.32)  circle [radius=0.03]   (0.32,0.32)  -- ( 0.32,0)  [fill]  (0.32,0)  circle [radius=0.03]  node [below] {$\tau_\mu'(+\infty)$};

\draw [fill] (0,2.3)  circle [radius=0.03] node [left] {$\tau_\mu'(-\infty) \ $} [dashed]   (0,2.3 )  -- (2.3,2.3)   [fill] (2.3,2.3) circle [radius=0.03]   (2.3,2.3)  -- (2.3,0) [fill]  (2.3,0)  circle [radius=0.03]  node [below] {$ \ \ \tau_\mu'(-\infty)$};

  }
\end{tikzpicture}
\caption{The mapping $\theta_p$ when $\sigma_\mu(\tau_\mu'(+\infty)) = \sigma_\mu(\tau_\mu'(-\infty))=0$.}
\label{fig-theta}
\end{figure}
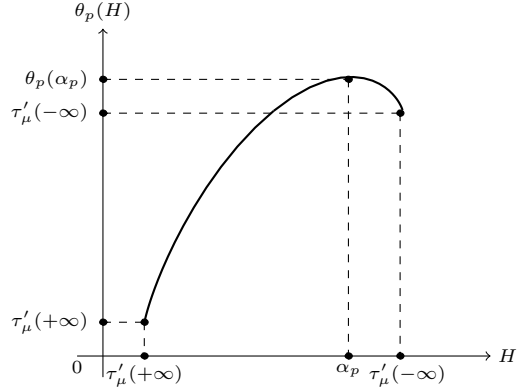

\medskip

Let $\theta_p^{-1}$ be the inverse branch of $\theta_p$ over $[\theta_p(\tau_\mu'(+\infty)),\theta_p(\alpha_p)]$, see Figure \ref{fig-theta}. The Legendre transform of $\zeta_{\mu,p}$  can be written as follows.

\begin{proposition}
\label{lemtetap}
Let $\mu\in\mathscr E_d$. One has
\begin{equation}\label{taumupstar}
\zeta_{\mu,p}^*(H)=
\begin{cases} p(H-\tau_\mu'(+\infty)) &\text{ if }H\in\big  [\tau_\mu'(+\infty), \theta_p(\tau_\mu'(+\infty))\big )\\
\tau_\mu^*(\theta_p^{-1}(H))&\text{ if }H\in [\theta_p(\tau_\mu'(+\infty)), \theta_p(\alpha_p)]\\
-\infty&\text{ if }H\not\in [\tau_\mu'(+\infty), \theta_p(\alpha_p)].
\end{cases}
\end{equation} 
\end{proposition}

The case $p=+\infty$ is trivial,  since as noticed in Remark \ref{rem-spectra},  $\zeta_{\mu,+\infty} =\tau_\mu$ and $\theta_\infty$ is the identity map.

\subsection{Main features of $\zeta_{\mu,p}$ and $\zeta_{\mu,p}^*$}\label{features}
 
These properties of  $\zeta_{\mu,p}$ and $\zeta_{\mu,p}^*$  follow from Proposition~\ref{lemtetap}, whose proof is  given in Section~\ref{proofof5.1}, or from the definition of $\zeta_{\mu,p}$. 

\mk
$\bullet$  As an immediate consequence of Proposition \ref{lemtetap},  $\tau_\mu'(+\infty)=\zeta_{\mu,p}'(+\infty)$ and $\theta_p(\alpha_p)=\zeta_{\mu,p}'(-\infty)$, although these equalities can be directly checked. Also, by definition of $\theta_p$, $\zeta_{\mu,p}'(-\infty)\le \tau_\mu'(-\infty)+\frac{d}{p}$. 
 
\mk
$\bullet$   When $p=+\infty$, $\zeta_{\mu,+\infty}\equiv \tau_\mu$.

\mk
$\bullet$  When $\tau_\mu^*( \tau_\mu'(+\infty)) =0$ (i.e. when $\theta_p(\tau_\mu'(+\infty)) = \tau_\mu'(+\infty)$),  the function $\zeta_{\mu,p}^*$ reduces to the  map $H\mapsto \tau_\mu^*(\theta_p^{-1}(H))$ on the interval $ [\theta_p(\tau_\mu'(+\infty)), \theta_p(\alpha_p)] $, see Figure~\ref{typical}.

\mk
$\bullet$   When $\tau_\mu^*( \tau_\mu'(+\infty)) >0$ and $p\in [1,+\infty)$,  (equivalently, when $\theta_p(\tau_\mu'(+\infty)) > \tau_\mu'(+\infty)$), $\zeta_{\mu,p}^*$ is linear over   $[\tau_\mu'(+\infty), \theta_p(\tau_\mu'(+\infty))\big )$. This occurs when   $\zeta_{\mu,p}$ is not differentiable at $p$, and in this case   $\zeta_{\mu,p}'(p^+)=\tau_\mu'(+\infty)$ and $\zeta_{\mu,p}'(p^-)=\theta_p(\tau_\mu'(+\infty))$.  

Note that this  affine  part in the singularity spectrum $\zeta_{\mu,p}^*$ of typical  functions $f \in \widetilde{B}^{\mu,p}_{q}(\R^d) $ follows from  the heterogeneous ubiquity property stated in Proposition \ref{ubiquity}.

\mk
$\bullet$   When $[\theta_p(\tau_\mu'(+\infty)), \theta_p(\alpha_p)]$ is non trivial, $\zeta_{\mu,p}^*$ is concave on this interval. 

\sk Moreover, using the notations of Remark~\ref{Linearisation}, $\zeta_{\mu,p}^*$ is differentiable at $\theta_p(\tau_\mu'(+\infty))$ if and only if $t_\infty=(\tau_\mu^*)'(\tau_\mu'(+\infty))=+\infty$. Otherwise, one has $(\zeta_{\mu,p}^*)'(\theta_p(\tau_\mu'(+\infty))^+)=\frac{t_\infty}{t_\infty+p} p<p=(\zeta_{\mu,p}^*)'(\theta_p(\tau_\mu'(+\infty))^-)$. This implies that $\zeta_{\mu,p}$ is affine over the interval $[\frac{t_\infty}{t_\infty+p} p,p]$, with slope $\theta_p(\tau_\mu'(+\infty))$. 

See Figures \ref{typical}  and \ref{fig_functions2} for  some examples of the shape of the spectrum of typical functions $f\in  \widetilde {B}^{\mu,p}_{q}(\R^d)$. 

\mk
$\bullet$  When  $ -p \not\in \overset{\, _\smallfrown}{\partial}(\tau_\mu^*)$, one has $t_{-\infty}>-\infty$, so both $\tau_{\mu}$ and $\zeta_{\mu,p}$ are affine near $-\infty$.

\subsection{Proof of Proposition~\ref{lemtetap}}
\label{proofof5.1}

The case $p=+\infty$ is trivial. Assume  $p\in[1,+\infty)$.

Let $\chi$ be  the mapping defined by the right hand side of  \eqref{taumupstar}. We are going to   prove that $\chi^* = \zeta_{\mu,p}$ (which is  defined by \eqref{deftaumup}). Next, the continuity and concavity of   $\chi$ is shown.   This and the Legendre duality    imply that $\zeta_{\mu,p}^*=\chi$.

Denote $[\tau_\mu'(+\infty),\tau_\mu'(-\infty)]$ by $[\alm ,\alpha_{\max}]$. It is convenient to write $\chi^* = \min( \zeta_1, \zeta_2)$ where, for $t\in\mathbb R$,
\begin{align*}
\zeta_1 (t) &=  \inf \{ tH-p(H -\alpha_{\min}) :\ {H\in [\alm , \theta_p(\alm)) }\}\\
\zeta_2 (t) &=  \inf \{ tH-\tau_\mu^*(\theta_p^{-1}(H)):\ {H\in [\theta_p(\alm), \theta_p(\alpha_p)]}\}.
\end{align*}

When  $t\neq p$,   set 
$$
t_p = \frac{pt}{p-t}.
$$
Then, whenever it exists, let  $\widetilde \alpha_{t_p} $ be the minimum of those real numbers $\alpha$ such that 
$$t_p  \in [(\tau_\mu^*)'(\alpha^+), (\tau_\mu^*)'(\alpha^-)] .$$
Otherwise, set $\widetilde \alpha_{t_p} =\alm$.

\subsubsection{Proof of the equality $\chi^* = \zeta_{\mu,p}$}

Recall that $\zeta_{\mu,p}$ is given by  formula \eqref{deftaumup}, i.e. $ \zeta_{\mu,p}(t)=
 \frac{p-t}{p}\tau_\mu\left (\frac{p}{p-t} t \right ) $ when $t<p $, and $ \zeta_{\mu,p}(t)= t \alm $ when $ t\geq p$.

\medskip

{\bf Case $t> p$.} In this case,   $t_p<-p$ (as shows a simple verification). Moreover, the mapping $H\mapsto  tH-p(H -\alpha_{\min})$ is increasing, hence $\zeta_1(t) = t\alm$. 
 
\smallskip
 
 Setting  $\alpha=\theta^{-1} _p(H)$ for $H\in [\theta_p(\alm), \theta_p(\alpha_p)]$,  one has 
\begin{align} 
\zeta_2 (t) &=  \inf _{\alpha \in [ \alm ,\alpha_p]}  \widetilde\chi_2(\alpha)  \ \mbox{ where } \ 
\label{defchit}
\widetilde\chi_2 (\alpha)  =  t \theta_p(\alpha)-\tau_\mu^*(\alpha).
\end{align}

Suppose that $\alpha_{\min}<\alpha_p$. Differentiating (formally) $ \widetilde\chi _2$ gives 
\begin{equation}\label{formalder} \widetilde\chi'_2(\alpha) = t+ \frac{t-p}{p} (\tau_\mu^*)'(\alpha) =\frac{t-p}{p}((\tau_\mu^*)'(\alpha)-t_p).
\end{equation}
Recall that  $\tau_\mu^*$ is concave, non-decreasing over $[\alpha_{\min},\tau_\mu'(0^+)]$ and non-increasing over  $[\tau_\mu'(0^+), \alpha_{\max}]$. Hence, by definition of $\alpha_p$,  $(\tau_\mu^*)'(\alpha^-)$ and $(\tau_\mu^*)'(\alpha^+)$ are both greater than $-p$ when $\alpha\in [\alpha_{\min},\alpha_p)$. 
So formula \eqref{formalder} and the fact that $t-p>0$ imply that the concave mapping $\widetilde\chi _2$ is non-decreasing over $[\alm, \alpha_p]$.   
Thus,  the infimum defining $ \zeta_2$ is reached at $\alm$, where it equals $t\alm +\frac{t-p}{p}\tau_\mu^*(\alm) \geq t\alm$. 

If $  \alpha_{p} =\alm$,  then $\zeta_2 (t) =t \theta_p(\alpha_{\min})-\tau_\mu^*(\alm) = t\alpha_{\min}$.

In both cases,   $\zeta_2(t)\geq t\alm$, and so $\chi^*(t) = \min(\zeta_1(t),\zeta_2(t)) = t\alm$, and \eqref{deftaumup} holds true. 

 \mk 

  The case $t=p$ follows by continuity.

\medskip
{\bf Case $t< p$.}    The mapping $H\mapsto  tH-p(H -\alpha_{\min})$ is non increasing, so $\zeta_1(t) =  (t-p)\theta_p(\alm) +p\alm= t\alm+ \frac{t-p}{p} \tau_\mu^*(\alm)$. 

Next we  determine  $\zeta_2(t)$. Since $t_p>-p$,   using \eqref{formalder}  and the fact that $t-p<0$ now shows that the convex mapping $\widetilde \chi_2$   reaches its minimum at $\widetilde \alpha_{t_p}$, which necessarily belongs to $[\alpha_{\min}, \alpha_p]$. Consequently,
\begin{align*} 
\zeta_2 (t) &=     t \theta_p(\widetilde \alpha_{t_p})-\tau_\mu^*(\widetilde \alpha_{t_p}).
 \end{align*}

Two subcases are distinguished:

\medskip
$\bullet$ Suppose   that  $t_p\le (\tau_\mu^*)'(\alpha_{\min}^+)$.

 In this case,  $\widetilde \alpha_{t_p}\geq \alm$, and one has  $\tau_\mu^*(\widetilde \alpha_{t_p}) = t_p \widetilde \alpha_{t_p} - \tau_\mu(t_p)  $ (even if $\widetilde \alpha_{t_p}= \alm$, because in this case $t_p=(\tau_\mu^*)'(\alpha_{\min}^+)=t_\infty$, hence $t_\infty< \infty$ and we can use Remark~\ref{Linearisation}). After simplification one gets 
$$\zeta_2 (t) =t\left( \widetilde \alpha_{t_p} +\frac{\tau_\mu^*(\widetilde \alpha_{t_p})}{p}\right) - \tau_\mu^*(\widetilde \alpha_{t_p} )  = \frac{p-t}{p}\tau_\mu(t_p) .$$

 If $\widetilde \alpha_{t_p} = \alm$,   then $\zeta_2 (t)  =     t \theta_p(\alm)-\tau_\mu^*(\alm) = t\alm +  (t-p)\tau_\mu^*(\alm) /p =  \zeta_1(t)$.  And a quick computation shows that $ t\alm +  (t-p)\tau_\mu^*(\alm) /p    =  \frac{p-t}{p}\tau_\mu\left (\frac{p}{p-t} t \right ) $.

If  $\widetilde \alpha_{t_p} > \alm$, then let us show that $\zeta_2 (t) \geq \zeta_1(t) $. Indeed,  this inequality reads  $ t\alm+ \frac{t-p}{p} \tau_\mu^*(\alm) \geq  \frac{p-t}{p}\tau_\mu(t_p)=  t \widetilde \alpha_{t_p}  - \frac{p-t}{p} \tau_\mu^* (\widetilde \alpha_{t_p} )$. The previous inequality is equivalent to   $t(\widetilde \alpha_{t_p}-\alm) \leq \frac{p-t}{p} (  \tau_\mu^* (\widetilde \alpha_{t_p}) -   \tau_\mu^* (\alm))$, i.e.
$$
\frac{ \tau_\mu^* (\widetilde \alpha_{t_p} ) -   \tau_\mu^* (\alm) }{\widetilde \alpha_{t_p}-\alm} \geq t_p .$$ 
The concavity of $\tau_\mu^*$ entails that this last inequality holds true.
     
  Hence, in all cases $\chi^*(t) = \min(\zeta_1(t),\zeta_2(t)) =  \frac{p-t}{p}\tau_\mu\left (\frac{p}{p-t} t \right ) $, so \eqref{deftaumup} holds.
%
%
%
%
%
%

\medskip
$\bullet$ Suppose   that  $t_p>(\tau_\mu^*)'(\alpha_{\min}^+)$.  In this case,  $t_\infty=(\tau_\mu^*)'(\alpha_{\min}^+)<+\infty$, which implies that  $\tau_\mu(t)=\alm t-\tau_\mu^*(\alm) $ for all $t\ge t_\infty$ (see Remark \ref{Linearisation}).

Also, since $t_p>(\tau_\mu^*)'(\alpha_{\min}^+)>0$, $\widetilde \alpha_{t_p}=\alpha_{\min}$ and the image of $\partial \widetilde \chi_2$ is included in  $(0,+\infty)$. In particular the convex mapping $ \widetilde\chi_2$ reaches its minimum at $\alm$. Consequently,   $\zeta_2(t)=\zeta_1(t)=t\alm+ \frac{t-p}{p} \tau_\mu^*(\alm)$. Since  $t_p \ge t_\infty$ and   $\tau_\mu$ is affine on $[t_\infty,+\infty)$, it follows that $\chi^*(t)= \frac{p-t}{p}\tau_\mu(t_p)$, as stated by \eqref{deftaumup}.

Note that the previous case corresponds to $\frac{t_\infty}{t_\infty+p} p<t<p$. In regard to the form taken by $\zeta_{\mu,p}^*$, it is convenient to rewrite $\zeta_{\mu,p}(t)= \theta_p(\alpha_{\min}) t-\tau_\mu^*(\alpha_{\min})$.

\subsubsection{Concavity  of $\chi $}

First, observe that $\chi $ is affine on the interval $[\alpha_{\min},\theta_p(\alpha_{\min})]$.

Let us   explain why $\chi$ is also concave over $[\theta_p(\alpha_{\min}),\theta_p(\alpha_p)]$.

Assume first that $\tau_\mu^*$ is differentiable over $(\alpha_{\min},\theta^{-1}_p(\alpha_p))$. Then this is also the case for $\theta_p^{-1}$ over  $(\theta_p(\alpha_{\min}),\theta_p(\alpha_p))$. For  $H\in (\theta_p(\alpha_{\min}),\theta_p(\alpha_p))$,  denoting $\alpha = \theta_p^{-1}(H)$ and $t= (\tau_\mu^*)'(\alpha)$, one gets $\chi'(H)=\frac{t}{1+t/p}$, which is increasing as a function of $t$. Since $H=\theta_p(\alpha)$ is an increasing function of $\alpha$ and $\alpha$ is a decreasing function of $t$, it follows that $\chi'$ is decreasing over $(\theta_p(\alpha_{\min}),\theta_p(\alpha_p))$. Hence $\chi$ is concave over $[\theta_p(\alpha_{\min}),\theta_p(\alpha_p)]$. If $\tau^*_\mu$ has  non differentiability points over $(\alpha_{\min},\theta^{-1}_p(\alpha_p))$, we get the same  conclusion by approximating $\tau^*_\mu$  by the differentiable $L^q$-spectra associated with the Bernoulli product generated by the probability vectors used to construct $\mu$.

\mk

Thus, one knows that $\chi$ is concave on the two intervals $[\alpha_{\min},\theta_p(\alpha_{\min})]$ and on $[\theta_p(\alpha_{\min}),\theta_p(\alpha_p)]$. If $\theta_p(\alpha_{\min}) = \theta_p(\alpha_p)$, or if $\theta_p(\alpha_{\min})=\alpha_{\min}$, the conclusion is immediate. Otherwise,   to get that $\chi$ is concave, one must check that   $\chi'(\theta_p(\alpha_{\min}^+))\le p=\chi'(\theta_p(\alpha_{\min}^-))$. With the notations used above, a direct computation then yields $\chi'(\theta_p(\alpha_{\min}^+))=p$ if $ (\tau_\mu^*)'(\alpha_{\min}^+)=t_\infty=+\infty$   and $\chi'(\theta_p(\alpha_{\min}^+))=\frac{t_\infty}{t_\infty+p} p$ if $t_\infty<+\infty$. Hence the conclusion that $\chi$ is concave.
 

\section{Lower bound for the $L^q$-spectrum, and upper bound for the  singularity spectrum in $\widetilde B^{\mu,p}_{q}(\R^d)$, when $\mu\in\MDs$}
\label{sec-upperbound}

This section    uses the  wavelet leaders and $L^q$-spectrum of a function introduced in Section~\ref{secMFF}.  Item (1) of Theorem~\ref{main} is proved by establishing a  general lower bound for the $L^q$-spectrum of all  $f\in \widetilde B^{\mu,p}_{q}(\R^d)$ when $\mu\in\MDs$ (Theorem~\ref{validity}(1)). 
 
The main result of this section is the following.  Recall the definition \eqref{rmu} of $s_\mu$.
\begin{theorem}\label{LBzeta}
Let $\mu\in\MDs$ and $p,q\in[1,+\infty]$. Let $\Psi\in \mathcal F_{s_\mu}$.  For all $f\in L^p(\R^d)$ such that $|f_{\mu,p,q}|<+\infty$, one has ${\zeta_f}_{|\R_+}\ge {\zeta_{\mu,p}}_{|\R_+}$. 
\end{theorem}
It is implicit in Theorem \ref{LBzeta} that the semi-norm $|f_{\mu,p,q}|$ defined in  \eqref{defbesovwavelet} is computed using the wavelet $\Psi\in \mathcal F_{s_\mu}$  fixed  by the statement.

Theorem \ref{LBzeta} yields the following corollary.
\begin{corollary}
Let $\mu\in\MDs$ and $p,q\in[1,+\infty]$. For all $f\in \widetilde B^{\mu,p}_q(\R^d)$, one has:
\begin{enumerate}
\item  ${\zeta_f}_{|\R_+}\ge {\zeta_{\mu,p}}_{|\R_+}$, i.e. the claim of Theorem~\ref{validity}(1) holds true.
\item For all $H\in\R$, 
\begin{align*}
\sigma_f(H)\le 
\begin{cases} \zeta_{\mu,p}^*(H)&\text{ if }H\le \zeta_{\mu,p}'(0^+)\\
\ \ \ d &\text{ if } H> \zeta_{\mu,p}'(0^+)
\end{cases},
\end{align*}
i.e. part (1) of Theorem~\ref{main} holds true. 
\end{enumerate}
\end{corollary}
\begin{proof} Part (1) follows from the definition of $\widetilde B^{\mu,p}_q(\R^d)$ and the continuity of ${\zeta_{\mu^{(-\ep)},p}}_{|\R_+}$ as a function of $\ep$. Part (2) is then a consequence of \eqref{UPF}. 
\end{proof}


The wavelets   $\Psi\in \mathcal F_{s_\mu}$ are fixed  for the rest of this section.

To obtain Theorem~\ref{LBzeta}, one needs to   estimate, for any  $f\in L^p(\R^d)$ such that $|f_{\mu,p,q}|<+\infty$ and any $N\in\N$,  the upper large deviations spectrum of the wavelet leaders $(L_\lambda^f)_{\lambda\subset N[0,1]^d}$ associated with     $\Psi$, defined as follows. Recall the notations $H\pm\ep$ introduced in  \eqref{defpmep}, and $N\lambda$ at the beginning of Section~\ref{statement}.

\begin{definition}
\label{defEN}
Let   $f\in L^1_{\rm loc}(\R^d)$ and $N\in \N^*$, with wavelet coefficients and   leaders  computed with the wavelet $\Psi$. For any compact subinterval $I$ of $\R$, set
\begin{equation*}
\label{defEF}
\mathcal{D}^{N}_{f}(j,I) =
  \left\{ \lambda\in \mathcal D_j: \lambda\subset  N[0,1]^d ,\, \frac{\log_2 |L^f_\lambda|}{ {-j}} \in
I\right\},
\end{equation*} 

\noindent 
The  upper wavelet leaders large deviation spectrum of $f$ associated with  $\Psi$ and $N[0,1]^d$ is  
\begin{eqnarray*} 
\overline  {\sigma}^{{\rm LD},N}_{f}(H)  & =  & \lim_{\ep\to 0}\limsup_{j\to + \infty} \frac{\log_2 \#\mathcal{D}^{N}_f(j, H\pm\ep)}{j}.
\end{eqnarray*}
\end{definition}
\begin{proposition}\label{ubLD}
Let $\mu\in\MDs$ and $p,q\in[1,+\infty]$. For all $f\in L^p(\R^d)$ such that $|f_{\mu,p,q}|<+\infty$, and all $N\in \N$, one has 
\begin{equation}\label{LDC}
\overline  {\sigma}^{{\rm LD},N}_{f}(H)  \le \begin{cases} \zeta_{\mu,p}^*(H)&\text{ if }H\le \zeta_{\mu,p}'(0^+)\\
d&\text{ if } H> \zeta_{\mu,p}'(0^+)
\end{cases}.
\end{equation}
\end{proposition}
Assuming that  Proposition \ref{ubLD} is proved, let us    explain how Theorem~\ref{LBzeta} follows. 

\begin{proof}[Proof of Theorem~\ref{LBzeta}] Note that  by large deviations theory \cite{DemboZeitouni},  $\zeta_f^{N,\Psi}$ defined in \eqref{zetaNPsi} is the Legendre transform of the concave hull of $\overline  {\sigma}^{{\rm LD}, N }_{f}$. By  Proposition~\ref{ubLD}, this concave hull is dominated by  the right hand-side of \eqref{LDC}. It is easily seen    that this right-hand side, as a function of $H$,  is  concave, and that its  Legendre transform is  equal to ${\zeta_{\mu,p}}_{|\R_+}$ over $\R_+$ and equal to $-\infty$ over $\R^*_+$. Consequently, ${\zeta_f^{N,\Psi}}_{|\R_+}\ge {\zeta_{\mu,p}}_{|\R_+}$, which allows to conclude since ${\zeta_f^{\Psi}}_{|\R_+}= \lim_{N\to +\infty} { \zeta^{N,\Psi}_f }_{|\R_+}$ does not depend on $\Psi$. 
\end{proof}

The rest of this section is devoted to the proof of Proposition~\ref{ubLD}. It  requires    large deviations estimates on the distribution of the wavelet coefficients of $f$ under the constraint   $|f_{\mu,p,+\infty}|<+\infty$, which holds automatically if $|f_{\mu,p,q}|<+\infty$.

\subsection{Large deviations estimates   for wavelet coefficients}\label{secub1}

\begin{definition}
Let $\mu\in \mathcal C(\R^d)$,  $I_H$ and  $I_\alpha$ be two compact subintervals of $\R$, and $f\in L^1_{\rm{loc}}(\R^d)$ with wavelet coefficients($c_\lambda)_{\lambda\in\Lambda}$. Then, define  
\begin{equation}
\label{defLA}
\Lambda_{f,\mu}(j,I_H,I_\alpha)  = \left\{ \lambda=(i,j,k)\in\Lambda:\lambda_{j,k}\subset 3[0,1]^d,\  \begin{cases}  \ds \frac{\log_2 |c_\lambda|}{ {-j}} \in
I_H\\  \ds \frac{\log_2\mu(\lambda_{j,k}) }{ {-j}} \in
I_\alpha \end{cases} \right\}.
\end{equation}
\end{definition}

In other words, $\Lambda_{f,\mu}(j,I_H,I_\alpha)$ contains those cubes $\la$ of generation $j$ such that  $\mu(\la) \sim |\la|^\alpha$ with $\alpha\in I_\alpha$ and $|c_\la|\sim 2^{-jh}$  with $h\in I_H$. The cube  $3[0,1]^d$ is considered, rather than $[0,1]^d$ because the computation of wavelet leaders on $\zud$  requires some knowledge of $\mu$ and $f$  in this neighborhood of~$\zud$.

\sk

The cardinality of $\Lambda_{f,\mu}(j,I_H,I_\alpha)$ is estimated to get a control of the wavelets leaders large deviations spectrum under the assumptions of Proposition~\ref{ubLD}. 

In the next lemma, the convention $\infty\times  x=+\infty$ for $x\ge 0$ is adopted.  
 \begin{lemma}
 \label{propmajnb}
Let $\mu\in\MDs$ and $p\in[1,+\infty]$. Let $\alpha_{\min} = \tau'_\mu(+\infty)$ and $\alpha_{\max}= \tau'_\mu(-\infty)$. Let $f\in L^p(\R^d)$ be such that $ |f|_{\mu,p,+\infty} <+\infty $ and let $I_H, I_\alpha$ be two compact subintervals of~$\R$. 
\begin{enumerate}

\item
If  $\max I_H <  \min I_\alpha $, then $\Lambda_{f,\mu}(j,I_H,I_\alpha) =\emptyset$ for $j$ large enough.

\item
If $I_\alpha \subset [\alpha_{\min},\alpha_{\max}]$ and $\min I_\alpha \leq \min I_H$, then for every $\eta>0$, there exists $\ep_0>0$ and $J_0\in \N$ such that for every $\ep\in [0,\ep_0]$ and $j\geq J_0$: 
\begin{equation}
\label{majemuf}
  \frac{ \log_2 \# \Lambda_{f,\mu}(j,I_H\pm\ep,I_\alpha\pm\ep) }{ j}  \leq  \max_{\beta\in I_\alpha\cap [0,\max I_H]}\min (p(\max I_H -  \beta),{\tau_\mu^*}(\beta))+ \eta.
\end{equation} 
\end{enumerate}

 \end{lemma}

\begin{proof} 

We treat the  case $p<+\infty$ and leave the simpler case $p=+\infty$ to the reader.

\medskip

(1) Recall that by definition  $\sup_{j\in\N}\Big \| \Big (\frac{c_\lambda}{ \mu(\lambda)}  \Big )_{\lambda\in \Lambda_j}\Big \|_{\ell_p(\Lambda_j) }  < + \infty$. There is $C_f\ge 1$ such that 
\begin{equation}
\label{defbmui}
\sup_{j\in\N}  \Big \| \Big (\frac{c_\lambda}{ \mu(\lambda)}  \Big )_{\lambda\in \Lambda_j}\Big \|_{\ell_p(\Lambda_j) } \le C_f.
\end{equation}
It follows that item  (1)  holds true, for otherwise \eqref{defbmui} would be contradicted. 

\medskip

\noindent (2)  Fix $\eta,\ep>0$ and set $\widetilde H=\max (I_H)$. Since $I_\alpha$ is compact and $\tau_\mu^*$ is continuous over its compact domain, there are finitely many numbers $\alpha_0<\ldots<\alpha_m$ such that $I_\alpha=\bigcup _{\ell=0}^{m-1} [\alpha_\ell,\alpha_{\ell+1}]$  and for  every $\ell$, $\alpha_{\ell+1}-\alpha_\ell \leq \eta/p$ and  $ |   {\tau_\mu^*}(\beta)- {\tau_\mu^*}(\beta')|\leq   \eta$ for all $\beta,\beta'\in  [\alpha_\ell,\alpha_{\ell+1}]$.

Let $j\in\N$. Consider the subset $ \Lambda_{f,\mu}(j,I_H, [\alpha_\ell,\alpha_{\ell+1}]\pm\ep)$ of $ \Lambda_{f,\mu}(j,I_H\pm\ep,I_\alpha\pm\ep)$. With  each cube $\lambda \in \Lambda_{f,\mu}(j,I_H\pm\ep, [\alpha_\ell,\alpha_{\ell+1}]\pm\ep)$  is associated a wavelet coefficient $c_\la$ whose absolute value is at least equal   to $2^{-j(\widetilde H+\ep)}$.  Thus, for each $\ell \in \{0,..., m-1\}$, 
\begin{equation}\label{Cf}
C^p_f\geq  \sum_{\lambda\in\Lambda_j} \left(\frac{\left| c_\lambda \right|}{ \mu(\lambda)} \right)^p  \geq  \sum_{\lambda\in \Lambda_{f,\mu}(j,I_H\pm\ep, [\alpha_\ell,\alpha_{\ell+1}]\pm\ep) } \left(\frac{ 2^{-j(\widetilde H+\ep)}}{ 2^{-j  (\alpha_{\ell}-\ep)}} \right)^p.
\end{equation}

\begin{remark}\label{rem66}
Recall that for  $\la=(i,j,k)\in \Lambda _j$,   we make a slight abuse of notation by identifying  $\la$ with $\la_{j,k} \in \mathcal{D}_j$  and writing    $\mu(\la) $ for $ \mu(\la_{j,k})$ and  $\la\subset  E$ for $\la_{j,k}\subset E$. 
\end{remark}

It follows from \eqref{Cf} that 
$$  
 \#  \Lambda_{f,\mu}(j,I_H, [\alpha_\ell,\alpha_{\ell+1}]\pm\ep)  \leq  C^p_f2^{jp (\widetilde H-\alpha_\ell+2\ep)}.
$$
On the other hand, observe  that for each $j\ge 0$,   one has 
 $$
 \Lambda_{f,\mu}(j,I_H\pm\ep,[\alpha_\ell,\alpha_{\ell+1}]\pm\ep ) \subset \Big \{\lambda=(i,j,k)\in\Lambda:\, \lambda \subset3[0,1]^d, \   \frac{\log_2 \mu(\lambda )}{-j}\in I \Big \},
 $$
where $I= [\alpha_\ell,\alpha_{\ell+1}]\pm\ep\cap [0,\widetilde H+\ep]$.  Applying  Proposition \ref{fm}(4) to each interval $[\alpha_\ell,\alpha_{\ell+1}]\pm\ep \cap [0,\widetilde H+\ep] $, one finds   $\ep_0>0$ and $J_0\in\N$ such that for all $\ep\in(0,\ep_0]$, $0\le \ell\le m-1$ and $j\geq J_0$, 
$$ 
  \#  \mathcal{D}_\mu(j,[\alpha_\ell,\alpha_{\ell+1}]\pm\ep\cap [0,\widetilde H+ \ep]   )\le   \#  \mathcal{D}_\mu(j,([\alpha_\ell,\alpha_{\ell+1}]\cap [0,\widetilde H]) \pm2\ep  ) \leq  2^{j (\gamma_\ell+\eta)},
 $$
where $\gamma_\ell=\max\{ {\tau_\mu^*}( \beta):\beta\in [\alpha_\ell,\alpha_{\ell+1}]\cap [0,\widetilde H]\}$. 
Then, taking into account the fact that $\mu$ is $\Z^d$-invariant, as well as the fact that with each dyadic cube $\lambda_{j,k}$ are associated $2^d-1$ wavelet coefficients, one  obtains
$$
\#\Lambda_{f,\mu}(j,I_H\pm\ep,[\alpha_\ell,\alpha_{\ell+1}] \pm\ep)\le 3^d(2^d-1)2^{j (\gamma_\ell+\eta)}.
$$
Combining  the previous estimates,  one gets for  $\ep\in(0,\ep_0]$ and $j\ge J_0$
  \begin{align*}
 \#  \Lambda_{f,\mu}(j,I_H,I_\alpha\pm\ep) & \le \sum_{\ell=0}^{m-1}\# \Lambda_{f,\mu}(j,I_H,[\alpha_\ell,\alpha_{\ell+1}]\pm\ep )\\
 & \leq   \sum_{\ell=0}^{m-1} \min\big( C^p_f 2^{jp (\widetilde H-  \alpha_\ell +2\ep)}, 3^d(2^d-1)\cdot 2^{j(\gamma_\ell+\eta)}\big)\\
 &  \leq   3^d(2^d-1)C^p_f\,  m \max \big\{2^{j \min ( p (\widetilde H- \alpha_\ell+2\ep), \gamma_\ell +\eta)}: \ell=0,1,...,m-1\big \}  .
 \end{align*} 
 Also, the constraints imposed to the exponents  $\alpha_\ell$ and the continuity of $\tau_\mu^*$ imply that  
 \begin{align*}&  \max \big\{ \min( p (\widetilde H- \alpha_\ell+2\ep),  \gamma_\ell+\eta):\ell=0,1,...,m-1\big \} \\
 &\leq \max\big\{ \min (p(\widetilde H -  \beta),{\tau_\mu^*}(\beta)): \beta\in I_\alpha\cap [0, \widetilde H]\big\}+2p\ep+3\eta.
  \end{align*}
Taking  $\ep_0\le \eta/p$ and  $J_0$ so large that $2^{J_0\eta} \geq 3^d (2^d-1)C^p_f m$, we finally get the desired upper bound 
 \eqref{majemuf} (with $6\eta$ instead of $\eta$).
\end{proof}
  
 We are now ready to get an upper bound for the wavelet leaders upper large deviations spectrum of $f$.


\subsection{Proof of Proposition~\ref{ubLD}}
\label{secub2}

Note that  since $\mu$ is $\Z^d$-invariant, and by definition of $|\ |_{\mu,p,q}$, any general upper bound for $\overline  {\sigma}^{{\rm LD},1}_{f_{|\zu^d}  }$ holds for  $\overline  {\sigma}^{{\rm LD},N}_{f}$.  Thus, without loss of generality we prove that  $\overline  {\sigma}^{{\rm LD},1}_{f}$ is upper bounded by the right hand side of  \eqref{LDC}.

\medskip

This proof is rather involved because all the possible interactions between the values $\mu(\la)$ and the corresponding wavelet coefficients $c_\la$  must be taken care of.  

\medskip

Note that the inequality $\overline  {\sigma}^{{\rm LD},1}_{f}  \le d$ obviously holds. So it is enough to deal with the case $H\le \zeta'_{\mu,p}(0^+)$. 

\medskip

Fix $H\le \zeta'_{\mu,p}(0^+)$. For $\ep>0$ small enough,   $\#\mathcal D^1_f(j,H\pm\ep)$  is going to be estimated from above (recall Definition \ref{defEN}). We are going to prove that   there exist $C,c>0$ such that for any $\eta>0$,  if $\ep_0\in (0,\eta]$ is chosen small enough, then for $j$ large enough, for all $\ep\in (0,\ep_0)$, 
\begin{equation}\label{estimfinale}
\#\mathcal D_f^1(j,H\pm\ep)\le C j 2^{j(\zeta_{\mu,p}^*(H)+c\eta)}. 
\end{equation}

 It is immediate to check that \eqref{estimfinale} implies \eqref{LDC}, hence Proposition~\ref{ubLD}.
 
 \medskip

Since $|f|_{\mu,p,+\infty}<+ \infty$, there exists $C>0$ such that  $|c_\lambda|\le C \mu(\lambda)$   for every $\la\in \bigcup_{j\ge 0} \Lambda_j$  (recall Remark~\ref{rem66}). Without loss of generality,   suppose that the above constant is equal to 1 and so
\begin{equation}
\label{majcla}
|c_\lambda|\le \mu(\lambda) \mbox{ for every }\lambda\in \bigcup_{j\ge 0} \Lambda_j.
\end{equation}

Recall  the definition \eqref{defleaders} of  wavelet   leaders:  $L^f_\lambda=\sup\{|c_{\lambda'}|: \lambda'=(i,j,k)\in \Lambda, \,\lambda' \subset 3\lambda\}.$
The following   observations are key.
\begin{lemma}
\label{lemleader}
A dyadic cube $\lambda$ belongs to $\mathcal{D}_f^1(j,H\pm\ep)$ if and only if:
\begin{itemize}
\item  $\lambda\subset [0,1]^d$;
\item
 There exists a dyadic cube $\lambda'\subset 3\lambda$ of generation $j' \ge j$  as well as $i\in \{1,\cdots 2^d-1\}$ and $k'\in \Z^d$ such that $\lambda'=\lambda_{j',k'}$, and $|c_{(i,j',k')}|=2^{-j'H'}$ with $H'\in\frac{j}{j'}[H-\ep,H+\ep]$;
 \item when $j$ is large enough,    $j'\le 2j(H+\ep) /\alpha_{\min}$. 
 \end{itemize}
 \end{lemma}
 \begin{proof}
 The first item is trivial, and the second one follows from the definition \eqref{defleaders} of the wavelet leaders and  the fact that $\frac{\log_2 |L^f_\lambda|}{ {-j}} \in H\pm\ep$ if and only if there exists some $\lambda'\subset 3\lambda$ of generation $j' \ge j$ and $i\in \{1,\cdots 2^d-1\}$  such that $ \frac{\log_2 |c_{(i,j',k')}|}{ {-j}} \in H\pm\ep$.
 
 For the third item,  Lemma~\ref{propmajnb}(2) implies that   $|c_{(i,j',k')}|\le 2^{-j'\alpha_{\min}/2}$ when $j$ (and so $j'$) is large. Hence  $H'\ge \alpha_{\min}/2$ and the fact that $j'\leq j(H+\ep)/H'$ implies the claim.
 \end{proof}
 
The second item of Lemma \ref{lemleader} is used repeatedly in the forthcoming pages.

\medskip Three cases are separated.

\medskip

\noindent
{\bf Case 1 : $H<\alpha_{\min}$.}

 Note that $\zeta_{\mu,p}^*(H)=-\infty$. Suppose that  $\ep>0$ is so small that $\alpha_{\min}-\ep>H+\ep$.  Due to Proposition~\ref{fm}(5), and the observation  made just above, for $j$ large enough  
 $$
 \#\mathcal{D}^1_f(j,H\pm\ep)\le \sum_{j\le j'\le 2j(H+\ep)/\alpha_{\min}}  \#\Lambda_{f,\mu}(j', [0,H+\ep], I_\alpha),
 $$ 
 with $I_\alpha =[\alpha_{\min}-\ep,\alpha_{\max}+\ep]$. However, $H+\ep < \alpha_{\min}-\ep$, so by Lemma~\ref{propmajnb}, $\mathcal{D}^1_f(j,H\pm\ep)=\emptyset$. This implies \eqref{LDC}, i.e.   $\overline  {\sigma}^{{\rm LD},1}_{f_{|\zu^d}  }(H)=-\infty$.

\medskip

To deal with the other cases, we  discretize  the interval $[\alpha_{\min},H]$.

\sk  Fix $\eta>0$,  $\ep_0\in (0, \min(1/2,\alpha_{\min}/2,\eta))$,  and split the interval $ [\alpha_{\min}, H]$ into finitely many contiguous closed intervals $I_1, ..., I_m$ ($m=m(\ep_0)$) such that 
\begin{itemize}
\item $|I_\ell|\leq \ep_0$ for every $\ell \in \{1,...,m\}$,
\item  Writing $I_\ell=[h_\ell,h_{\ell+1}]$,  one has $ 1\le h_{\ell+1}/h_\ell \leq 1+\ep_0$  for every $1\le \ell \le m$. 
\end{itemize}
In particular, $H/h_\ell \geq 1$ for every $\ell$.

\sk

By Lemma \ref{lemleader}, if $j\ge J_0$ and $\lambda\in\mathcal{D}^1 _f(j,H\pm\ep)$,  there exist $j'\ge j$ and $\la'=(i,j',k') \in \La_{j'}$   such that $\lambda'\subset 3\lambda$   and $|c_{\la'}|=2^{-j'H'}$ with $H'\in\frac{j}{j'}[H\pm\ep]$.  By \eqref{majcla}, $|c_{\la'}|\leq \mu(\la')$, so  there exist $1\le \ell'\le \ell\le m$ such that $ \la' \in  \La_{f,\mu}(j',I_{\ell}\pm\ep, I_{\ell'}\pm\ep)$ (recall \eqref{defLA}).

In addition,  $H'\in I_\ell\pm\ep\subset I_\ell\pm\ep_0$, $j'\in \frac{j}{H'}[H\pm\ep]\subset \left[j\frac{H-\ep_0}{h_{\ell+1}+\ep_0},j\frac{H+\ep_0}{h_\ell-\ep_0}\right]$,  and $h_{\ell+1}\le H$. Consequently, 
\begin{equation}
\label{discret2}
\mathcal{D}^1_f(j,H\pm\ep)\subset\bigcup_{1\le\ell'\le \ell\le m}\mathcal{D}^{\ell,\ell'}_f(j,H\pm\ep),
\end{equation}
where (recall Remark~\ref{rem66})
$$
\mathcal{D}^{\ell,\ell'}_f(j,H\pm\ep) \!\!=\bigcup_{ {j'}  \in  j\cdot\left[\frac{H-\ep_0}{h_{\ell+1}+\ep_0},\frac{H+\ep_0}{h_\ell-\ep_0}\right]}\left\{\lambda\in\mathcal D_j \cap \zud :\, \begin{cases}  \exists \,  \la' \in \Lambda_{f,\mu}(j',I_{\ell}\pm\ep, I_{\ell'}\pm\ep) \\ \mbox{ such that  } \la' \subset 3\lambda \!\!\!\!  \end{cases}\right\}.
$$

Next, the cardinality of $\mathcal{D}^{\ell,\ell'}_f(j,H\pm\ep)$ (and thus of $\mathcal{D}^1_f(j,H\pm\ep)$)   is going to be bounded from above using different estimates.
 
To do so,   Lemma~\ref{propmajnb}(2)  is applied to each pair $\{I_\ell,I_{\ell'}\}$:  there exist  $\ep\in (0,\ep_0)$ and $J_0\in \N$ such that for all  $j ' \ge J_0$, for all $1\le \ell'\le \ell\le m$, 
\begin{equation}\label{discret}
\frac{\log_2 \#\Lambda_{f,\mu}(j',I_{\ell}\pm\ep, I_{\ell'}\pm\ep)}{j'}\le  d(\ell,\ell')+\eta
\end{equation}
where 
\begin{equation}\label{dll'}
d(\ell,\ell')=\max\left\{\min(p(h_{\ell+1} -   \beta),\tau_\mu^*(\beta)): \beta\in I_{\ell'}\right \}.
\end{equation}

\medskip
\noindent
{\bf Case 2: $\alpha_{\min}\le H <\theta_p(\alpha_{\min})=\alpha_{\min}+\frac{\tau_\mu^*(\alpha_{\min})}{p}$.} This case occurs only when $\tau_\mu^*(\alpha_{\min})>0$. 

\sk Let $j\ge J_0$. For every  $1\le \ell'\le \ell\le m$,   one has    $ p(h_{\ell+1}-h_{\ell'})\le p(H-\alm)  \leq \tau_\mu^*(\alm)  \leq \tau_\mu^*(\beta)$, for every $\beta \in I_{\ell'}$. So,    from \eqref{dll'} one deduces that $d(\ell,\ell')\le  p(h_{\ell+1}-\alpha_{\min})$.  Thus, if $j'\in  \left[j\frac{H-\ep_0}{h_{\ell+1}+\ep_0},j\frac{H+\ep_0}{h_\ell-\ep_0}\right]$, then  $j'd(\ell,\ell')\le j p (H+\ep_0) \frac{h_{\ell+1}-\alpha_{\min}}{h_\ell- \ep_0}  $.   Then observing that $\sup_{\ell\in\{1,...,m\}} \frac{h_{\ell+1}-\alpha_{\min}}{h_\ell}   = \frac{H-\alm}{H} + O(\ep_0)$, one has
$$j'(d(\ell,\ell')+\eta)\le j(p(H-\alpha_{\min})+O(\ep_0)+\eta)=j(\zeta_{\mu,p}^*(H)+O(\ep_0)+\eta).$$
 Consequently, since~\eqref{discret2} implies
\begin{equation*}\label{discret3}
\#\mathcal{D}^1_f(j,H\pm\ep)\le \sum_{1\le\ell'\le \ell\le m}\sum_{j'\in \left[j\frac{H-\ep_0}{h_{\ell+1}+\ep_0},j\frac{H+\ep_0}{h_\ell-\ep_0}\right]} \# \Lambda_{f,\mu}(j',I_{\ell}\pm\ep, I_{\ell'}\pm\ep),
\end{equation*}
the inequality~\eqref{discret}  combined with the previous remarks yields
$$
\#\mathcal{D}^1_f(j,H\pm\ep)\le m^2 j\frac{H+\ep_0}{\alpha_{\min}-\ep_0}2^{j(\zeta_{\mu,p}^*(H)+O(\ep_0)+\eta)}
 = C 2^{j(\zeta_{\mu,p}^*(H)+O(\ep_0)+\eta)},
$$
so \eqref{estimfinale} holds true.

\medskip

\noindent
{\bf Case 3: $\theta_p(\alpha_{\min})\le H \le  \zeta_{\mu,p}'(0^+)=\theta_p(\tau_\mu'(0^+))$.}

  This case is  divided into four subcases in order to estimate $\#\mathcal{D}^{\ell,\ell'}_f(j,H\pm\ep)$. 

 The term  $d(\ell,\ell')$ can easily be expressed in terms of the mappings $\theta_p$ defined  in \eqref{thetap} and $\tau_\mu^*$. The mapping $\theta_p$ is an increasing map over  $[\alpha_{\min}, \alpha_p]$ and $\alpha_p\geq \tau_\mu'(0^+)$, so using that $h_\ell\le H$, one deduces that
\begin{eqnarray}
\label{eq15}
d(\ell,\ell')\ = \,  \begin{cases}
\    {\tau_\mu^*}(h_{\ell'+1})& \mbox{ if  } h_{\ell'+1  } \leq  \theta_p^{-1}(h_{\ell+1}),\\
\   p(h_{\ell+1} -  h_{\ell'}) & \mbox{ if } h_{\ell'} \geq \theta_p^{-1}(h_{\ell+1}), \\
\    \tau_\mu^*(\theta_p^{-1}(h_{\ell+1}))=\zeta_{\mu,p}^*(h_{\ell+1}) & \mbox{ otherwise. } 
\end{cases}
\end{eqnarray}
Moreover, the maximum of the  three possible values is always  $\zeta_{\mu,p}^*(h_{\ell+1})$.

 \medskip
\noindent
{\bf Subcase $(3a)$:} $\frac{H}{h_{\ell+1}}d(\ell,\ell')\le \zeta_{\mu,p}^*(H)$. 
Using the definition of $\mathcal{D}^{\ell,\ell'}_f(j,H\pm\ep)$ and inequality~\eqref{discret},   for $j\ge J_0$  
\begin{align*}
\#\mathcal{D}^{\ell,\ell'}_f(j,H\pm\ep)&\le\sum_{j'\in \left[j\frac{H-\ep_0}{h_{\ell+1}+\ep_0},j\frac{H+\ep_0}{h_\ell-\ep_0}\right]} \# \Lambda_{f,\mu}(j',I_{\ell}\pm\ep, I_{\ell'}\pm\ep) \\
&\le \sum_{j'\in \left[j\frac{H-\ep_0}{h_{\ell+1}+\ep_0},j\frac{H+\ep_0}{h_\ell-\ep_0}\right]} 2^{j' (d(\ell,\ell')+\eta)} \le  j\frac{H+\ep_0}{h_\ell-\ep_0} 2^{ j\frac{H+\ep_0}{h_\ell-\ep_0}(d(\ell,\ell')+\eta)} .
\end{align*}
By our assumption,  $ \frac{H+\ep_0}{h_{\ell}-\ep_0}  d(\ell,\ell') \leq \Big (\frac{H }{h_{\ell+1}} +O(\ep_0)\Big )  d(\ell,\ell') \leq  \zeta_{\mu,p}^*(H) +O(\ep_0)$, this $O(\ep_0)$ being uniform with respect to $\ell$. So 
\begin{align*}
\#\mathcal{D}^{\ell,\ell'}_f(j,H\pm\ep)&  \le  C 2^{j(\zeta_{\mu,p}^*(H)+O(\ep_0)+\eta)}.
\end{align*}

\medskip

\noindent
{\bf Subcase $(3b)$:}   $\frac{H}{h_{\ell+1}}d(\ell,\ell')> \zeta_{\mu,p}^*(H)$ and $h_{\ell'+1} \leq  \theta_p^{-1}(h_{\ell+1})$.

\mk 

  Recall the definition \eqref{thetap} of $\theta_p$.
A technical  lemma  is needed.

\begin{lemma}
\label{discret4}
For every $j$ large enough,  
\begin{equation*}
\mathcal{D}^{\ell,\ell'}_f(j,H\pm\ep)\subset \mathcal{D}_\mu\left (j,\left[\alpha_{\min},\alpha_{\min}+\frac{H}{h_{\ell+1}}(h_{\ell'+1}-\alpha_{\min})\right ]\pm O(\ep_0)\right ),
\end{equation*} 
where $O(\ep_0)$ is independent of $(\ell,\ell')$.
\end{lemma}
\begin{proof}
Take $\lambda\in \mathcal{D}^{\ell,\ell'}_f(j,H\pm\ep)$ and applying Lemma \ref{lemleader}, consider $j'\in \left[j\frac{H-\ep_0}{h_{\ell+1}+\ep_0},j\frac{H+\ep_0}{h_\ell-\ep_0}\right]$ such that there exists  $\la'= (i,j',k')\in \Lambda_{f,\mu}(j',I_{\ell}\pm\ep, I_{\ell'}\pm\ep)$ for which $\lambda' \subset 3\lambda$. 

Denote by $\lah $ the unique dyadic cube of $\mathcal D_j$   containing $\lambda'$. Then, note that: 
\begin{itemize}
\item
The two cubes $\la$ and $\lah$ are either equal or neighbors.  Hence, by property (P$_2$) of $\mu$,   $\mu(\lambda)\ge 2^{-j\ep_0}\mu(\lah)$ when $j$ is large enough.
\item
$\mu(\lah)= \mu(\lambda')\frac{\mu(\lah)}{\mu(\lambda')}$, and by construction of $\mu$ (see  \eqref{loulou}), $\frac{\mu(\lah)}{\mu(\lambda')}\ge 2^{-j\ep_0}2^{(j'-j)(\alpha_{\min}-\ep_0)}$.
\item 
Since  $\la' \in \Lambda_{f,\mu}(j',I_{\ell}\pm\ep, I_{\ell'}\pm\ep)$, $\mu(\lambda')\ge 2^{-j'(h_{\ell'+1}+\ep_0)}$.
\end{itemize} 

Consequently, 
\begin{align*}
\frac{\log \mu(\lambda )}{-j\log(2)} & \leq  \ep_0+ \frac{\log \mu(\lah )}{-j\log(2)}  \le 2\ep_0 + \frac{j'}{j}(h_{\ell'+1}+\ep_0) + (1-  \frac{j'}{j}) (\alpha_{\min}-\ep_0)\\
&\leq  \alpha_{\min}+\frac{j'}{j}( h_{\ell'+1}-(\alpha_{\min}-4\ep_0)) \leq \alm+   \frac{H}{h_{\ell+1}}(h_{\ell'+1}-\alpha_{\min})+O(\ep_0),
\end{align*}
where $O(\ep_0)$ is independent of $(\ell,\ell')$. This yields the result.
\end{proof}

Let us now  bound $\alpha_{\min}+\frac{H}{h_{\ell+1}}(h_{\ell'+1}-\alpha_{\min})$ from above. Thanks to \eqref{eq15},  $h_{\ell'+1} \leq  \theta_p^{-1}(h_{\ell+1})$ implies that $d(\ell,\ell')= \tau_\mu^*(h_{\ell'+1})$. Using that  $ \theta_p^{-1}(h_{\ell+1})\le  \theta_p^{-1}(H)\le \tau_\mu'(0^+)$ and  that $\tau_\mu^*$ is non decreasing over $[\alpha_{\min},\tau_\mu'(0^+)]$,  one has
$$\frac{H}{h_{\ell+1}}\tau_\mu^*(\theta_p^{-1}(h_{\ell+1})) \ge \frac{H}{h_{\ell+1}} \tau_\mu^*(h_{\ell'+1}) =\frac{H}{h_{\ell+1}} d(\ell,\ell') > \zeta_{\mu,p}^*(H)=\tau_\mu^*(\theta_p^{-1}(H)),$$
from which one deduces that 
\begin{equation}\label{ineq69}
\frac{\tau_\mu^*(\theta_p^{-1}(h_{\ell+1}))}{h_{\ell+1}}> \frac{\tau_\mu^*(\theta_p^{-1}(H))}{H}.
\end{equation}
Observe that the definition \eqref{thetap} of $\theta_p$ implies that
\begin{equation}
\label{thetapinverse}
 {\theta^{-1}_p(\beta)+p^{-1}\tau_\mu^*(\theta_p^{-1}(\beta))} = {\beta}
 \end{equation}  for all $\beta \in [\alpha_{\min}, \zeta_{\mu,p}'(0^+)]$.   Applying \eqref{thetapinverse} 
 to both sides of \eqref{ineq69} yields
\begin{equation}\label{hh}
\frac{\theta_p^{-1}(h_{\ell+1})}{h_{\ell+1}}<\frac{\theta_p^{-1}(H)}{H},
\end{equation}
and since $\frac{H}{h_{\ell+1}}>1$,  the following series of inequalities  holds:
\begin{equation}\label{hihi}
\alpha_{\min}+\frac{H}{h_{\ell+1}}(h_{\ell'+1}-\alpha_{\min})\le \frac{H}{h_{\ell+1}}h_{\ell'+1}\le   \frac{H}{h_{\ell+1}}\theta_p^{-1}(h_{\ell+1})\le  \theta_p^{-1}(H).
\end{equation}
Consequently, Lemma~\ref{discret4} yields
\begin{equation}\label{discret5}
\mathcal{D}^{\ell,\ell'}_f(j,H\pm\ep)\subset \mathcal{D}_\mu\left (j,[\alpha_{\min},\theta_p^{-1}(H)]\pm O(\ep_0)\right ).
\end{equation} 

Recall that  $\tau_\mu^*$ is continuous and non-decreasing over $[\alpha_{\min},\theta_p^{-1}(H)]$ by Proposition \ref{fm}(4). Hence,   choosing initially  $\ep_0$ small enough yields  for  $j$ large enough  that 
\begin{equation}\label{uh}
\#\mathcal{D}^{\ell,\ell'}_f(j,H\pm\ep)\le 2^{j(\tau_\mu^*(\theta^{-1}_p(H))+\eta)}=2^{j(\zeta_{\mu,p}^*(H)+\eta)}.
\end{equation}
 
\medskip

\noindent
{\bf Subcase $(3c )$: $\frac{H}{h_{\ell+1}}d(\ell,\ell')> \zeta_{\mu,p}^*(H)$ and $h_{\ell'}<\theta_p^{-1}(h_{\ell+1})<h_{\ell'+1}$.}

 Here one has  $h_{\ell'+1}\le (1+\ep_0) h_{\ell'} \le (1+\ep_0) \theta_p^{-1}(h_{\ell+1})$, so 
\begin{eqnarray}
\nonumber
\alpha_{\min}+\frac{H}{h_{\ell+1}}(h_{\ell'+1}-\alpha_{\min})   & \le& (1+\ep_0) \frac{H}{h_{\ell+1}}\theta_p^{-1}(h_{\ell+1})+\alpha_{\min} \left (1-\frac{H}{h_{\ell'+1}}\right ) \\
&\leq& (1+\ep_0) \frac{H}{h_{\ell+1}}\theta_p^{-1}(h_{\ell+1}).\label{eq75}
\end{eqnarray}
Also,   \eqref{eq15}  gives  $d(\ell,\ell')=\tau_{\mu}^*(\theta_p^{-1}(h_{\ell+1}))$, so $\frac{H}{h_{\ell+1}}d(\ell,\ell')> \zeta_{\mu,p}^*(H)$ is equivalent to \eqref{ineq69}, and it implies~\eqref{hh}. Finally, arguing as in the subcase $(3b)$ and using \eqref{eq75}, one sees that    \begin{equation}\label{ineqcruciale}
\alpha_{\min}+\frac{H}{h_{\ell+1}}(h_{\ell'+1}-\alpha_{\min})\le  H-\frac{\zeta_{\mu,p}^*(H)}{p}+O(\ep_0)=\theta_p^{-1}(H)+O(\ep_0).
\end{equation}
Applying Lemma \ref{discret4},   one deduces that     \eqref{uh} holds once again.

\medskip

\noindent
{\bf Subcase $(3d)$:   $\frac{H}{h_{\ell+1}}d(\ell,\ell')> \zeta_{\mu,p}^*(H)$ and $h_{\ell'} \geq  \theta_p^{-1}(h_{\ell+1})$.}

  By \eqref{eq15},   $d(\ell,\ell')=p(h_{\ell+1} -  h_{\ell'})$. Consequently,  $h_{\ell'} = h_{\ell'} -\frac{d(\ell,\ell')}{p} <h_{\ell+1}- \frac{h_{\ell+1}}{H} \zeta_{\mu,p}^*(H)/p$, and   
\begin{align*}
\alpha_{\min}+\frac{H}{h_{\ell+1}}(h_{\ell'+1}-\alpha_{\min})& \le  \frac{H}{h_{\ell+1}}h_{\ell'+1} <\frac{H}{h_{\ell+1}} \Big (h_{\ell+1}   -  \frac{h_{\ell+1}}{H}  \frac{\zeta_{\mu,p}^*(H)}{p}\Big)+ H\frac{h_{\ell+1}-h_{\ell}}{h_{\ell+1}}.
\end{align*}  
Thus, \eqref{ineqcruciale} and then  \eqref{uh} hold in this subcase as well.

\medskip

Collecting the estimates obtained along the cases considered above,   \eqref{estimfinale} is proved, and so is Proposition \ref{ubLD}.


\section{Typical singularity spectrum in  $\widetilde B^{\mu,p}_{q}(\R^d)$ }
\label{sec-typical}

In this section, the singularity spectrum of typical functions in  $\widetilde B^{\mu,p}_{q}(\R^d)$ when $\mu\in \MDs$ is computed, proving item (2) of  Theorem~\ref{main}. 

The   strategy   is similar to the one used to derive the generic multifractal behavior in classical Besov  spaces. First,   a   \textit{saturation} function is built, whose multifractal structure is   the one claimed   be generic in $\widetilde B^{\mu,p}_q(\R^d)$. Then,   this particular function is used to perturb a   countable family of dense  sets  in $\widetilde B^{\mu,p}_q(\R^d)$, in order to obtain a countable family of dense open sets on the intersection of which  the desired multifractal behavior holds. However, the construction of the saturation function and the multifractal analysis of typical functions are much more delicate in $\widetilde B^{\mu,p}_q(\R^d)$ than in  $B^{s,p}_q(\R^d)$.  
\medskip

The environment  $\mu\in \MDs$ is fixed for the rest of this section, as well as $(p,q)\in [1,+\infty]^2$ and  $\Psi\in\mathcal F_{s_\mu}$. 

\subsection{A saturation function}
\label{sec_saturation}

In this section, a saturation function $g^{\mu,p,q} \in \widetilde B^{\mu,p}_{q}(\R^d)$ is built via its wavelet coefficients, which are {\em as large as possible in $\widetilde B^{\mu,p}_{q}(\R^d)$}, and  its wavelet leaders are estimated.

The definition of $g^{\mu,p,q}$ demands some preparation. 

When $\alpha_{\min}=\alpha_{\max}$, we set $(M_N:= N^2)_{N\in\N^*}$ and   $I^N_i=\{\alpha_{\min}\}$ for all $1\le i\le M_N$.

When $\alpha_{\min}<\alpha_{\max}$, for every $N\in\N^*$, it is possible to find an integer $M_N$ such that  the interval $[\alpha_{\min}, \alpha_{\max}]=[\tau_\mu'(+\infty),\tau_\mu'(-\infty)]$  can be split into $M_N$ non-trivial contiguous closed intervals $I^N_1,I^N_2,...,I^N_{M_N}$ satisfying for every $i\in \{1,..., M_N\}$,
\begin{equation}
\label{decoup}
 |I^N_i |\leq 1/N\quad\text{and}\quad  \max\{ |{\tau_\mu^*}(\alpha)-{\tau_\mu^*}(\alpha')|: \alpha,\alpha'\in I^N_i\} \leq 1/N.
\end{equation}

Without loss of generality, we assume that the sequence $(M_N)_{N\ge 1}$ is  increasing.

In any case, item (4) of Proposition \ref{fm}  yields a decreasing sequence $(\eta_N)_{N\in\N^*}$ converging to 0 as $N\to\infty$, and for all $N\in\N^*$, $M_N$ integers $J_{N,1}$, $J_{N,2}$, ...,  $J_{N,M_N}$, such that for every $i\in \{1,.., M_N\}$,  for every  $j\geq J_{N,i}$, 
\begin{equation}
\label{maj1}
\left| \frac{ \log_2 \# \mathcal{D}_\mu(j,I^N_i\pm 1/N)}{ j } - 
 \max_{ \alpha\in I^N_i} 
 {\tau_\mu^*}( \alpha) \right| \leq \eta_N .
 \end{equation}
Without loss of generality, we assume that $\eta_N\ge 1/N$. 

Then,   define  inductively the non-decreasing    sequences of integers $(J_N)_{N\in\N^*}$  and $(N_j)_{j\geq 1}$ such that: 
\begin{equation}
\label{maj2}
\begin{cases}
\forall\, N\ge 1, \,J_N \geq  \max\{ J_{N,i}:\, i\in \{1,..., M_N\}\} \\  
 \forall\, N\ge 2,\,  M_N  \leq 2^{J_N\eta_{N-1}} ,\\
  \forall\, N\ge 3, \,  J_{N-1}\eta_{N-2}  < J_N\eta_{N-1},\\
\mbox{for every }J_N\leq j  < J_{N+1},  \mbox{ we  set }N_j=N.
  \end{cases}
 \end{equation}
 %
%

Moreover, \eqref{encadr2} makes it possible to impose that   for every $j\geq J_N$ and  $\lambda\in \mathcal D_j$,
$$2^{-j(\alpha_{\max} +1/N)} \leq  \mu(\lambda)\leq 2^{-j(\alpha_{\min} -1/N)}.$$

\mk 

Finally, let us introduce some coefficients depending on the elements $\la\in \Lambda_j$:
\begin{itemize} 
\mk\item If $L\in \Z^d$, $j\ge J_2$ and $\lambda\in \Lambda_j^L=\{\lambda=(i,j,k)\in \Lambda_j: \lambda_{j,k}\subset L+[0,1]^d\}$,  
  set
\begin{equation}
\label{formulagw0}
w_{\lambda}=\begin{cases}
\displaystyle \frac{2^{- \frac{3j\eta_{N_j-1}}{p}}}{j^{\frac{1}{p}+\frac{2}{q}} (1+\|L\|)^{\frac{d+1}{p}}} &\text{if $p<+\infty$}
\\
\displaystyle j^{-\frac{2}{q}}  &\text{if }p=+\infty, 
\end{cases}
\end{equation}
with the convention $\frac{2}{\infty}=0$.

\mk\item  If   $j\ge J_2$ and $\lambda=(i,j,k) \in \Lambda_j  $,  set $\alpha_{j,k}= \displaystyle\frac{\log_2\mu(\lambda_{j,k})}{-j} $ and 
$$
 \alpha_\lambda = \begin{cases} \alpha_{j,k}& \mbox{ if }  \alpha_{j,k}\in [ \alpha_{\min} , \alpha_{\max} ],\\     \alpha_{\min} & \mbox{ if } \alpha_{j,k}<  \alpha_{\min}, \\   \alpha_{\max} & \mbox{ if }  \alpha_{j,k} > \alpha_{\max} .
\end{cases}
$$
\end{itemize}
\begin{remark}\label{alphalambda}
Note that $\widetilde \ep_{\lambda}=\frac{\log_2\mu(\lambda)}{-j}-\alpha_\lambda$ tends to 0 uniformly in $\lambda\in\mathcal D_j$ as $j\to +\infty$.  In other words, there exists $\widetilde {\widetilde{\phi}}\in \Phi $ (recall Definition~\ref{defAD}) such that $| \frac{\log_2\mu(\lambda)}{-j}-\alpha_\lambda|\leq \frac{\widetilde {\widetilde{\phi}}(j)}{j} $.
\end{remark}

Recall the Definition \ref{defirreducible}  of the  irreducible dyadic    cube $\lai := \la_{\ji,\ki}$.

\begin{definition}
The saturation function ${g^{\mu,p,q}}  :\R^d\to\R$ is defined by its wavelet coefficients in the wavelet basis associated with $\Psi$,  denoted by $(c^{\mu,p,q}_\la)_{\lambda\in\Lambda}$, as follows:
\begin{itemize} 

\mk\item $c^{{\mu,p,q}}_\lambda=0$ if $\lambda\in \bigcup_{j<J_2} \Lambda_j$. 

\mk\item If  $j\ge J_2$ and  $\lambda=(i,j,k)\in \Lambda_j$,   set 
\begin{equation}
\label{formulagw}
c^{{\mu,p,q}}_\lambda=\begin{cases} 
\displaystyle w_\lambda  \cdot \mu(\lambda_{j,k})&\text{if }p=+\infty,\\
w_\lambda \cdot \mu(\lambda_{j,k})\,2^{-\ds\ji\frac{{\tau_\mu^*}(\ds\alpha_{{\small \lai}})}{p}}&\text{if  $p<+\infty$}.
 \end{cases}
\end{equation}
 \end{itemize}
\end{definition}

\begin{remark}\label{remg2}
\begin{enumerate}

\item Note that $c^{{\mu,p,q}}_\lambda$ does not depend on $i$ if $\lambda=(i,j,k)$. Consequently,  $c^{{\mu,p,q}}_\lambda$ is defined without ambiguity by the same formula for $\lambda\in\mathcal D_j$. 
\item
The choice of  $\ji $ and $\lai $ in the exponent $2^{- \ji \frac{{\tau_\mu^*}(\alpha_{\lai })}{p}}$  in \eqref{formulagw}
implies that at a given generation $j$, the wavelet coefficients of ${g^{\mu,p,q}}$ display several order of magnitudes, which are influenced by the values of $\mu$ along the $j$ first generations of dyadic cubes. One can also guess from this choice that approximation by dyadic vectors   plays an important role in our analysis, since the local behavior of $g^{\mu,p,q}$ around a point $x$   depends on how close $x$ is to the dyadic vectors.

\item When  $p<+\infty$ and $\tau_\mu^*(\alpha_{\min})=0$,    in \eqref{formulagw} $\mu(\lambda)2^{-\ji \frac{{\tau_\mu^*}(\alpha_{\lai })}{p}}$ can be replaced by the simpler term $\mu(\lambda)$ to get a relevant saturation function ${g^{\mu,p,q}}$. When $\tau_\mu^*(\alpha_{\min})>0$, the situation is more subtle, and the ubiquity properties pointed out in Proposition \ref{ubiquity} come into play.
\end{enumerate}
\end{remark}

\begin{lemma}\label{gdansB}
The function $g^{\mu,p,q} $ belongs to  $  B^{\mu,p}_q(\R^d)$ and $\widetilde B^{\mu,p}_q(\R^d)$.
\end{lemma}

\begin{proof}
Suppose that $p<+\infty$. 

For $j\in\N$ and $L\in\Z^d$, set $\mathcal D_j^L=\{\lambda\in \mathcal D_j: \lambda\subset L+ [0,1]^d\}$ 	and $\Lambda_j^L=\{(i,j,k)\in \Lambda_j: \lambda_{j,k}\in \mathcal D_j^L\}$. 

Recall  that for $\lambda=(i,j,k)$, $\mu(\lambda)$ stands for $\mu(\lambda_{j,k})$.

Let us define, for $j\ge J_2$ and $L\in \Z^d$,    $A_{j,L} =  \sum_{\lambda\in \Lambda^L_j } \left(\frac{\left| c^{{\mu,p,q}}_\lambda \right|}{ \mu(\lambda)} \right)^p$. To prove that  $g^{\mu,p,q}\in B^{\mu,p}_q(\R^d)\subset \widetilde B^{\mu,p}_q(\R^d)$, it is enough to show that $A_j:= \Big (\sum_{L\in\Z^d} A_{j,L}\Big)^{1/p}\in~\ell^q(\N)$.

For $j\in [J_N, J_{N+1})$,  by \eqref{formulagw} and \eqref{formulagw0}, one has 
\begin{align}
\label{AJL}  
A_{j,L}  & =    \!\!\!\sum_{\lambda\in \Lambda_j^L }  \! \left(\frac{ 2^{-3j\eta_{N_j-1}/p }  \mu(\lambda) 2^{  - \ji \frac{{\tau_\mu^*}(\alpha_{\lai })}{p}}}{j^{\frac{1}{p}+\frac{2}{q}}(1+\|L\|)^{(d+1)/p} \mu(\lambda) } \right) ^p  \!\!\! =  \frac{(2^d-1)2^{- 3j\eta_{N_j-1}}}{j^{1+\frac{2p}{q}} (1+\|L\|)^{(d+1)}}  \sum_{\lambda\in \mathcal{D} _j ^0  } \! 2^{ - \ji {\tau_\mu^*}(\alpha_{\lai})},
\end{align}
where the factor $2^{d}-1$ comes from the fact that $c^{\mu,p,q}_\lambda$, $\lambda=(i,j,k)$, is independent of $i\in\{1,\ldots,2^{d}-1\}$.   The periodicity of $\mu$, i.e. $\mu_{|\zu^d}=\mu_{|L+\zu^d}$ is also used.

Recalling the notations in Remark \eqref{rem-philambda},  if $\la\in \mathcal D_j$ and $ \lai$ is the cube associated with its irreducible representation, then one can write  $\lambda= \lai \cdot [0,2^{-(j-\ji)}]^d $.

Then, after regrouping in \eqref{AJL} the terms according to the generation of their irreducible representation, one has
\begin{align}
\nonumber
A_{j,L} 
&=  (2^d -1)\frac{2^{- 3j\eta_{N_j-1}}}{j^{1+\frac{2p}{q}} (1+\|L\|)^{(d+1)}} \Big (1+\sum_{J=1}^j \sum_{  \la \in \mathcal D_J^0\setminus  ( \mathcal D_{J-1}^0\cdot [0,2^{-1}]^d)} 2^{  - J{\tau_\mu^*}(\alpha_{  \la })}\Big )\\
\nonumber
& \leq 2^d \frac{2^{- 3j\eta_{N_j-1}}}{j^{1+\frac{2p}{q}} (1+\|L\|)^{(d+1)}} \Big (1+\sum_{J=1}^j \sum_{  \lambda\in \mathcal D_J^0 } 2^{- J{\tau_\mu^*}(\alpha_{ \lambda})}\Big )\\
\label{eqajl}
&=  2^d \frac{2^{- 3j\eta_{N_j-1}}}{j^{1+\frac{2p}{q}} (1+\|L\|)^{(d+1)}} \left (\sum_{J=0}^{J_1-1} +\Big (\sum_{N=1}^{N_j-1} \sum_{J=J_N}^{J_{N+1}-1}\Big)+\sum_{J=J_{N_j}}^{j} \right)  \sum_{ \lambda\in \mathcal D_J^0} 2^{- J{\tau_\mu^*}(\alpha_\lambda)}.
\end{align}
For each $J_N\le J<J_{N+1}$, using   \eqref{decoup} and then~\eqref{maj1}, we obtain
\begin{align*}
 \sum_{ \lambda\in\mathcal D_J^0} 2^{- J{\tau_\mu^*}(\alpha_\lambda)}& \leq  \sum_{i=1}^{M_{N_J}}  \sum_{\lambda\in  \mathcal{D}_\mu(j,I^{N_J}_i\pm 1/N) }  2^{- J  (\max\{ \tau_\mu^*( \alpha) : \alpha\in I^{N_J}_i\} -1/N_J)}\\
& \leq  \sum_{i=1}^{M_{N_J}}   2^{J  (  \max\{\tau_\mu^*( \alpha):\alpha\in I^{N_J}_i\}+\eta_{N_J}) }  2^{- J  (\max\{ \tau_\mu^*( \alpha) : \alpha\in I^{N_J}_i\} -1/N_J)}\\
& = M_{N_J} 2^{J(\eta_{N_J}+1/{N_J})}\le M_{N_J} 2^{2J\eta_{N_J}}.
\end{align*}
Consequently, by \eqref{maj2}, 
\begin{align*}
&\left (\Big (\sum_{N=1}^{N_j-1} \sum_{J=J_N}^{J_{N+1}-1}\Big)+\sum_{J=J_{N_j}}^{j} \right)  \sum_{ \lambda\in \mathcal D_J^0} 2^{- J{\tau_\mu^*}(\alpha_\lambda)}\\
&\le\sum_{N=1}^{N_j-1} \sum_{J=J_N}^{J_{N+1}-1} M_{N} 2^{2J\eta_{N}}+\sum_{J=J_{N_j}}^{j} M_{N_j}  2^{2J\eta_{N_j}} \\
&\le\sum_{N=1}^{N_j-1}  (J_{N+1}- J_N) M_{N} 2^{2J_{N+1}  \eta_{N}}+ (j- J_{N_j}+1 ) M_{N_j}  2^{2j\eta_{N_j}} \\ 
&\le jM_{N_j}2^{2j\eta_{N_j-1}},
\end{align*}
since all terms  $M_{N} 2^{2J_{N+1}  \eta_{N}}$, for $N\leq N_{j-1}$, are less than $M_{N_j} 2^{2j   \eta_{N_{j-1}}}  $. 

Setting $C_\mu=\sum_{J=0}^{J_1-1}\sum_{ \lambda\in \mathcal D_J^0} 2^{- J{\tau_\mu^*}(\alpha_\lambda)}$, by \eqref{maj2} and $ M_{N_j} 2^{2 J_N  \eta_{N_{j-1}}} \leq 1$ one has 
$$
A_{j,L}\le 2^d  \frac{M_{N_j}2^{-J_n\eta_{N_j-1}}}{j^{\frac{2p}{q}} (1+\|L\|)^{(d+1)}}(C_\mu+ 1)\le  \frac{2^d(C_\mu+ 1)}{j^{\frac{2p}{q}} (1+\|L\|)^{(d+1)}}.
$$
  Finally, 
$$
\Big (\sum_{L\in\Z^d} A_{j,L}\Big)^{1/p}=\Big\|\Big (\frac{c^{{\mu,p,q}}_\lambda}{\mu(\lambda)}\Big )_{\lambda\in\Lambda_j}\Big\|_p  =O(j^{-2/q}),
$$ hence  $\Big (\Big\|\Big (\frac{c^{{\mu,p,q}}_\lambda}{\mu(\lambda)}\Big )_{\lambda\in\Lambda_j}\Big\|_p\Big)_{j\in\N}$ belongs to  $\ell^q(\N)$. This implies that $g^{\mu,p,q}\in B^{\mu,p}_q(\R^d)$. 

\medskip

When $p=+\infty$, the estimate is much simpler and  left to the reader. 
\end{proof}

 Next lemma  shows that  the wavelet leader (recall \eqref{defleaders}) $\ld ^{g^{\mu,p,q}} _{\lambda}$ of $g^{\mu,p,q}$  at   $\la\in \mathcal D_j$ is essentially comparable to the wavelet coefficients $c^{{\mu,p,q}}_{\la'}$ indexed by the cubes $\lambda'$ of generation $j$ which are neighbors of  $\lambda$. This property is key  to estimate the $L^q$-spectrum of $g^{{\mu,p,q}}$ relative to $\Psi$. 

\begin{lemma}
\label{lem-maj2} Fix $L\in \Z^d$. For every $\ep>0$, there exists $J^\ep\in\N$ such that if $j\geq J^\ep$, for every $\lambda\in \mathcal D^L_j$,
$${\widetilde{c}_\lambda}^{{\,\mu,p,q}}\le  \ld  ^{g^{\mu,p,q}} _\lambda \le 2^{j \ep}{\widetilde{c}_\lambda}^{{\,\mu,p,q}},$$
where ${\widetilde{c}_\lambda}^{{\,\mu,p,q}}=\max\{c^{{\mu,p,q}}_{\widetilde\lambda}:\lambda\in \mathcal D_j,\, \widetilde\lambda\subset 3\lambda\}$.
\end{lemma}

\begin{proof}  It is enough to prove the result for $L=0$. 
Let $\ep,\ep'\in (0,1)$. Let $j\geq 1$ and $\la\in \mathcal D^0_j$. Let us begin with some remarks:
\begin{itemize}
\item
in \eqref{formulagw}, the term $w_\la$  depends only on $j$, and is decreasing with $j$.
\item
 if $\la'\subset \la$,  $\mu(\la')\leq \mu(\la) $ since $\mu\in \mathcal{C}(\R^d)$. 
\item
 by  Remark~\ref{remg2}(1) $c^{{\mu,p,q}}_\la$ does not depend on the index $i$ of $\la=(i,j,k)$.
\end{itemize}

Next, observe that if $\la'\subset \la$,   the irreducible cubes  $  \overline{\la'}  \in \mathcal{D}_{ \overline{j'} }$ and $  \lai  \in \mathcal{D}_{ \ji}$ respectively associated with $\la'$ and $\la$, are such  that  $\ji  \leq  \overline{j'} $. 

Then one controls the wavelet coefficients as follows:
\begin{enumerate}

 \sk \item [(i)] 
By the property   (P$_1$) of $\mu$, there exists $M\in\N^*$ such that for every $ \la' \in \mathcal D_{Mj}$ one has  $\mu( \la' ) \leq 2^{-j (d/p+2\alpha_{\max}+1)}$. So $\mu(\la) 2^{- \ji\frac{{\tau_\mu^*}( \alpha_{{\small \lai}})}{p}}   \geq 2^{-j(\alpha_{\max}+1) -jd/p } \geq \mu(\la')$, which implies that for  $j'\ge Mj$,   
$c^{\mu,p,q}_{\lambda'}   \le  c^{\mu,p,q}_{\la}.$

\noindent
Hence, the only wavelet coefficients $c_{\la'}$  to consider to compute  $ \ld  ^{g^{\mu,p,q}} _\lambda$ for $\la\in\mathcal{D}_j$ are those of generations $j'$ such that $j\leq j'\leq Mj$.

\item[(ii)]
 if  $j'\leq Mj$ and  $\ji\le jp\ep/(2d)$, then  $2^{-  \ji\frac{{\tau_\mu^*}( \alpha_{{\small \lai}})}{p}}    \geq 2^{- jp\ep/(2d)\cdot   d/p } \geq 2^{-j\ep}$, so   $c^{\mu,p,q}_{\la}\ge w_\lambda\mu(\lambda)2^{-j\ep} $ and  by the remarks of the beginning of the proof,  
 $$c^{\mu,p,q}_{\la'}\le w_{\la'} \mu(\la')  \leq w_\la \mu(\la)\leq  c^{\mu,p,q}_{\la}2^{j\ep}.$$
 
 \sk 
\item[(iii)]
It is possible to choose $\ep'$ small enough so that if $\overline{j'} -  \ji\le \ep' \overline{j'}$, then  since $\mu$ is almost doubling, $|\alpha_{\overline{\la'}}-\alpha_{\overline{\la}}|$ is so small that $|\overline{j'}\tau_\mu^*(\alpha_{\overline{\la'}})-\ji\tau_\mu^*(\alpha_{\overline{\la}})|\le  \ji p \ep$.

 \sk  \item[(iv)]
If $j'\leq Mj$, $\ji> jp\ep/(2d)$  and   $\overline{j'} -  \ji\le \ep' \overline{j'}$, then by (iii) one has  (for $j$ is large enough)
$$c^{\mu,p,q}_{\la'}  \le c^{\mu,p,q}_{\la} 2^{\overline{j} \ep}  \leq   c^{\mu,p,q}_{\la} 2^{j \ep} .$$

 \sk  \item[(v)]
If  $j'\leq Mj$, $\ji> jp\ep/(2d)$ and  $\overline{j'} -  \ji > \ep' \overline{j'}$, then
\begin{equation}\label{cccc}
\overline{j'} \alpha_{ \overline{\la'}} =    \ji \alpha_{ \lai} +(\overline{j'} -  \ji) \alpha
\end{equation}
for some $\alpha\in [\alm-\ep, \alpha_{\max}+\ep]$.  
The concavity of $\tau_\mu^*$ then implies that for some $  \ep''$ independent of $j$ and $j'$, 
$$   \overline{j'}   \tau_\mu^*(\alpha_{\overline{\la'}}  )\geq  \ji  \tau_\mu^*(\alpha_{  \lai })  +(  \overline{j'}  -  \ji ) (\tau_\mu^*(\alpha^*) -\ep'') , \ \ 
 \mbox{  where }  \alpha^* = \begin{cases} \alpha \mbox{  when } \alpha\in [\alm, \alpha_{\max}] ,\\  \alpha_{\max}  \mbox{  when }  \alpha \geq \alpha_{\max} , \\  \alpha_{\min}  \mbox{  when }  \alpha \leq \alpha_{\min} . \end{cases} $$

 In particular, $ \overline{j'} \tau_\mu^*(\alpha_{\overline{\la'}} )\geq   \ji  \tau_\mu^*(\alpha_{\lai  })  -(  \overline{j'} -\ji ) \ep''$, hence 
 $$2^{-   \overline{j'}  \tau_\mu^*(\alpha_{\overline \la'} )/p} \leq 2^{-   \ji  \tau_\mu^*(\alpha_{ \lai} )/p} 2^{  (  \overline{j'} - \ji ) \ep''/p}\leq 2^{-  \ji  \tau_\mu^*(\alpha_{  \lai} )/p} 2^{    \overline{j'}  \ep''/p}  \leq 2^{-  \ji  \tau_\mu^*(\alpha_{  \lai} )/p} 2^{   Mj  \ep''/p}.$$
 One checks that  $  \ep''$  can be chosen  as small as necessary when $j$ tends to infinity, in particular so that one has for large $j$ that $ M  \ep''/ p \leq \ep $. Finally, with this choice of $\ep''$, $c^{\mu,p,q}_{\la'}   \leq   c^{\mu,p,q}_{\la} 2^{j \ep} $.
  \end{enumerate}
 
%
%
%

Putting together all the previous information yields that when $j$ is  large enough, for all $\lambda\in \mathcal D_j^0$ and all $\lambda'\in \mathcal D_{j'}$ such that $\lambda'\subset\lambda$, one has $c^{{\mu,p,q}}_{\lambda'}\leq c^{{\mu,p,q}}_{\lambda}2^{ j\ep }$. 

The same property holds true for all   $\widetilde\lambda\in\mathcal D_j$ such that $\widetilde \lambda\subset  3[0,1]^d$ and $\lambda'\in \mathcal D_{j'}$ such that $\lambda'\subset\widetilde \lambda$.  This yields the desired property. 
\end{proof}

\subsection{The singularity spectrum of the saturation function ${g^{\mu,p,q}} $ and some of its perturbations}
\label{sec_saturation2}  
 
We now determine the  singularity spectrum of ${g^{\mu,p,q}} $, and more generally of any function whose wavelet coefficients are ``comparable''  to  those of ${g^{\mu,p,q}}$ over infinitely many generations. 

\begin{proposition}
\label{propspecG}
Let $f\in \widetilde B^{\mu,p}_q(\R^d) $ such that for any $L\in\Z^d$, there exists  an increasing sequence of integers $(j_n)_{n\in\N}$, and a positive sequence $(\ep_n)_{n\in\N}$  converging to 0 such that for all $n\ge 1$ and  $\lambda=(i,j_n,k)\in \Lambda_{j_n}$ such that $\lambda_{j_n,k}\subset L+3[0,1]^d$ the inequality $2^{-j_n\ep_{n}} c^{{\mu,p,q}}_\lambda\le |c^f_\lambda|$ holds. Then ${\sigma}_f={\sigma}_{g^{\mu,p,q}} =\zeta_{\mu,p}^*$. 
\end{proposition}

Only the case  $p<+\infty$ is treated, the   case $p=+\infty$ is simpler and   deduced from arguments similar to those developed below. Fix $(j_n)_{n\in\N}$ and  $(\ep_n)_{n\in\N}$ as in the statement. 
\medskip

It is enough to prove that $\dim E_f(H)\cap (L+[0,1]^d) =\zeta_{\mu,p}^*(H)$ for all $H\in\R$ and $L\in\Z^d$. Without loss of generality we work     with $L=0$ and show  that  $\dim E_f(H)\cap [0,1]^d =\zeta_{\mu,p}^*(H)$ for all $H\in\R$.

Note that the characterization \eqref{defleaders2} and  the assumptions on $(j_n)_{n\in \N}$   imply that   for all $x\in [0,1]^d$, for all  $\la_{j_n}=(i,j_n,k)$  such that $x \in \la_{j_n }$,   
\begin{equation}\label{tutu}
 \liminf_{n\to + \infty} \frac{\log c^{{\mu,p,q}}_{\lambda_{j_n}}} {\log 2^{-j_n} }\ge  \liminf_{n\to + \infty} \frac{\log |c^f_{\lambda_{j_n}}| }{\log 2^{-j_n} }\ge \liminf_{j\to + \infty} \frac{\log \ld ^f_{j_n}(x) }{\log 2^{-j_n} }\ge  h_f(x).
 \end{equation}   
Recall that the value of $c^{{\mu,p,q}}_{\lambda_{j_n} }$ does not depend on the index $i$ of $\la_{j_n}=(i,j_n,k)$.

\subsubsection{The upper bound $\sigma_f\le\zeta_{\mu,p}^*$.}   Theorem~\ref{main}(1)  gives $\si_f(H) \leq \zeta_{\mu,p}^*(H) $ for all $H\leq \zeta_{\mu,p}'(0^+)$. Note also that  $\zeta_{\mu,p}^*(H)=d$ for all $H\in [\zeta_{\mu,p}'(0^+),\zeta_{\mu,p}'(0^-)]$.  Hence it remains us  to treat the case $H>\zeta_{\mu,p}'(0^-)$, which  corresponds to the decreasing part of the spectrum of $f$.

\mk 

Fix $H>\zeta_{\mu,p}'(0^-)$ and $x\in[0,1]^d$ such that  $h_f(x)\ge H$.  


By \eqref{tutu}, denoting  $\la_{j_n}$  any $\la=(i,j_n,k)\in \La_{j_n}$ such that $x \in \la_{j_n }$,   one has
\begin{equation}\label{tutu'}
 \liminf_{n\to+ \infty} \frac{ \log c^{{\mu,p,q}}_{\lambda_{j_n} }} {\log 2^{-j_n} }\ge H.
 \end{equation}

Recall that  $ \overline{\lambda_{j_n} } \in \mathcal{D}_{ \overline{j_{n}} }$ is the  irreducible representation of $\la_{j_n}$. Using the concatenation of cubes introduced   after Definition~\ref{defirreducible}, one writes $\lambda_{j_n} = \overline{\lambda_{j_n} }\cdot [0,2^{-(j_n- \overline{j_n} )}]^d$, and 
\begin{align} 
\frac{\log c^{{\mu,p,q}}_{\lambda_{j_n} } }{\log 2^{-j_n} }
\label{interm}&= \frac{\log_2 w_{ {\lambda_{j_n}}}}{ j_n}  +  \frac{\log_2\mu(\lambda_{j_n}))}{-j_n}  + \frac{\overline{j_n}}{j_n} \frac{\tau_\mu^*(\alpha_{\overline{\lambda_{j_n}}})}{p} .
\end{align}
Recall   \eqref{lala} and  the fact that for $j,j'\in\mathbb N$ and $\lambda\in\mathcal D_j$, one has $\mu( \lambda\cdot [0,2^{-j'}]^d )=\mu(\la) 2^{-\phi_\lambda } 2^{-j'\alpha_{\min}+ \tilde\phi_{\la}( j')}$, where by \eqref{majphiphitilde} $|\phi_\lambda| $ and  $|\widetilde \phi_\lambda(j') |$ are uniformly bounded by a $o(j)$ and a $o(j')$ respectively. So, 
\begin{align*}
\frac{\log_2\mu(\lambda_{j_n})}{-j_n}  &= \frac{\overline{j_n}}{j_n}  \frac{\log_2\mu(\lambda_{\overline {j_n}})}{-\overline j_n}+  \frac{\phi_{\overline {\lambda_{j_n}} } }{j_n} +   \frac{j_n-\overline { {j_n}}}{j_n } \alm + \frac{  \tilde\phi_{\overline{\la_{j_n}}}(j_n-\overline { {j_n}} ) }{j_n}
\end{align*} 
which combined with \eqref{interm} yields
\begin{align}
\label{yeye}
\frac{\log c^{{\mu,p,q}}_{\lambda_{j_n} } }{\log 2^{-j_n} }&=\frac{\overline{j_n}}{j_n} \theta_p(  \alpha_{\overline{\lambda_{j_n} } }  )+\Big (1-\frac{\overline{j_n}}{j_n}\Big )\alpha_{\min} +r_n(x) ,
\end{align}
where 
\begin{align*}
r_n (x) & = \frac{\log_2 w_{ {\lambda_{j_n}}}}{ j_n} +  \frac{\overline{j_n}}{j_n} \Big(  \frac{\log_2\mu(\lambda_{\overline {j_n}})}{\overline {j_n}} - \alpha_{\overline{\lambda_{j_n}}} \Big) +    \frac{\phi_{\overline {\lambda_{j_n}} } }{j_n} + \frac{  \tilde\phi_{\overline{\la_{j_n}}}(j_n-\overline { {j_n}} ) }{j_n}.
\end{align*}
The dependence of $r_n(x)$ on $x$ is explicit,   to remember it. But it does not play any role in the bounds above, which are uniform in $j_n$ and $j_n-\overline{j_n}$.

\begin{lemma}
One has  $\lim_{n\to+\infty}r_n (x)=0$. 
\end{lemma}
\begin{proof}

The first term in $r_n(x)$   tends to zero when $n\to +\infty$, by definition \eqref{formulagw0} of $w_\la$.

For the other terms in $r_n(x)$, let us define 
$$C=\max \left( \sup_{j\geq 1}  \left\{\frac{\wwphi(j)}{j} \right\}, \sup_{j\geq 1}  \left\{ \frac{ | \phi_\la  |}{j} :  \lambda\in \mathcal D_j\ \right\},\sup_{j'\geq 1}  \left\{  \frac{|\tilde\phi_{ \la}(j')| }{j'}:  \lambda\in \bigcup_{j\in\N}\mathcal D_j \right\}\right)\!\! .$$
By \eqref{lala} and Remark \ref{alphalambda}, one has $C<+\infty$. 

Now fix $\eta\in (0,1)$ and let us treat the second term. Remark \ref{alphalambda} again gives that $\frac{\overline{j_n}}{j_n} \Big(  \frac{\log_2\mu(\lambda_{\overline {j_n}})}{\overline {j_n}} - \alpha_{\overline{\lambda_{j_n}}} \Big) \leq \frac{\wwphi(\overline{j_n})}{j_n}.$
 When $j_n$ is large, one sees that:
\begin{itemize}
\item
if $\frac{\overline{j_n}}{j_n}>\eta$, then  $\overline{j_n}$ is   large  and $\frac{|\wwphi(\overline{j_n}) }{j_n}  | \leq | \frac{\wwphi(\overline{j_n}) }{\overline{j_n}}| \le \eta$,
\item
 if $\frac{\overline{j_n}}{j_n}\le \eta$, then $ \frac{\overline{j_n}}{j_n}  |\wwphi(\overline{j_n})  |   \leq  C\eta$.
\end{itemize}

  In any case, for $n$ large enough      $\frac{\overline{j_n}}{j_n}|\wwphi(\overline{j_n}) |  \le (C+1)\eta$. 
  
  The same argument applies to  the third term $\frac{\overline{j_n}}{j_n}  \phi_{\overline{\lambda_{j_n} }}$. 
  
  Finally, for the fourth term, one has:
  \begin{itemize}
\item
if $\frac{j_n -\overline{j_n}}{j_n} >  \eta$, then $j_n-\overline{j_n}$ is also large and $\frac{  \tilde\phi_{\overline{\la_{j_n}}} (j_n-\overline{j_n} ) }{j_n} \leq  \frac{ \tilde\phi_{\overline{\la_{j_n}}}(j_n-\overline { {j_n}} ) } { j_n-\overline { {j_n}}} \le \eta$,
\item
 if $ \frac{j_n -\overline{j_n}}{j_n} \leq   \eta$ , then   $ \frac{ | \tilde\phi_{\overline{\la_{j_n}}} (j_n-\overline { {j_n}} ) | }{j_n} = \frac{ | \tilde\phi_{\overline{\la_{j_n}}} (j_n-\overline { {j_n}} ) | }{ j_n-\overline { {j_n}}  } \frac{   j_n-\overline { {j_n}}   }{j_n}    \leq  C\eta$.
\end{itemize}

 This concludes the proof of the Lemma.
 \end{proof}

Note now that  $\theta_p(\alpha)\geq \alpha_{\min}$ for all $\alpha\in[\alpha_{\min},\alpha_{\max}]$. Since $\alpha_{\min} \leq \zeta_{\mu,p}'(0^-) <H$,  \eqref{tutu'} and \eqref{yeye}  together imply that  necessarily, for every $\ep>0$, $ \theta_p(  \alpha_{\overline{\lambda_{j_n} }}  ) \geq H-\ep$ for infinitely many integers $n$. Hence,  on one hand $H\le \theta_p(\alpha_p)$ and in particular $E_f(H)=\emptyset$ if $H>\theta_p(\alpha_p)$, and on the other hand
$$ \overline{h}_{\mu}(x)  \geq  \limsup_{j\to +\infty} \frac{\log_2 \mu(\la_{j}(x)) }{-j}  \geq \limsup_{n\to +\infty} \alpha_{\overline{\la_{j_n}}}  \geq   \theta_p^{-1}(H) ,$$
where  the same notations as above are used, i.e. $\la_{j_n}$ is here the unique cube of generation $j_n$ that contains $x$. This implies that    $x\in \overline E_\mu^{\geq} (\theta_p^{-1}(H))$. 

As a conclusion,  $H\le \theta_p(\alpha_p)$ and $E_f(H)\subset \overline E_\mu^{\geq} (\theta_p^{-1}(H))$. Since $\theta_p^{-1}(H)\ge \tau_\mu'(0^-)$ lies in the decreasing part of the singularity spectrum of $\mu$,  Proposition~\ref{fm}(3) yields that     $\dim E_f(H)\le \dim \overline {E^\geq _\mu}(\theta_p^{-1}(H)) =  \tau_\mu^*(\theta_p^{-1}(H))$. This is the desired upper bound.

\medskip

\subsubsection{ The lower bound $\sigma_f \geq \zeta_{\mu,p}^*$  over the range $[\alpha_{\min}, \theta_p(\alpha_p)]=[\zeta_{\mu,p}'(+\infty),\zeta_{\mu,p}'(-\infty)]$.} 

Two cases must be separated. 
\medskip

\noindent {\bf Case 1:}   $H\in [\theta_p(\alpha_{\min}),\theta_p(\alpha_p)]$.  

Let $\alpha\in[\alpha_{\min},\alpha_p]$ such that $H=\theta_p(\alpha)  ( =\alpha+ \tau_\mu^*(\alpha)/p)$.  Our goal is to show that  $\si_f(H) = \dim E_f(H) \geq \zeta_{\mu,p}^*(H) = \tau_\mu^*(\alpha)$. To achieve this, we prove that $\mu_\alpha(E_f(H))>0$, where $\mu_\alpha$ is the measure built in Section \ref{dimensiond}. Since $\mu_\alpha$ is exact dimensional with exponent $\tau_\mu^*(\alpha)$, this yields the claim.
\medskip

For any $H'\ge 0$ set 
$$
 E_f^{\leq}(H')  :=  \{y\in [0,1]^d: h_f(y)\le H'\}.
 $$ 
Let us start with one technical lemma.
 \begin {lemma}
 \label{lem76}
For every  $\eta>0$,  $\mu_\alpha (E_\mu(\alpha)\cap E_f^{\leq}(H-\eta))=0$.
 \end{lemma}
 \begin{proof}
Fix $\eta>0$,  $J_0\in\N$, and  set 
$$
E_{\mu,\eta,J_0}(\alpha)=\left\{x\in[0,1]^d: \, \begin{cases}   \forall \, J\ge J_0,\  \forall  \, \lambda\in \mathcal D_J \mbox{ such that } \lambda\subset 3\lambda_J(x)  , \\ 2^{-J(\alpha+\frac{\eta}{8} )}\le  \mu (\lambda)\le 2^{-J(\alpha-\frac{\eta}{8} )} \end{cases}\right\}
$$  
and for $j\geq J\ge J_0$  
\begin{equation}\label{DepJj}
\mathcal{D}_{\eta,J,j}(\alpha) =\left \{\lambda \in \mathcal D_J: \,\begin{cases} \lambda\cap  E_{\mu,\eta,J_0}(\alpha)\cap  E_f^{\leq}(H-\eta)\neq\emptyset \ \mbox{ and }\\\exists \, \lambda'=(i,j,k)\in\Lambda_j,\ \lambda'\subset 3\lambda,\, |c^f_{\lambda'}|\ge 2^{-J(H-\frac{\eta}{2})}
\end{cases}\right\}.
\end{equation}

Recall the following fact  stated along the proof of Lemma \ref{lem-maj2}:  there exists a constant $M$ such that the only wavelet coefficients $c_{\la'}$  to consider to compute  $ \ld  ^{g^{\mu,p,q}} _\lambda$ for $\la\in\mathcal{D}_j$ are the $j'$ such that $j\leq j'\leq Mj$.

\begin{lemma}
There exists $C>0$ such that for $  J_0\leq J \leq j \leq MJ$,
$$  \# \mathcal{D}_{\eta ,J,j} (\alpha) \le  C 2^{-(j-J)p \frac{\alpha_{\min}}{2} } 2^{J(\tau_\mu^*(\alpha)-p\frac{\eta}{8})} , $$
and when $j> MJ$, $\mathcal{D}_{\eta ,J,j} (\alpha) $ is empty.
\end{lemma}

\begin{proof}
The case $j> MJ$ follows from the remark just before the Lemma.

Let  $x\in E_{\mu,\eta,J_0}(\alpha)\cap  E_f^{\leq}(H-\eta) $.  By \eqref{defleaders2}, there  are infinitely many integers $J\ge J_0$ for which $L^f_J(x)\ge 2^{-J(H- \eta/2 )}$. For such a  generation $J$,  the definition of the wavelet leader as a supremum  implies that there exist $MJ\geq j\ge J$ and $\lambda=(i,j,k)\in\Lambda_j$ with $\lambda \subset 3\lambda_J(x)$ such that $|c^f_\lambda|\ge 2^{-J(H-\eta/2)}$. This means that $\la_J(x) \in    \mathcal{D}_{\eta ,J,j} (\alpha) $.

Recalling \eqref{loulou},       assume that $J_0$ is so large that $\mu(\lambda )\le  \mu(\lambda_J(x)) 2^{J\eta/8}2^{-(j-J)\alpha_{\min}/2}$.

Then,  the definition of $E_{\mu,\eta,J_0}(\alpha)$ and the fact that $ \alpha+ \tau_\mu^*(\alpha)/p =H$ give    
\begin{equation}\label{crucial!}
\frac{|c^f_\lambda|}{\mu(\lambda )}\ge   2^{-J\eta/8}2^{(j-J)\frac{\alpha_{\min}}{2}} 2^{-J(H - \frac{\eta}{2})}2^{ J(\alpha-\eta/8) }\ge  2^{(j-J)\frac{\alpha_{\min}}{2}}2^{-J \big (\frac{\tau_\mu^*(\alpha)}{p}-\frac{\eta}{4}  \big)}.
\end{equation}

 Since $f\in \widetilde B^{\mu,p }_{q }(\R^d)$,    $f\in   B^{\mu^{(-\frac{\eta}{8M})},p }_{q }(\R^d)$, and so  $\sum_{\la \in \La_j} \left(2^{-j\frac{\eta}{8M}}\frac{|c^f_\lambda|}{\mu(\lambda)}\right)^p =C< \infty$.  Thus, 
 $$ C   \geq \sum_{ \la\in \La_j}  \left(2^{-j\frac{\eta}{8M }}\frac{|c^f_\lambda|}{\mu(\lambda)}\right)^p  \mathbf{1}_{ \frac{|c^f_\lambda|}{\mu(\lambda )} \ge  2^{(j-J)\frac{\alpha_{\min}}{2}}2^{-J \big (\frac{\tau_\mu^*(\alpha)}{p}-\frac{\eta}{4}  \big)}  }.$$
 The number of cubes $\la \in \La_j$ such that the above indicator function is 1 is by   \eqref{DepJj} larger than the cardinality of $\mathcal{D}_{\eta ,J,j}(\alpha) $. It follows that 
$$
C \geq  \# \mathcal{D}_{\eta ,J,j} (\alpha)   2^{-j p\frac{\eta}{8M}}   \Big ( 2^{(j-J)\frac{\alpha_{\min}}{2}}2^{-J \big (\frac{\tau_\mu^*(\alpha)}{p}-\frac{\eta}{4}  \big)} \Big)^p  .$$
Noting that $j\leq MJ$ implies   $2^{j p\frac{\eta}{8M}} \leq 2^{Jp\frac{\eta}{8}}$, the last inequality  yields the result.
\end{proof}

 In particular, $\mathcal{D}_{\eta,J,j}=\emptyset$ for $j\ge J(p\frac{\alpha_{\min}}{2}+\tau_\mu^*(\alpha_{\min}))$.

Note that 
$$E_{\mu,\eta ,J_0}(\alpha)\cap  E_f^{\leq}(H - \eta ) \subset \bigcap_{J\ge J_0}  \ \bigcup_{j\ge J} \  \bigcup_{\lambda\in \mathcal{D}_{\eta,J,j}(\alpha)} \lambda.
$$
For  any $\delta>0$, denote by $\mathscr{H}^s_{\delta} $ the pre-$s$-Hausdorff measure on $\R^d$ associated with coverings by sets of diameter less than or equal to $\delta$. Using $\bigcup_{j\ge J}  \bigcup_{\lambda\in \mathcal{D}_{\eta,J,j}(\alpha)}  \lambda$ as covering of  $E_{\mu,\eta ,J_0}(\alpha)\cap  E_f^{\leq}(H - \eta )$, one deduces that  for every $J\geq J_0$, 
 \begin{align*}
 \mathscr{H}_{\sqrt{d}\cdot  2^{-J}} ^s\big (E_{\mu,\ep ,J_0}(\alpha)\cap  E_f^{\leq}(H - \eta )\big ) &\leq  \sum _{J\le j\le J(p\frac{\alpha_{\min}}{2}+\tau_\mu^*(\alpha_{\min}))} (\#\mathcal{D}_{\eta,J,j}) (\alpha) (\sqrt{d}\cdot  2^{-j})^s\\
 & \leq (\sqrt{d})^s C\left ( \sum _{m\ge 0} 2^{- mp \frac{\alpha_{\min}}{2} } \right ) 2^{J(\tau_\mu^*(\alpha)-p\frac{\eta}{8}-s )},
 \end{align*}
 which tends to zero as soon as $s> \tau_\mu^*(\alpha)-p\frac{\eta}{8} $. It follows that  
 $$\dim  \big( E_{\mu,\eta ,J_0}(\alpha)\cap  E_f^{\leq}(H-\eta) \big )  \le \tau_\mu^*(\alpha)-p\frac{\eta}{8},$$
 and thus  $\mu_\alpha(E_{\mu,\eta,J_0}(\alpha)\cap E_f^{\leq}(H-\eta))=0$, because $\mu_\alpha$ may give a positive mass to a set $E$ only if $\dim E\geq \tau_\mu^*(\alpha)$.  
 
To conclude, observe that   the almost doubling property of~$\mu$ yields 
$$E_\mu(\alpha)=\bigcap_{m\ge 1}\bigcup_{J_0\in\N}E_{\mu,\frac{1}{m},J_0}(\alpha).$$
 This equality combined with the previous estimate on $\mu_\alpha$ gives 
 $\mu_\alpha (E_\mu(\alpha)\cap E_f^{\leq}(H-\eta))=0$. 
 \end{proof}

  We are now equipped to prove the lower bound $\dim E_f (H )\geq \tau_\mu^*(\alpha)$.
  
  First,     \eqref{yeye} states that   $\frac{\log c^{{\mu,p,q}}_{\lambda_{j_n}(x) } }{\log 2^{-j_n} } =\frac{\overline{j_n(x)}}{j_n(x)} \theta_p(  \alpha_{\overline{\lambda_{j_n}(x)} }  )+\Big (1-\frac{\overline{j_n(x)}}{j_n(x)}\Big )\alpha_{\min} +r_n(x)  $.

 By Proposition \ref{propmualpha},  for $\mu_\alpha$-almost every $x$, $\lim_{j\to + \infty}\alpha_{\la_{j}(x)} = \alpha$. 
 
By Lemma  \ref {lemdiop}, for $\mu_\alpha$-almost every $x$, $\lim_{n \to + \infty}\frac{\overline{j_n(x)}}{j_n(x)} =1$. 
 
One deduces that  $h_f(x)\le  \theta_p(\alpha)= H$ for $\mu_\alpha$-almost every $x$, i.e. $ \mu_\alpha  (E_f^{\leq}(H )  )=1$ (the equality $h_f(x) =H$ does not hold in general, since \eqref{yeye} is true only for a subsequence of integers $(j_n)_{n\geq 1}$).
   
Combining all  the above results, one concludes that
\begin{align*}
 \mu_\alpha( E_f (H )) & = \mu_\alpha(E_\mu(\alpha)\cap E_f (H ))  \\
 & \geq   \mu_\alpha (E_\mu(\alpha)\cap E_f^{\leq}(H )) - \sum_{m\geq 1}   \mu_\alpha (E_\mu(\alpha)\cap E_f^{\leq}(H-1/m) ) =1.
 \end{align*}
This proves that  necessarily $\dim E_f (H )\geq \tau_\mu^*(\alpha)$, as expected.

\medskip\sk

\noindent {\bf Case 2:}  $H\in   [\alpha_{\min}, \theta_p(\alpha_{\min}))$: this corresponds to the affine part of the spectrum, which occurs  only when $\si _\mu(\alpha_{\min}) = \tau_\mu^*(\alpha_{\min})>0$, see Figure \ref{fig_functions2}.

If $H\in [\alpha_{\min}, \theta_p(\alpha_{\min}))$, write $H=\alpha_{\min}+ \frac{\tau_\mu^*(\alpha_{\min})}{\delta p}$, where $\delta>1$.  Observe that   Proposition~\ref{ubiquity} is established when $\mu\in\mathcal M_d$ but immediately extends to the case where $\mu \in \mathcal{E}_d$, i.e. $\mu$ is a positive power of an element of $\mathcal M_d$. So by Proposition~\ref{ubiquity} applied to the sequence   $(j_n)_{n\in\N}$ given by  Proposition \ref{propspecG}, the set $S({\delta},(\eta_j)_{j\in\N^*} ,(j_n)_{n\in\N})$ supports a Borel probability measure $\nu$ of lower Hausdorff dimension at least equal to  $\tau_\mu^*(\alpha_{\min})/\delta= p(H-\alpha_{\min})=\zeta_{\mu,p}^*(H)$. Note that    $(\eta_j)_{j\in\N*}$ depends only on $\mu$.

For $x\in S({\delta},(\eta_j)_{j\in\N^*} ,(j_n)_{n\in\N})$, one checks that 
$$h_f(x)\le \liminf_{n\to+\infty} \frac{\log c^{{\mu,p,q}}_{j_n}(x)}{\log 2^{-j_n}}\le \alpha_{\min}+ \frac{\tau_\mu^*(\alpha_{\min})}{\delta p}=H.$$   

In addition, $\{y\in [0,1]^d: h_f(y)<H\} = \bigcup_{m\geq 1} E_f^{\leq }(H-1/m)$,  and each set  $E_f^{\leq }(H-1/m)$  has a $\nu$-measure equal to 0, since due to Proposition~\ref{fm}(2) applied to the capacity provided by the leaders of $f$, $\dim E_f^{\leq }(H-1/m)\le (\zeta_f^\Psi)^*(H-1/m) < \zeta_{\mu,p}^*(H)$. Consequently, $\nu(E_f(H))=1 $ and $\dim E_f(H)\ge \zeta_{\mu,p}^*(H)$. 

\medskip

Finally, if $H=\alpha_{\min}$, the set $F=\bigcap_{p\in\N} S(p,(\eta_j)_{j\geq 1} ,(j_n)_{n\in\N})$ is easily seen to be non empty (by taking $\delta=p$ at step $p$ of the construction in the proof of proposition~\ref{ubiquity}) and to be included in  $E^{\le }_f(\alpha_{\min})$, by using the previous estimates. However we know that $E^{\le }_f(h)=\emptyset$ for all $h<\alpha_{\min}$ by Theorem~\ref{main}. Consequently, $E^{\le }_f(\alpha_{\min})=E^{\le }_f(\alpha_{\min})\neq\emptyset $, so $\sigma_f(\alpha_{\min}) = \dim E_f(\alpha_{\min})\ge 0$.
 


\begin{figure}
 \begin{tikzpicture}[xscale=2.2,yscale=2.2]
    {\tiny

\draw [dashed, domain=0:5]  plot ({-(exp(\x*ln(1/4))*ln(0.25)+exp(\x*ln(1/4))*ln(0.25)+exp(\x*ln(0.5))*ln(0.5))/(ln(3)*(exp(\x*ln(1/4))+exp(\x*ln(1/4)) +exp(\x*ln(0.5)) ) )} , {-\x*( exp(\x*ln(1/4))*ln(0.25)+ exp(\x*ln(1/4))*ln(0.25)+ exp(\x*ln(0.5))*ln(0.5))/(ln(3)*( exp(\x*ln(1/4)) +  exp(\x*ln(1/4)) +exp(\x*ln(0.5))))+ ln((exp(\x*ln(1/4))+exp(\x*ln(1/4))+ exp(\x*ln(0.5))))/ln(3)});
\draw [dashed, domain=0:5]  plot ({-( ln(0.2)+ ln(0.8))/(ln(2))  -0.27 +(exp(\x*ln(1/5))*ln(0.2)+exp(\x*ln(0.8))*ln(0.8))/(ln(2)*(exp(\x*ln(1/5))+exp(\x*ln(0.8)) ) )} , {-\x*( exp(\x*ln(1/5))*ln(0.2)+exp(\x*ln(0.8))*ln(0.8))/(ln(2)*(exp(\x*ln(1/5))+exp(\x*ln(0.8))))+ ln((exp(\x*ln(1/5))+exp(\x*ln(0.8))))/ln(2)});

\draw [->] (0,-0.2) -- (0,1.25) [radius=0.006] node [above] {$\si_f(H) $};
\draw [->] (-0.2,0) -- (3,0) node [right] {$H$};

\draw [thick, color=black, domain=0:5]  plot ({-(exp(\x*ln(1/4))*ln(0.25)+exp(\x*ln(1/4))*ln(0.25)+exp(\x*ln(0.5))*ln(0.5))/(ln(3)*(exp(\x*ln(1/4))+exp(\x*ln(1/4))+exp(\x*ln(0.5)) )) -\x*( exp(\x*ln(1/4))*ln(0.25)+ exp(\x*ln(1/4))*ln(0.25)+exp(\x*ln(0.5))*ln(0.5))/(ln(3)*(exp(\x*ln(1/4))+exp(\x*ln(1/4))+exp(\x*ln(0.5))))+ ln((exp(\x*ln(1/4))+exp(\x*ln(1/4))+exp(\x*ln(0.5))))/ln(3)} , {-\x*( exp(\x*ln(1/4))*ln(0.25)+exp(\x*ln(1/4))*ln(0.25)+exp(\x*ln(0.5))*ln(0.5))/(ln(3)*(exp(\x*ln(1/4))+exp(\x*ln(1/4))+exp(\x*ln(0.5))))+ ln((exp(\x*ln(1/4))+exp(\x*ln(1/4))+exp(\x*ln(0.5))))/ln(3)});

\draw [thick, color=black, domain=0:1]  plot ({-( ln(0.2)+ ln(0.8))/(ln(2)) - 0.27+(exp(\x*ln(1/5))*ln(0.2)+exp(\x*ln(0.8))*ln(0.8))/(ln(2)*(exp(\x*ln(1/5))+exp(\x*ln(0.8)) ) )-\x*( exp(\x*ln(1/5))*ln(0.2)+exp(\x*ln(0.8))*ln(0.8))/(ln(2)*(exp(\x*ln(1/5))+exp(\x*ln(0.8))))+ ln((exp(\x*ln(1/5))+exp(\x*ln(0.8))))/ln(2)} , {-\x*( exp(\x*ln(1/5))*ln(0.2)+exp(\x*ln(0.8))*ln(0.8))/(ln(2)*(exp(\x*ln(1/5))+exp(\x*ln(0.8))))+ ln((exp(\x*ln(1/5))+exp(\x*ln(0.8))))/ln(2)});

\draw [fill] (0.93,0.26) circle [radius=0.02] ;

\draw [fill, color= red]  (0.93,0.26) --  (0.64,0) circle [radius=0.02, color=black] node [below]  {$\tau_\mu'(+\infty)=\zeta_{\mu,1}'(+\infty) $};

\draw [dashed, color= black]    (0.67,0.26)  [fill,radius=0.02, color=black] circle;

\draw  [fill] (0,0.26)  circle [radius=0.03]   node [left]    {$\si_\mu(\alpha_{\min})>0$}  [dashed] (0,0.26) -- (0.96,0.26)  ;
 
  \draw[dashed] (0,1) -- (2.6,1);
\draw [fill] (-0.1,-0.10)   node [left] {$0$}; 
\draw [fill] (0,1) circle [radius=0.03] ; 
\draw  [fill] (2.32-0.27,0) circle [radius=0.03] (2.32-0.1-0.27,0) node [below] {$\tau'_\mu(-\infty)$};
\draw  [fill] (2.65-0.27 ,0.75) circle [radius=0.03]  [dashed]   (2.65-0.27,0.75) -- (2.65-0.275,0)  [fill] (2.65-0.27 ,0) circle [radius=0.03]  (2.65  ,-0.01) node [below] {$\zeta_{\mu,1}'(-\infty)$};
}
\end{tikzpicture}
\caption{Case where $\si_\mu(\alpha_{\min})>0$ and $p=1$: the dashed graph represents the spectrum of $\mu$, the plain graph represents the multifractal spectrum $\si_f$ of typical functions  $f\in \widetilde B^{\mu,1}_q(\R^d)$. An affine segment (in red) with slope $p=1$ appears in the spectrum $\si_f$.  }
\label{fig_functions2}
\end{figure}
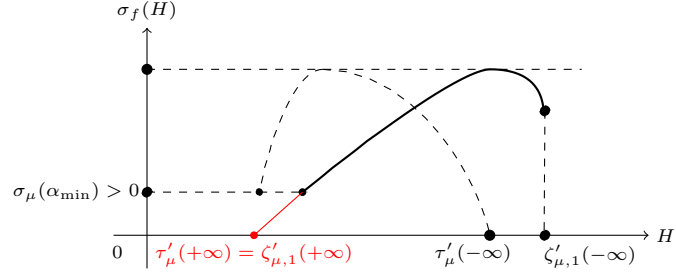


\subsection{Typical multifractal behavior in $\widetilde B^{\mu,p}_q(\R^d)$}
\label{sec_typi} 

We finally prove item (2) of  Theorem~\ref{main}, hence obtaining the \ml behavior of typical functions in $\widetilde B^{\mu,p}_q(\R^d)$. 
 
\medskip

Recall the definition~\eqref{Nm0} of the basis $\{\mathcal N_m\}_{ m\in\N }$ of neighborhoods of the origin in $\widetilde B^{\mu,p}_q(\R^d)$. 

For every integer $m>m_0=\lfloor \max (1, s_1^{-1})\rfloor +1$, set 

$$
V_{m}=\left\{f\in \widetilde B^{\mu,p}_q(\R^d): \forall j\ge J_2, \, \forall \lambda\in \Lambda_j, \, \frac{|c^f_\lambda|}{c^{{\mu,p,q}}_\lambda}\in 
m^{-1}\{1,\ldots,m^2\}\right \}.
$$
Then let 
\begin{equation}\label{Gdelta}
\mathcal{G}=\limsup_{m\to\infty} (V_m+ \mathcal{N}_{2^{\lceil m\log(m)\rceil}}).
\end{equation}
Each $\bigcup_{\ell\ge m}V_\ell$, $m\ge m_0$, is dense in $\widetilde B^{\mu,p}_q(\R^d)$, so $\mathcal{G}$ contains  a dense $G_\delta$ set.

When $f\in \mathcal{G}$, there exists an increasing sequence $(j_n)_{n\ge 0}$ such that  $f\in V_{j_n}+\mathcal{N}_{2^{\lceil j_n \log(j_n)\rceil}}$ for all $n\ge 0$.

Fix $L\in \mathbb{Z}^d$. Looking at the particular generation $j_n$, for all $\lambda\in \Lambda_{j_n}$ such that $\lambda\subset L+3[0,1]^d $, by definition of  $V_{j_n}$ and $\widetilde{\mathcal N}_{2^{\lceil j_n\log(j_n)\rceil}}$, the lower bound  $|c^f_{\lambda}|\ge j_n^{-1}c^{{\mu,p,q}}_{\lambda}-2^{-\lceil j_n\log(j_n)\rceil}\mu(\lambda)2^{j_n2^{-j_n\log(j_n)}} $ holds. By construction of the coefficients $c^{{\mu,p,q}}_{\lambda}$, this implies that for $n$ large enough one has $|c^f_{\lambda}|\ge j_n^{-1}c^{{\mu,p,q}}_{\lambda}/2$,  hence there exists a positive sequence $(\varepsilon_{n})_{n\in\N }$  converging to 0 such that $|c^f_{\lambda}|\ge 2^{-j_n\varepsilon_n}|c^{{\mu,p,q}}_{\lambda}|$ for all $\lambda\in\Lambda_{j_n}$ such that $\lambda\subset L+3[0,1]^d $. Consequently,   Proposition~\ref{propspecG} yields ${\sigma}_f={\sigma}_{g^{\mu,p,q}} =\zeta_{\mu,p}^*$.  

 \begin{remark}
 \label{rk-typical}
In fact, the definitions of  $V_{j_n}$, $\widetilde{\mathcal N}_{2^{\lceil j_n\log(j_n)\rceil}}$, and  $c^{{\mu,p,q}}_{\lambda}$, imply  that if $(j_n)_{n\ge 1}$ is an increasing sequence of integers and $f\in \bigcap_{n\ge 1}V_{j_n}+  \mathcal{N}_{2^{\lceil j_n \log(j_n)\rceil}}$, then for all $N,K\in\N^*$, for all $n\ge 1$ large enough and $\lambda\in \bigcup_{j=j_n}^{K j_n} \Lambda_{j }$ such that $\lambda\subset N[0,1]^d$,  one has
\begin{align*}
 \frac{1}{2 j_n}c^{\mu,p,q}_{\lambda} \le  |c^f_{\lambda}|  \le 2 j_n c^{{\mu,p,q}}_{\lambda} .
 \end{align*} 
 These bounds will be useful to estime the $L^q$-spectrum of $f$.  
 \end{remark}

\section{Validity of the WMF and the WWMF  in $\widetilde B^{\mu,p}_q(\R^d)$}
\label{sec-formalism}

 Recall that the    multifractal formalisms for functions were defined in Section~\ref{secMFF}. In this last section, we first discuss the validity of the WMF  for the saturation function ${g^{\mu,p,q}} $. This helps in establishing part (3) of Theorem~\ref{validity} in Section~\ref{sec-formalism-typical3}, while Section~\ref{sec-formalism-typical2} provides the proof of part (2) of Theorem~\ref{validity}.

\subsection{WMF and WWMF for the saturation function ${g^{\mu,p,q}} $}
\label{sec-formalism-g}

Recall that the wavelet $\Psi$ is fixed, and that ${g^{\mu,p,q}} $ is built via its wavelet coefficients in the wavelet basis generated by $\Psi$. Also, recall  \eqref{zetaNPsi} for the definition of $\zeta^{N,\Psi} _{g^{\mu,p,q},j}$, and the various notations concerning $L^q$-spectra for functions.
\begin{proposition}\label{MFgmupq}
The  WMF  holds for   ${g^{\mu,p,q}} $   on the interval  $[\zeta_{\mu,p}'(+\infty),\zeta_{\mu,p}'(0^+)]$, and the WWMF  holds for   ${g^{\mu,p,q}} $   on the interval  $[\zeta_{\mu,p}'(+\infty),\zeta_{\mu,p}'(-\infty)]$. 

Moreover, for all $N\in \N^*$, one has  $\lim_{j\to+ \infty}\zeta^{N,\Psi} _{g^{\mu,p,q},j}={\zeta_{\mu,p}}$.
\end{proposition}
The second part of the statement shows that   the convergence  of the sequence $\big (\zeta^{N,\Psi} _{g^{\mu,p,q},j}\big )_{j\ge 1}$ is stronger than what is required for the WWMF  to hold (only the convergence over a subsequence is needed). 
\begin{proof} Suppose that it is established that for all $N\in \N^*$, one has  $\lim_{j\to +\infty}\zeta^{N,\Psi} _{g^{\mu,p,q},j}={\zeta_{\mu,p} }$. In particular $\zeta^{N,\Psi} _{g^{\mu,p,q}}={\zeta_{\mu,p} }$ for all $N\in\mathbb N^*$, so $\zeta^{\Psi} _{g^{\mu,p,q}}={\zeta_{\mu,p} }$. Since it was shown in the previous section that $\sigma_{g^{\mu,p,q}}  = \zeta_{\mu,p}^*$, one concludes  that $g^{\mu,p,q}$ satisfies the WMF.  

\sk 

Now, fix $N\in \N^*$. Let us prove that $\lim_{j\to+\infty}\zeta^{N,\Psi} _{g^{\mu,p,q},j}={\zeta_{\mu,p} }$. 

The $\Z^d$-invariance of $\mu$ and the definition of $g^{\mu,p,q}$ show that if is enough to work on $\zu^d$ and to prove that $\lim_{j\to+ \infty}j^{-1} \log \sum_{\lambda\in \mathcal D^0_j}(L^{g^{\mu,p,q}}_\lambda)^t=\zeta_{\mu,p}(t)$. 

Fix $t\in\R$. Recall Remark~\ref{remg2}(1) and Lemma~\ref{lem-maj2}. The reader  can check that due to these two facts,
$$
\lim_{j\to\infty}j^{-1} \log \frac{\sum_{\lambda\in \mathcal D^0_j}(L^{g^{\mu,p,q}}_\lambda)^t}{\sum_{\lambda\in \mathcal D^0_j}(c^{g^{\mu,p,q}}_\lambda)^t}=0.
$$
Moreover, by definition of the coefficients $c^{g^{\mu,p,q}}_\lambda$, and since  $\log (w_\la)=o(\log(\mu(\lambda)))$ uniformly in $\lambda\in\Lambda_j$ as $j\to+\infty$,  $$
\lim_{j\to\infty}j^{-1} \log \frac{\sum_{\lambda\in \mathcal D^0_j}(c^{g^{\mu,p,q}}_\lambda)^t}{B (j,t)  }=0,\text{ where }
B (j,t) =    \sum_{\lambda\in \mathcal D^0_j} \left( \mu(\lambda) 2^{\ds -\overline{j} \frac{{\tau_\mu^*}(\alpha_{ \lai })}{p}} \right)^t.
$$  
Thus, one must prove that 
\begin{equation}
\label{limbjt}
\lim_{j\to+\infty} j^{-1} \log_2  B(j,t) =\zeta_{\mu,p}(t).
\end{equation}
 When $p=+ \infty$, this was established when $\mu$ is an element of $\mathcal M_d$ in Section~\ref{sectaumu}, but in the general case where $\mu$ is a positive power of such a measure the result holds as well by a direct calculation. 

\sk 

Assume now that $p<+\infty$. Fix $t\in\R^*$, the case $t=0$ being obvious.  

\medskip

Fix $\ep>0$. Using  the same decomposition as that used in the proof of Lemma \ref{gdansB},  
$$
B(j,t)  = \sum_{J=0}^j \sum_{  \lambda\in \mathcal D^0_J\setminus  ( \mathcal D^0_{J-1}\cdot [0,2^{-1}]^d)} 
 \mu( \la\cdot[0,2^{-(j-J)}]^d)^t  2^{-  \frac{t}{p} J {\tau_\mu^*}(\alpha_{  \la})}.
 $$
 Then, from \eqref{lala} we deduce that there exists a positive sequence $(C_j)_{j\ge 1}$ depending on $t$ and $\mu$ such that $\lim_{j\to+ \infty}\frac{\log(C_j)}{j}=0$ and for all $j\ge 1$,
 $$ 2^{(j-J)(\alm+\ep) }C_j ^{-1} \leq \frac{\mu(\la)}{\mu( \la\cdot[0,2^{-(j-J)}]^d)} \leq C_j2^{(j-J)(\alm-\ep)}.$$
 
Observe that when $\la$ and $\la'$ are neighbors in $\La_J$, the two numbers $ \mu( \la)^t   2^{-  \frac{t}{p} J {\tau_\mu^*}(\alpha_{  \la})}$ and $ \mu( \la')^t   2^{-  \frac{t}{p} J {\tau_\mu^*}(\alpha_{  \la'})}$ differ by a factor at most $2^{J\ep}$. This follows from the almost doubling property (P$_2$) of $\mu$ and the continuity of $\tau_\mu^*$.

These considerations prove that  there exists another positive sequence $(\widetilde C_j)_{j\ge 1}$ depending on $t$ and $\mu$ such that $\lim_{j\to+ \infty}\frac{\log(\widetilde C_j)}{j}=0$ and 
\begin{align}
\label{interm33}
\widetilde C_j^{-1}\widetilde B(j,t,\alpha_{\min}, s(t)\ep )\le B(j,t) \le  \widetilde C_j \widetilde B(j,t,\alpha_{\min}, -   s(t)\ep ),
\end{align} 
where   $s(t)$ is the sign of $t$ and 
 \begin{equation}\label{tildeBjbeta}
\widetilde B(j,t,\beta,\gamma )=\sum_{J=0}^j 2^{-(j-J)t(\beta+\gamma)}2^{-J\gamma} \sum_{  \lambda\in \mathcal D^0_J} 
 \mu( \la)^t   2^{-  \frac{t}{p} J {\tau_\mu^*}(\alpha_{  \la})},
\end{equation}

The quantity  $\sum_{  \lambda\in \mathcal D^0_J} 
 \mu( \la)^t   2^{-  \frac{t}{p} J {\tau_\mu^*}(\alpha_{  \la})}$ is now controlled. Using Proposition \ref{fm}(4),  the interval $[\alm , \alpha_{\max}]$ can be split  into $M$ contiguous intervals $I_i=[\alpha_i,\alpha_{i+1}]$, $i=1, ...M$,  of length less than $\ep$ such that for every $i\in \{1,...,M\}$, 
 $$ \Big |\sup_{\alpha \in I_i}\tau^*_\mu(\alpha)-  \frac{\log_2 \#\mathcal{D}_\mu(j,I_i)}{j}\Big |\leq \ep \ \  \mbox{ and } \ \  \sup_{\alpha,\alpha' \in I_i}| \tau^*_\mu(\alpha)-\tau^*_\mu(\alpha') \big| \leq \ep  .$$
Define the mapping $\chi_3:\alpha\in[\alpha_{\min},\alpha_{\max}]\mapsto t\theta_p(\alpha)-\tau^*_\mu(\alpha)$ (in \eqref{defchit}   its restriction $\widetilde \chi_2$ to the interval $[\alpha_{\min},\alpha_p]$ was considered). Without loss of generality,   suppose that  there exists $1\le i_0 \le M$ such that $t\theta_p(\alpha_{i_0})-\tau^*_\mu(\alpha_{i_0} )= \min \{ \chi_3 (\alpha)   :\alpha\in[\alpha_{\min},\alpha_{\max}]\}:=\zeta_3(t)$. 

Also, by Remark~\ref{alphalambda}, there exists $C\ge 1$ such that for all $j\in\N$ and $\lambda\in\mathcal D_{j}^0$, one has $C^{-1}2^{-j(\alpha_\lambda+\ep)}\le \mu(\lambda)\le C  2^{-j(\alpha_\lambda-\ep)}$.

If follows from the previous information that
$$
\sum_{  \lambda\in \mathcal D^0_J} 
 \mu( \la)^t   2^{-  \frac{t}{p} J {\tau_\mu^*}(\alpha_{  \la})}
\begin{cases}
    \le \displaystyle C^{|t|} \sum_{i=1}^M 2^{J (\tau^*_\mu(\alpha_i )+\ep) } 2^{-Jt(\alpha_i -2s(t)\ep)}  2^{-  \frac{t}{p} J {(\tau_\mu^*}(\alpha_{ i}) -s(t)\ep)} \\
 \ge \displaystyle C^{-|t|} \sum_{i=1}^M 2^{J (\tau^*_\mu(\alpha_i )-\ep) } 2^{-Jt(\alpha_i +2s(t)\ep)}  2^{-  \frac{t}{p} J {(\tau_\mu^*}(\alpha_{ i}) +s(t)\ep)}, 
\end{cases}
$$ 
which implies that 
\begin{equation}\label{interm44}
  2^{J  s(t)\ep} \sum_{  \lambda\in \mathcal D^0_J}   \mu( \lambda)^t   2^{-  \frac{t}{p} J {\tau_\mu^*}(\alpha_{ \lambda})} 
  =m_J(t,\ep)\sum_{i=1}^M 2^{-J \chi_3(\alpha_i ) }
\end{equation}
where $|\log (m_J(t,\ep))|\le |t|\log (C)+ (2+2|t|+\frac{|t|}{p})J\ep$. 

Then, incorporating \eqref{interm44} in \eqref{tildeBjbeta}  and using that the infimum of $\chi_3(\alpha_i)$ is reached at $i_0$, i.e. $\chi_3(\alpha_{i_0})=\zeta_3(t)$, one gets
\begin{equation}\label{tildeBjbetabis}
\widetilde B(j,t,\beta, \pm  \ep )=\sum_{J=0}^j  2^{-(j-J)t(\beta+\gamma)}\widetilde m_J(t,\ep)2^{-J \zeta_3(t)},
\end{equation}
where $|\log (\widetilde m_J(t,\ep))|\le \log (M)+|t|\log (C)+ (2+2|t|+\frac{|t|}{p})J\ep$. Incorporating \eqref{tildeBjbetabis} in  \eqref{interm33} then implies 
\begin{equation}\label{Bjfin}
B_j=\widehat m_j(t,\ep)2^{-jt\alpha_{\min}}\sum_{J=0}^j \overline m_J (t,\ep)\widetilde m_J(t,\ep)2^{-J(\zeta_3(t)-t\alpha_{\min})},
\end{equation}
where $\max (|\log (\overline  m_j(t,\ep)), |\log (\widehat m_j(t,\ep))|)\le j|t|\ep+\log (\widetilde C_j)$.  

It follows from \eqref{Bjfin} and the fact that $\ep$ is arbitrary, that:\begin{itemize}
\item
 $\zeta_3 (t)-t\alpha_{\min}\ge 0$ implies $\lim_{j\to+\infty}\frac{\log_2 B(j,t)  }{-j}=t\alpha_{\min}$, 
 \item
  $\zeta_3(t)-t\alpha_{\min}\le 0$ implies $\lim_{j\to+\infty}\frac{\log_2 B(j,t) }{-j}=\zeta_3(t)$. 
  \end{itemize}
  hence, to   prove \eqref{limbjt} and   Proposition \ref{MFgmupq},     the value of  $\zeta_3(t)$ and   the sign of $\zeta_3(t)-t\alpha_{\min}$ must be investigated. According to the previous observations, this will give the desired conclusion. 

\sk 

The two cases  $\alpha_{\min} = \alpha_{\max}$ and $\alpha_{\min} < \alpha_{\max}$ are split.

\sk 

Suppose first that $\alpha_{\min} = \alpha_{\max}$. Then, $\tau_\mu(t)=\alpha_{\min}t-d$ for all $t\in\R$, and 
$$\zeta_{\mu,p}(t)= \begin{cases} (\alpha_{\min}+\frac{d}{p})t-d \ \  \mbox { when } t<p ,\\   \alpha_{\min}t  \ \ \mbox { when } t\ge p.\end{cases}$$

A straightforward computation gives   $\zeta_3(t)= t\alpha_{\min}+\Big (\frac{t}{p}-1\Big)d$.  Thus  when $t<p$, $\zeta_3 (t)=\zeta_{\mu,p}(t)$. Moreover,  $\zeta_3 (t)-t\alpha_{\min}=\Big (\frac{t}{p}-1\Big)\tau_\mu^*(\alpha_{\min})$  is non negative if and only if $t\ge p$, i.e. $\zeta_{\mu,p}(t)=\alpha_{\min}t$; and  when $p> t$ one has $\zeta_{\mu,p}(t)=\zeta_3(t)$, hence the result.

\medskip

Assume next that $[\alpha_{\min},\alpha_{\max}]$ is non trivial. 

When  $t\ge p$, the mapping $\chi_3$ rewrites  $\chi_3(\alpha)=t\alpha+\Big (\frac{t}{p}-1\Big)\tau_\mu^*(\alpha)$ so it is concave, and it reaches  its minimum $\zeta_3 (t)$  either at $\alpha_{\min}$ or  at $\alpha_{\max}$. In either case, $\zeta_3  (t)-t\alpha_{\min}\ge 0$. Moreover, in this range  $\zeta_{\mu,p}(t)=t\alpha_{\min}$, so \eqref{limbjt} holds true.

\medskip

When $t<p$, recall the notations introduced and the fact established in the proof of Proposition~\ref{lemtetap}.

If  $t_p=\frac{pt}{p-t}\le t_\infty=(\tau_\mu^*)'(\alpha_{\min}^+)$,   the convex function $\chi_3 $ reaches its minimum  $\frac{p-t}{p}\tau_\mu(\frac{p}{p-t} t)=\zeta_{\mu,p}(t)$ at $\widetilde \alpha_t$, i.e. $\zeta_3 (t)=\zeta_{\mu,p}(t)$.

If  $t_p>t_\infty$, then   $ \chi_3 $ is increasing and reaches  at $\alpha_{\min}$ its minimum  equal to $ t\alpha_{\min}+\Big (\frac{t}{p}-1\Big)\tau_\mu^*(\alpha_{\min})=\zeta_{\mu,p}(t)$ (here   $\zeta_3(t)=\zeta_{\mu,p}(t)$ as well). In both cases,    $\zeta_3(t)-t\alpha_{\min}\le  \zeta_3(t)-\chi_3(\alpha_{\min})\le 0$ and  \eqref{limbjt} holds true.
 \end{proof}

\subsection{Proof of Theorem~\ref{validity}(2)}
\label{sec-formalism-typical2}

As recalled in the introduction,  it is known  \cite{JaffardPSPUM}  that for any smooth function $f$, one has $
{\sigma}_f\le  \zeta_f^*$. Since it was  shown in Section \ref{sec_typi}  that ${\sigma}_f = (\zeta _{\mu,p})^*$  for typical functions in $\widetilde B^{\mu,p}_q(\R^d) $,   for such   functions one necessarily has $\zeta_f \le  \zeta_{\mu,p}$ by inverse Legendre transform. Simultaneously,   Theorem~\ref{LBzeta} states   that ${\zeta_f}_{|\R_+}={\zeta_f^\Psi}_{|\R_+}\ge  {\zeta_{\mu,p}}_{|\R_+}$, which yields the desired result.

\subsection{Proof of Theorem~\ref{validity}(3)}
\label{sec-formalism-typical3}
It is enough to get part (i). Then part (ii) follows from the fact that the class of residual sets is stable by countable intersection. 

Let $f\in \mathcal{G}$, where $\mathcal{G}$ is the $G_\delta$ set defined by \eqref{Gdelta}, and consider a sequence $(j_n)_{n\ge 1}$ such that $f\in V_{j_n}+\mathcal{N}_{2^{\lceil m\log(m)\rceil}} $ for all $n\ge 1$. Fix $N\in\N^*$. We prove that $\zeta^{\Psi,N}_{f,j_n} $  converges pointwise to $\zeta_{\mu,p}$ as $n\to + \infty$,  which is enough to show that the WWMF  holds  relatively to $\Psi$ over $[\zeta_{\mu,p}'(+\infty), \zeta_{\mu,p}'(-\infty)]$, since it was established that  $\sigma_f=\zeta_{\mu,p}^*$. 

Since a  function  $f\in \mathcal{G}$  necessarily belongs to ${\mathscr C}^{\alpha_{\min} -\ep}(\R^d) $ (for every $\ep>0$), one has $|c_\la^f | \leq 2^{-j(\alpha_{\min}-\ep)}$ for every large $j$ and $\la\in \La_j$ such that $\lambda\subset (N+1)[0,1]^d$.  

Fix $\ep=\alpha_{\min}/2$. By construction, when $j$ is large and $\la\in \La_j$, $c^{{\mu,p,q}}_{\lambda}\ge 2^{-2j\alpha_{\max}}$. Hence,    from the previous fact and  Remark \ref{rk-typical} applied with $K=\lfloor 4\alpha_{\max}/\alpha_{\min}\rfloor +1$, one sees that  when $n$ becomes large,  for all $j\ge j_n$ and $\la\in \La_j$  such that $\lambda\subset (N+1)[0,1]^d$:
\begin{itemize}
\item
 either $j\in \{j_n,\ldots,Kj_n\}$ and the wavelet coefficient $c_\la^f$ of $f$ satisfies $ \frac{1}{2j_n} c^{{\mu,p,q}}_{\lambda} \leq 
  |c^f_{\lambda}|   \leq 2 j_n c^{{\mu,p,q}}_{\lambda}$, 
  ìtem
  or $j> Kj_n$ and $|c^f_{\lambda}|\le c^{{\mu,p,q}}_{\lambda}$.  This implies that for all $\lambda\in \mathcal D_{j_n}$ such that $\lambda\subset N[0,1]^d$, the wavelets leader $L^f_{\lambda}  $ of $f$  satisfies 
  $$\frac{1}{2j_n}\,  L^{g^{\mu,p,q}}_{\lambda}  \leq 
  L^f_{\lambda}   \leq 2j_n\, L^{g^{\mu,p,q}} _{\lambda}.$$
  \end{itemize}
Consequently,  $\lim_{n\to+ \infty}j_n^{-1}\log_2 \Big (\frac{\zeta^{\Psi,N}_{f,j_n}}{ \zeta^{\Psi,N}_{g,j_n}}\Big )=0$, and by Proposition~\ref{MFgmupq},  $\zeta^{\Psi,N}_{f,j_n}$ indeed converges to $\zeta_{\mu,p}$ as $n\to\infty$.  

\medskip

Finally, when $q<+\infty$, to establish that for a typical $f\in\widetilde B^{\mu,p}_q(\R^d)$ one has ${\zeta^{\Psi}_f}_{|\R^*_-}=-\infty$,   consider for all $m\in \N^* $ the set 
$$
\widetilde V_{m}=\left\{f\in \widetilde B^{\mu,p}_q(\R^d):\forall \ m\le j\le m\log (m), \, \forall \lambda\in \Lambda_j, \, c^f_\lambda=0\right \}.
$$
The set $\limsup_{m\to\infty}\widetilde V_{m}$ is dense in $\widetilde B^{\mu,p}_q(\R^d)$ and  \begin{equation*}
\widetilde {\mathcal{G}}=\mathcal{G}\cap \limsup_{m\to\infty} (\widetilde V_m+\mathcal{N}_{2^{\lceil m\log(m)\rceil}} ).
\end{equation*}
is a  dense $G_\delta$-set. 
When $f\in \widetilde{\mathcal{G}}$, there exists  an increasing sequence of integers $(m_n)_{n\in\N}$  such that $f\in \widetilde V_{m_n}+\mathcal \mathcal{N}_{2^{\lceil m_n\log(m_n)\rceil}} $ for all $n\in\N$. It is easily checked that for any $A>0$ and $N\in\N$, for $n$ large enough, if  $\lambda\in\mathcal D_{m_n}$ and $\lambda\subset  N\zud$,  one has $L^f_\lambda\le 2^{-A m_n}$. This implies that for $t<0$,
$$
\sum_{\lambda\in \mathcal D_{m_n},\  \lambda\subset N\zud }\mathbf{1}_{L^f_\lambda>0}(L^f_\lambda)^t\ge \#\{\lambda\in \mathcal D_{m_n},\  \lambda\subset N\zud : L^f_\lambda>0\}\cdot 2^{-A tm_n},
$$ 
hence $\zeta^{\Psi,N}_f(t)\le At$.  Consequently, $A$ being  arbitrary and $t<0$, the desired conclusion holds. 

\section{Proof of Theorem~\ref{exhaustion}}\label{proof of exhaustion}

Part (1) follows from the fact that for $\sigma\in\mathscr{S}_s$ to be the typical singularity spectrum in $\widetilde {B}^{\mu,p}_{q}(\R^d)$ with $p<+\infty$, by Theorem~\ref{main} and Proposition~\ref{lemtetap} it is necessary that $\sigma(H_{\min})=0$, and by Theorem~\ref{main} the function $\sigma^*$ is linear over $[p,+\infty]$ so $\sigma'(H_{\min}^+)\le p$ by Remark~\ref{Linearisation}. 

To prove part (2),   the cases  $p\not\in\overset{\, _\smallfrown}{\partial} \sigma((H_{\min},H_{\max}])$ and  $p\in\overset{\, _\smallfrown}{\partial}\sigma((H_{\min},H_{\max}])$ are separated.

\medskip

{\bf Case $p\not\in\overset{\, _\smallfrown}{\partial}\sigma ((H_{\min},H_{\max}])$:} Define the mapping
$$
A: H\in [H_{\min},H_{\max}]\mapsto H-\frac{\sigma(H)}{p}.
$$ 
It is a continuous increasing bijection onto its image,  that we denote by $I=[\alpha_{\min},\alpha_{\max}]$. For $\alpha\in I$, denote $A^{-1}(\alpha)$ by $H(\alpha)$. It is easily checked that the mapping 
$$
\widetilde\sigma : \alpha\in I\mapsto p(H(\alpha)-\alpha)$$
belongs to $\mathscr{S}_d$ as well, and that if  $\mu\in\mathscr{E}_d$ is chosen such that $\sigma_\mu=\widetilde\sigma$, the study achieved in Section~\ref{sec-description} implies that $\sigma$ is the singularity spectrum of the typical functions in $\widetilde {B}^{\mu,p}_{q}(\R^d)$, for all $q\in[1,+\infty]$ (the function $A$ is then the inverse of the function $\theta_p$ defined in \eqref{thetap}). 

Suppose, moreover, that $\sigma'(H_{\max}^-)=-\infty$ and  $\sigma(H_{\max})>0$. This is equivalent to suppose that  $\widetilde\sigma' (\alpha_{\max}^-)=-p$ and  $\widetilde\sigma(\alpha_{\max})>0$. Again, the study achieved in Section~\ref{sec-description} shows that for any element $\widehat\sigma$ of $\mathscr{S}_d$  whose domain takes the form $[\alpha_{\min},\alpha_{\max}']$ with $\alpha_{\max}'>\alpha_{\max}$ and $\widehat\sigma_{|[\alpha_{\min},\alpha_{\max}]}=\widetilde\sigma$, for any $\nu\in  \mathscr{E}_d$ such that $\sigma_\nu=\widehat\sigma$, $\sigma$ is still the singularity spectrum of the typical functions in $\widetilde {B}^{\nu,p}_{q}(\R^d)$, for all $q\in[1,+\infty]$. Note that there are infinitely many ways to consider such an extension. 
 
\medskip

{\bf Case $p\in\overset{\, _\smallfrown}{\partial}\sigma((H_{\min},H_{\max}])$:} this means that there is a non-trivial maximal subinterval $[H_{\min},\widetilde H_{\min}]$ of $[H_{\min},H_{\max}]$ such that for all $H\in [H_{\min},\widetilde H_{\min}]$ one has $\sigma(H)=p(H-H_{\min})$. 

If $\widetilde H_{\min}=H_{\max}$, one  chooses   $\widetilde\sigma=d\cdot \mathbf{1}_{H_{\max}}$, so that $\mu=(\mathcal {L}^d)^{\frac{H_{\max}}{d}}$ is such that $\sigma_\mu=\widetilde\sigma$ and $\sigma$ is the singularity spectrum of the typical functions in $\widetilde {B}^{\mu,p}_{q}(\R^d)$, for all $q\in[1,+\infty]$. 

If $\widetilde H_{\min}<H_{\max}$,   the same choices as in the case $p\not\in\partial ((H_{\min},H_{\max}])$ are made, except that  $ \sigma$  is replaced by its restriction to $[\widetilde H_{\min},H_{\max}]$, the difference here being now that $\sigma(\widetilde H_{\min})>0$. 

The claim about the validity of the WMF and WWMF  follows from Theorem~\ref{validity}.



\bibliographystyle{plain}
\bibliography{Biblio_BME}

\end{document}